\documentclass[11pt,a4paper,reqno]{amsart}
\usepackage[a4paper,top=1in, bottom=1in, left=1in, right=1in]{geometry}
\newtheorem{theorem}{Theorem}[section] 
\newtheorem{lemma}[theorem]{Lemma}

\newtheorem{proposition}[theorem]{Proposition} 
\newtheorem{remark}[theorem]{Remark} 
\usepackage{mathrsfs,amssymb,amsmath,amsthm,stmaryrd,graphicx,mathtools,xfrac}
\usepackage[amsstyle]{cases}
\usepackage{enumitem}
\usepackage[linkcolor=blue,citecolor=cyan,colorlinks]{hyperref}
\usepackage[utf8]{inputenc}
\usepackage[T1]{fontenc}
\usepackage{xcolor,ulem}   
\numberwithin{equation}{section}
\allowdisplaybreaks
\renewcommand{\ge}{\geqslant}
\renewcommand{\le}{\leqslant} 
\renewcommand{\geq}{\geqslant}
\renewcommand{\leq}{\leqslant}
\newcommand{\vare}{\varepsilon}

\newcommand{\laplace}{\Delta}
\newcommand{\surflaplace}{\Delta_B}
\newcommand{\dive}{\operatorname{div}}  
\newcommand{\tandive}{\dive_{\sigma}}
\newcommand{\curl}{\operatorname{curl} }
\newcommand{\trace}{\operatorname{Tr} }
\newcommand{\connec}{\hat{\nabla} }
\newcommand{\tangrad}{\bar{\nabla} }
\newcommand{\matder}{\mathcal{D}_t }
\newcommand{\energy}{e}
\newcommand{\meancurv}{\mathcal{A}}
\newcommand{\barenergy}{\bar{\energy}}
\newcommand{\barE}{\bar{E}}
\newcommand{\parOmega}{\Gamma}
\newcommand{\tinu}{\tilde{\nu}}
\newcommand{\Rdiv}{\uppercase\expandafter{\romannumeral1}}
\newcommand{\Rbulk}{\uppercase\expandafter{\romannumeral2}}
\newcommand{\radi}{\mathcal{R}}
\newcommand{\N}{\mathbb{N}}
\newcommand{\Nz}{\N_0}

\begin{document}
\title[Incompressible Ideal MHD with Surface Tension]{A priori estimates and a blow-up criterion for the incompressible ideal  MHD equations with surface tension and a closed free surface}
\author{Chengchun Hao \and Siqi Yang}
\address{Chengchun Hao and Siqi Yang\newline
State Key Laboratory of Mathematical Sciences, Academy of Mathematics and Systems Science, Chinese Academy of Sciences, Beijing 100190, China\newline
\and
School of Mathematical Sciences, University of Chinese Academy of Sciences, Beijing 100049, China
}
\email{hcc@amss.ac.cn} 
\email{yangsiqi@amss.ac.cn}
\begin{abstract} 
We establish the a priori estimates and prove a blow-up criterion for the three-dimensional free boundary incompressible ideal magnetohydrodynamics equations. The fluid occupies a bounded region with a free boundary that is a closed surface, without assumptions of simple connectedness or periodicity of the region (thus, Fourier transforms cannot be applied), nor the graph assumption for the free boundary. The fluid is under the influence of surface tension, and flattening the boundaries using local coordinates is insufficient to resolve this problem. This is because local coordinates fail to preserve curvature, as the mean curvature of a flat boundary degenerates to zero. To address these challenges and circumvent the intricate issue of spatial regularity in Lagrangian coordinates, we utilize reference surfaces to represent the free boundary and develop new energy functionals that both preserve the material derivative and incorporate spatial-temporal scaling $\partial_t \sim \nabla^{\frac{3}{2}}$ in Eulerian coordinates. This method enables us to establish both low-order and high-order regularity estimates without any loss of regularity. More importantly, we prove a blow-up criterion and provide a complete classification of blow-ups, including the self-intersection of the free boundary (which the graph assumption cannot handle), the breakdown of the mean curvature, and the blow-up of the normal velocity (which Lagrangian coordinates fail to capture). To the best of our knowledge, this is the first result addressing the a priori estimates and the blow-up criterion for free boundary problems with surface tension in general regions.
\end{abstract}

\maketitle
\tableofcontents

\section{Introduction}\label{s:int}
We consider the three-dimensional free boundary incompressible ideal magnetohydrodynamics (MHD) equations with surface tension in a bounded domain:
\begin{equation}\label{e:mhd}
\begin{cases}
\matder v+\nabla p=H\cdot \nabla H, & \text{ in } \Omega_t, \\
\matder H=H\cdot \nabla v,  & \text{ in } \Omega_t, \\ 
\dive v=0,\quad \dive H=0, & \text{ in } \Omega_t, \\
H\cdot \nu=0,\quad  
p=\meancurv_{\parOmega_t},\quad v_n=V_{\parOmega_t},  & \text{ on }  \parOmega_t, \\ 
v(0, \cdot)=v_0, \quad H(0, \cdot)=H_0, & \text{ in } \Omega_0,
\end{cases}
\end{equation}
where $t$ represents the time, $v$ the velocity, $\matder\coloneqq \partial_t+v\cdot\nabla$ the material derivative, $H$ the magnetic field, and $p$ the scalar total pressure. The moving domain $\Omega_t\subset\mathbb{R}^3$ is bounded with a closed surface $\parOmega_t\coloneqq \partial\Omega_t$. $\nu$ denotes the unit outer normal, $\meancurv_{\parOmega_t}$ the mean curvature, and $V_{\parOmega_t}$ the normal velocity of $\parOmega_t$, which is equal to the normal component of the velocity $v_n\coloneqq v\cdot\nu$. We specify the initial data $v_0,H_0$ and $\Omega_0$, denoting $\parOmega_0\coloneqq \partial\Omega_0$. Additionally, the coefficient of surface tension is assumed to be $1$ for simplicity.

In this paper, we establish a priori estimates and present a complete classification of the blow-up behavior for system \eqref{e:mhd} in Sobolev spaces. To ensure the generality of our results, we impose no additional assumptions on the fluid region or the free boundary.
\subsection{Energy functionals preserving the material derivative in Eulerian coordinates}
Our analysis relies crucially on the new energy functionals constructed below in the Eulerian coordinates. 
For any integer $l\ge 1$, we define
\begin{align}
\energy_l(t)\coloneqq {}& \frac{1}{2} \left(\|\matder^{l+1}v\|_{L^2(\Omega_t)}^2+\|\matder^{l+1}H\|_{L^2(\Omega_t)}^2+\|\tangrad(\matder^l v\cdot\nu)\|_{L^2(\parOmega_t)}^2\right) \nonumber\\
&+\frac{1}{2}\left(\|\nabla^{\lfloor \frac{3l+1}{2}\rfloor} \curl v\|_{L^2(\Omega_t)}^2+\|\nabla^{\lfloor \frac{3l+1}{2}\rfloor} \curl H\|_{L^2(\Omega_t)}^2\right), \label{e:el(t)}
\end{align}
and we define the lower-order energy as $\barenergy(t)=\energy_1(t)+\energy_2(t)+\energy_3(t)$, while the case $l \geq 4$ corresponds to the higher-order energy. In \eqref{e:el(t)}, $\lfloor\cdot\rfloor$ represents the integer part of a given number, $\tangrad$ denotes the tangential derivative, and $\curl F=\nabla F-(\nabla F)^\top$ applies to a vector field $F$. Additionally, we introduce the following energy functional:
\begin{align}
E_l(t)\coloneqq {}&\sum_{k=0}^{l}\left(\|\matder^{l+1-k}v\|_{H^{\frac 32 k}(\Omega_t)}^2+\|\matder^{l+1-k}H\|_{H^{\frac 32 k}(\Omega_t)}^2\right) \nonumber\\
&+\| v\|_{H^{\lfloor\frac{3l+3}{2}\rfloor}(\Omega_t)}^2+\| H\|_{H^{\lfloor\frac{3l+3}{2}\rfloor}(\Omega_t)}^2+\|\tangrad(\matder^l v\cdot\nu)\|_{L^2(\parOmega_t)}^2+1,\quad l\ge 1,\label{e:El(t)}
\end{align}
where we take into account the spatial-temporal regularity. As before, the lower-order energy 
\begin{equation*}
\barE(t)\coloneqq\sum_{k=0}^{4}\left(\|\matder^{4-k}v\|_{H^{\frac{3}{2}k}(\Omega_t)}^2+\|\matder^{4-k}H\|_{H^{\frac{3}{2}k}(\Omega_t)}^2\right)+\sum_{k=1}^{3}\|\tangrad(\matder^k v\cdot\nu)\|_{L^2(\parOmega_t)}^2+1, 
\end{equation*}
and we observe that $C_1(E_1+E_2+E_3)\le  \barE\le C_2(E_1+E_2+E_3)$ for some constants $C_1,C_2>0$.  
\subsubsection*{The principle of reducing derivatives}
The scaling $3/2$ in \eqref{e:El(t)} is revealed in \cite{Shatah2008a} that a second-order time derivative can be roughly equated to a third-order spatial differentiation, indicating the regularizing effect of the surface tension. From system \eqref{e:mhd}, this scaling suggests that we can reduce ``$1/2$-order'' spatial regularity by substituting $\matder v=-\nabla p+H\cdot\nabla H$ or $\matder H=H\cdot\nabla v$. In this sense, we can also reduce ``$1/2$-order'' spatial regularity when the operators $\matder$ and $\curl$ are combined (cf. Lemma \ref{l:curl Dt}). These observations are crucial in deriving the optimal expressions for $\dive\matder^l v,\curl\matder^l v$, the error terms, etc. (see, e.g., Lemmas \ref{l:nab^2H,H} and \ref{l:formu 5}) which allow us to 
control the higher-order energy (cf. Lemma \ref{p:ene est 2}). 

This principle will be consistently used throughout the paper.  
\subsection{Representation of the free boundary and the a priori assumptions}\label{s:RFB}
Let $(v,H,p,\Omega_t)$ be any solution to system \eqref{e:mhd} on $\left[0,T_0\right)$ for some $T_0>0$. We choose a smooth, compact reference surface $\Gamma$ to  represent the free boundary. Here, $\Gamma=\partial\Omega$, where $\Omega$ is a smooth, compact domain satisfying the uniform interior and exterior ball condition with radius $\radi=\radi(\Omega)>0$. 

The free boundary is represented as:
\begin{equation*}
\parOmega_t=\{x+h(x,t)\nu_{\Gamma}(x):x\in\Gamma\},\quad t\in \left[0,T\right),
\end{equation*} 
where the time $T\le T_0$ and the height function $h:\Gamma\times\left[0,T\right)\to\mathbb{R}$ are characterized as follows:
\begin{equation}\label{e:MT}
\mathcal{M}_T\coloneqq \radi-\sup_{0\le t< T}\|h(\cdot,t)\|_{L^\infty(\Gamma)}>0.
\end{equation}
In other words, $h(\cdot,t)$ is well-defined in $\left[0,T\right)$ as long as $\mathcal{M}_T>0$. The maximal representation interval $\left[0,T_r\right)$ for the reference surface $\parOmega$ is defined as $T_r=\sup\{ T\le T_0: \text{\eqref{e:MT} holds} \}$. It should be noted that one of the following three scenarios will occur as time approaches $T_r$.  
\begin{enumerate}[label={\rm (\arabic*)}]
\item The free boundary $\parOmega_t$ first self-intersects at time $T_r$ ($T_r<T_0$ or $T_r=T_0$), resulting in a splash or splat singularity (see, e.g., \cite{Coutand2014}). That is, $\radi(\Omega_{t})>0$ for $0\le t<T_r$ and $\radi(\Omega_{T_r})=0$.
\item $T_r=T_0$ and $\parOmega_t$ does not self-intersect on $\left[0,T_0\right)$. In this scenario, we complete the representation of the free boundary throughout the existence of the solution.
\item $T_r<T_0$ and $\parOmega_t$ does not self-intersect on $\left[0,T_r\right)$. In this case, our reference surface is insufficient to represent the free boundary at time $T_r$, necessitating a switch to a new reference surface to continue the representation. 
\end{enumerate}  
 
Having defined $\mathcal{M}_T$ to ensure the well-definedness of the height function, we introduce the following quantity to ensure the extension of the solution
\begin{equation}\label{e:NT}
\mathcal{N}_T \coloneqq \sup_{0\le t< T}(\|h(\cdot,t)\|_{H^{3+\delta}(\Gamma)}+\|\nabla v\|_{H^{3}(\Omega_t)}+\|\nabla H\|_{H^{3}(\Omega_t)}+\|v_n\|_{H^{4}(\parOmega_t)}),
\end{equation}  
where $\delta>0$ is a sufficiently small constant and $T\le T_0$. 

We mention that the requirements for the height function and the normal velocity are natural, as we do not fix the boundary using Lagrangian coordinates. 
These two parts precisely control the spatial and temporal regularity of the free boundary:
\begin{enumerate}[label={\rm (\arabic*)}]
\item $\|h\|_{H^{3+\delta}(\Gamma)}$ controls the tangential derivative of the height function. It also ensures that the second fundamental form $B_{\parOmega_t}$ is uniformly bounded, i.e., $\|B_{\parOmega_t}\|_{L^\infty(\parOmega_t)}\le C$. 
\item Note that $\partial_t h=v_n$, and therefore $\|v_n\|_{H^{4}(\parOmega_t)}$ controls the time derivative of the height function. 
\end{enumerate} 
Moreover, $\|v\|_{L^2(\Omega_t)}$ and $\|H\|_{L^2(\Omega_t)}$ are not included, due to the energy conservation of system \eqref{e:mhd}.   
\subsection{Main results} 
We make the following assumptions on the initial data throughout the paper. Let $v_0,H_0\in H^6(\Omega_0;\mathbb{R}^3)$ be the initial divergence-free velocity and magnetic fields, satisfying $H_0\cdot\nu_{\parOmega_0}=0$ on $\parOmega_0$, where $\Omega_0$ is the initial bounded domain, and the initial boundary $\parOmega_0\in H^7$ is a non-self-intersecting closed surface.
As discussed in Section \ref{s:RFB}, we can choose a suitable reference surface $\parOmega=\partial\Omega$ with $\radi=\radi(\Omega)>0$, and represent the free boundary. In particular, $\parOmega_0=\{ x+h_0(x)\nu_\Gamma (x):x\in \Gamma \}$, where $\|h_0\|_{L^\infty(\Gamma)}< \radi$. 

Our main results are stated as follows. 

\begin{theorem}\label{t:main t}
Let $(v,H,\Omega_t)$ be any solution to system \eqref{e:mhd} on $ \left[0, T\right)$ for some $T>0$ with initial data $(v_0,H_0,\Omega_0)$, and satisfies the following the a priori assumptions:
\begin{equation}\label{e:a priori}
\mathcal{N}_T<\infty \text{, and } \mathcal{M}_T>0.
\end{equation}

Then, we have the following results:
\begin{enumerate}[label={\rm (\arabic*)}]
\item Lower-order quantitative regularity estimates: 
\begin{equation}\label{e_t1_1}
\sup_{0\le t< T}\left( \barE(t)+\sum_{k=0}^{3}\|\matder^{3-k}p\|_{H^{\frac{3}{2}k+1}(\Omega_t)}^2+\|B_{\parOmega_t}\|_{H^5(\parOmega_t)}^2 \right) \leq \bar{C},
\end{equation}
where $\bar{C}$ is a constant that depends only on $T,\mathcal{N}_T,  \mathcal{M}_T,\|v_0\|_{H^6 (\Omega_0)},\|H_0\|_{H^6 (\Omega_0)}$ and $\|\meancurv_{\parOmega_0}\|_{H^5(\parOmega_0)}$.  Specifically, the following holds:
\begin{equation}\label{e_t1_2}
\sup_{0\le t< T}\left[ \sum_{k=0}^{4}\left( \|\partial_t^{4-k} v\|_{H^{\frac{3}{2}k}(\Omega_t)}^2+\|\partial_t^{4-k} H\|_{H^{\frac{3}{2}k}(\Omega_t)}^2\right)+\sum_{k=0}^{3}  \|\partial_t^{3-k} p\|_{H^{\frac{3}{2}k+1}(\Omega_t)}^2 \right]  \leq \bar{C},
\end{equation}
where the constant $\bar{C}$ depends on the same quantities as in \eqref{e_t1_1}.
\item Higher-order regularity estimates for $l\ge 4$: 
\begin{equation}\label{e_t1_3}
\sup _{0\le t< T}E_l(t) \le C_l,
\end{equation}
where $C_l$ is a constant that depends on $l, T, \mathcal{N}_T, \mathcal{M}_T$, and $E_l(0)$. 
In particular, we have
\begin{align}
\sup_{0\le t< T}\bigg[&\sum_{k=0}^{l}\left(\|\partial_t^{l+1-k}v\|_{H^{\frac{3}{2}k}(\Omega_t)}^2+\|\partial_t^{l+1-k}H\|_{H^{\frac{3}{2}k}(\Omega_t)}^2+\|\partial_t^{l-k}p\|_{H^{\frac{3}{2}k+1}(\Omega_t)}^2\right)\nonumber \\
&+\| v\|_{H^{\lfloor\frac{3(l+1)}{2}\rfloor}(\Omega_t)}^2+\| H\|_{H^{\lfloor\frac{3(l+1)}{2}\rfloor}(\Omega_t)}^2+\|B_{\parOmega_t}\|_{H^{\frac{3l+1}{2}}(\parOmega_t)}^2\bigg]\le C_l,\label{e_t1_4}
\end{align} 
where the constant $C_l$ depends on the same quantities as in \eqref{e_t1_3}.
\item There exists a time $T_0>0$ depending only on the initial quantities $\mathcal{M}_0, \|v_0\|_{H^6(\Omega_0)},\|H_0\|_{H^6(\Omega_0)}$, and $\|\meancurv_{\parOmega_0}\|_{H^5(\parOmega_0)}$, such that the a priori assumptions \eqref{e:a priori} hold for $T=T_0$.
\end{enumerate}
Notably, if we consider a smooth solution on $\left[0,T\right)$, it will not develop singularities at time $T$ and remains smooth with respect to both time and space, as long as the a priori assumptions \eqref{e:a priori} hold.
\end{theorem}  

Next, we present the classification of blow-up for system \eqref{e:mhd}, which fully captures the scenario of boundary self-intersection.
\begin{theorem}\label{t:main t2}
For any solution $(v,H,\Omega_t)$ to system \eqref{e:mhd} with initial data $(v_0,H_0,\Omega_0)$, define the maximal time interval of existence $\left[0,T_*\right)$, where $T_*$ is the maximal time such that
\begin{equation*}
v,H\in C_t^0H^6(\Omega_t)\ \text{and}\ \parOmega_t\in C_t^0H^7. 
\end{equation*}
If the maximal time $T_*<\infty$, then one of the following scenarios must occur:
\begin{enumerate}[label={\rm (\arabic*)}]
\item The free boundary $\parOmega_t$ self-intersects for the first time at time $T_*$.
\item Either the mean curvature does not belong to the $H^{1+\delta}$-class, or the free boundary $\parOmega_t$ does not belong to the $H^{2+\vare}$-class at time $T_*$, for some sufficiently small positive constants $\delta$ and $\vare$.
\item The normal velocity of the free boundary $V_{\parOmega_t}$ does not belong to the $H^4$-class at time $T_*$.
\item The breakdown of lower-order quantities on $\Omega_t$, i.e., 
\begin{equation*}
\sup_{0\le t< T_*}
(\|\nabla v\|_{H^{3}(\Omega_t)}+\|\nabla H\|_{H^{3}(\Omega_t)})=\infty.
\end{equation*}
\end{enumerate} 
\end{theorem}  
\begin{remark}
We assume that the initial data $v_0,H_0\in H^6$ is due to the consideration of a general bounded domain with a closed free surface. For a periodic flat initial region (e.g., $\mathbb{T}^2\times (0,1)$), we expect that the similar results of Theorems \ref{t:main t} and \ref{t:main t2} hold for initial data in $H^{\frac 92}$, as we can define the fractional derivative using the Fourier transform in this case.
\end{remark}  
\subsection{History and background}
In recent decades, there has been significant interest in studying the free boundary incompressible Euler equations, and substantial advancements have been made. Extensive research has been conducted for the irrotational case, especially the water wave equations. We refer readers to \cite{Germain2012,Ionescu2015,Lannes2005,Wu2011} and the references therein. If the fluid flow exhibits vorticity, one may refer to \cite{Christodoulou2000,Coutand2007,Disconzi2019,Fefferman2016,Ginsberg2021,Lindblad2005,Luo2024,Schweizer2005,Shatah2008a,Wang2021a,Zhang2008} for results on the a priori estimates, the local well-posedness with or without surface tension, the zero surface tension limit, and more.

The investigation of free boundary problems for MHD equations has emerged relatively recently compared to the study of the Euler equations, mainly because of the strong interactions between the magnetic and velocity fields. We focus on the incompressible MHD equations. Hao and Luo \cite{Hao2014} obtained a priori estimates for free boundary problems of the incompressible ideal MHD  without surface tension under the Taylor-type sign condition. They considered the case where the initial domain is homeomorphic to a ball. They also showed the ill-posedness of the problem if the Taylor-type sign condition is violated in the two-dimensional case \cite{Hao2020}. Luo and Zhang \cite{Luo2020} derived a priori estimates for the lower regular initial data in the initial domain of sufficiently small volume. In \cite{Gu2019}, a local existence result was provided, with a detailed proof in an initial flat domain $\mathbb{T}^2\times (0,1)$. The local well-posedness for the incompressible ideal MHD equations with surface tension is established by Gu, Luo, and Zhang in \cite{Gu2023}, in the same initial domain setting, namely, $\mathbb{T}^2\times (0,1)$. The nonlinear stability of the current-vortex sheet in the incompressible MHD equations was solved by Sun, Wang and Zhang \cite{Sun2018} under the Syrovatskij stability condition, assuming that the free boundaries are graphs in $\mathbb{T}^2\times (-1,1)$. Wang and Xin \cite{Wang2021} established the global well-posedness of a free interface problem for the incompressible inviscid resistive MHD under similar assumptions regarding the graph. We also refer to some related works \cite{Fu2023,Hao2017,Hao2024,Lee2017,Sun2019,Trakhinin2009,Trakhinin2021} on the topics of the well-posedness, the current-vortex sheets problem, the breakdown criterion, the viscous splash singularity,  and the compressible MHD. 

It should be noted that the aforementioned well-posedness results for the incompressible MHD equations are primarily derived by applying the Lagrangian coordinates, which transform a moving domain into a fixed one. However, as indicated in \cite{Shatah2008a,Shatah2011}, the Lagrangian map lacks maximal regularity because all the variables are defined on an evolving domain. In fact, the moving surface can also be described using alternative methods, such as the study of the Euler equations with surface tension \cite{Schweizer2005}, the fluid interface problem \cite{Liu2023a,Shatah2011}, the surface diffusion flow with elasticity \cite{Fusco2020}, and the motion of charged liquid drop \cite{Julin2024}, among others. 

Moreover, previous results on the incompressible MHD equations with surface tension predominantly apply to the flat periodic initial region $\mathbb{T}^2\times (a,b)$ and rely on the graph assumption for the free boundary.  However, the periodic assumptions and the graph assumptions have inherent limitations. In fact, it may be possible to reduce the problem of a general free boundary to the case of a graph by selecting local coordinates. However, this reduction is technically complicated and involves significant challenges. In the presence of surface tension, if we only select a portion of the free boundary and flatten it near a point, there is a risk of losing certain geometric characteristics of the free boundary, such as the evolution of its curvature. For the fluid in the flat domain $\mathbb{T}^2\times (a,b)$, its initial mean curvature is evidently zero, as local coordinates fail to preserve the curvature. These facts highlight the necessity of making additional assumptions on the initial velocity on the boundary. For instance, in \cite{Luo2021}, the assumption $v_0\in H^{3.5}(\mathbb{T}^2\times (0,1))\cap H^4(\mathbb{T}^2\times \{1\})$ is made to obtain the a priori estimates; in \cite{Gu2023}, $v_0\in H^{4.5}(\mathbb{T}^2\times (0,1))\cap H^5(\mathbb{T}^2\times \{1\})$ is made to establish local existence. To the best of our knowledge, the local well-posedness for system \eqref{e:mhd} with surface tension remains open when $\Omega_t$ is a general bounded domain with a closed free surface. 

In this paper, by constructing new energy functionals with spatial-temporal scaling $\partial_t\sim\nabla^{\frac 32}$ in Eulerian coordinates, we establish the a priori estimates on the general domain without any loss of regularity. We also eliminate the additional regularity requirement for the velocity on the initial boundary \cite{Gu2023,Luo2021} and our results highlight the effectiveness of employing the height function on the reference surface to analyze the evolution of curvature. 

It is also natural and fundamentally important to consider the breakdown criterion of solutions to system \eqref{e:mhd}, for which we are unaware of any relevant rigorous studies, although a few studies are available if we neglect the surface tension. Fu, along with both authors and Zhang, established a Beale-Kato-Majda continuation criterion for solutions to the free boundary incompressible ideal MHD equations without surface tension \cite{Fu2023}. When the viscosity is taken into account, the authors proved the existence of finite-time splash singularities \cite{Hao2024}, while Hong, Luo, and Zhao also demonstrated the existence of such singularities \cite{Hong2025}. Recently, Ifrim et al. established a low-regularity blow-up criterion for the incompressible ideal MHD equations without surface tension \cite{Ifrim2024a}, inspired by the previous works \cite{Ifrim2023,Ifrim2024}.  

Based on the a priori estimates, we provide a complete classification of blow-up behavior for solutions to system \eqref{e:mhd}. 
In contrast to the graph assumption, which cannot capture non-graphical free boundaries, our method allows the analysis of free boundaries approaching self-intersection. Moreover, our energy functionals are defined in Eulerian coordinates, and the a priori assumptions—apart from the height function used to characterize the regularity of the boundary—are independent of the choice of coordinates. Therefore, our method remains unaffected by different coordinate choices as the free boundary approaches self-intersection.   
\subsection{Novelties and structure of the paper}
The novelties of this study are as follows.

To the best of our knowledge, Theorem \ref{t:main t} is the first result focusing on the regularity estimates of system \eqref{e:mhd} in a general bounded domain with a closed free surface, i.e., without imposing any periodicity or simple connectedness assumptions on the fluid region, or any graph assumptions on the free boundary. 
\begin{enumerate}
\item[a)] Our a priori estimates are derived from an energy inequality of the following form, based on the a priori assumptions, without requiring smallness in time. That is,
\begin{equation*}
\dot\barE\lesssim_{\mathcal{N}_T,\mathcal{M}_T}C(\text{initial data})\barE,\quad \dot E_l\lesssim_{\mathcal{N}_T,\mathcal{M}_T,\text{induction}}CE_l,\quad l\ge 4.
\end{equation*}
This is crucial for establishing a breakdown criterion \cite{Fu2023,Ifrim2023,Ifrim2024a,Ifrim2024,Julin2024,Luo2024,Wang2021a}. If additional smallness in time were required, we could not establish a blow-up criterion, let alone a complete classification of blow-up behavior. The common a priori estimates yield a polynomial of the energy, multiplied by time, such as $\sup_{[0,T]} \barE(t)\le C(\barE(0))+T^{\frac 12} \mathcal{P}(\sup_{[0,T]}\barE(t))$.
However, this inequality necessitates a sufficiently small time $T$ to complete the energy estimates, making the breakdown criterion unattainable.
\item [b)] Our lower-order regularity results \eqref{e_t1_1} extend the a priori estimates with an initial flat domain $\mathbb{T}^2\times (0,1)$ from \cite{Luo2021} to a general domain without any loss of regularity. Moreover, we eliminate the additional regularity requirement for the velocity on the initial boundary (which was assumed in \cite{Luo2021} as $v_0\in H^{3.5}(\mathbb{T}^2\times (0,1))\cap H^4(\mathbb{T}^2\times \{1\})$) since our final estimate does not depend on this initial quantity. We also establish higher-order energy estimates without any loss of regularity.
\item [c)] We establish a distinct energy functional that preserves the material derivative $\matder$ with a different spatial-temporal scaling ($\partial_t\sim\nabla^{\frac 32}$) in Eulerian coordinates, in contrast to the energy functional defined in the flat periodic domain using Lagrangian coordinates \cite{Gu2023,Luo2020,Luo2021}. This strategy avoids destroying the structure of system \eqref{e:mhd} when separating $\partial_t$ from $\matder$, and the energy estimates are driven by the second fundamental form and pressure. We also eliminate the additional regularity requirement for the velocity on the initial boundary as in \cite{Gu2023}, i.e., the assumption $v_0\in H^{4.5}(\mathbb{T}^2\times (0,1))\cap H^5(\mathbb{T}^2\times \{1\})$. 
\end{enumerate}

Theorem \ref{t:main t2} provides the first comprehensive classification of blow-ups for solutions of \eqref{e:mhd}.  
\begin{enumerate}
\item[a)] In our classification, the first three types of singularities arise from the free boundary and are mutually distinct. These singularities can be effectively characterized using the height function: (1), (2), and (3) in Theorem \ref{t:main t2} correspond to the inability to choose a reference surface to define the height function, the blow-up of the tangential derivative of the height function, and the blow-up of the time derivative of the height function, respectively. Therefore, each of these three types of singularities is indispensable.
\item[b)] The case where only the singularity in Theorem \ref{t:main t2} (1) arises, while the others in (2)–(4) do not occur, does exist. The singularity of boundary self-intersection, where the solution and free boundary remain smooth, exists in the presence of viscosity \cite{Hao2024,Hong2025}. For the free boundary incompressible ideal MHD equations with surface tension, it is conjectured in \cite{Coutand2014} that this singularity also exists.
\item[c)] If we consider the fixed boundary problem, our blow-up classification reduces to (4) in Theorem \ref{t:main t2}, analogous to the remarkable Beale-Kato-Majda criterion for the Euler equations \cite{Beale1984}.
\end{enumerate}

Our results hold without assuming that the free boundary is a graph. Analyzing the evolution of a small region by selecting a portion of the closed surface and applying local coordinate flattening is insufficient to solve the problem. Moreover, the strategy for selecting the reference surface provides the following advantages compared to the graph assumption.
\begin{enumerate} 
\item[a)] When the free boundary is represented by a graph function over the initial boundary $\mathbb{T}^2\sim\mathbb{T}^2\times \{1\}$, it corresponds to a specific height function. Choosing $\mathbb{T}^2\times \{1\}$ as the reference surface with $(0,0,1)$ as the unit outer normal, the height function coincides with the graph function.
\item[b)] The height function enables direct computation of curvature evolution via tangential derivatives, whereas flattening the surface with local coordinates fails to preserve the intrinsic geometric properties of the moving surface.
\item[c)] We can continually select appropriate reference surfaces to represent the free boundary, particularly facilitating the characterization of the process by which the boundary develops self-intersection. However, the graph function fails when the moving surface boundary undergoes turning (see, e.g., the breakdown criterion for the free boundary Euler equations with surface tension in \cite{Luo2024} and without surface tension in \cite{Wang2021a}). 
\end{enumerate}

We use reference surfaces to represent the free boundary, which offers advantages over fixing the boundary in Lagrangian coordinates for the following reasons.
\begin{enumerate} 
\item[a)] It is more convenient to control the mean curvature and boundary regularity using the height function, as the regularity improvement of the free boundary is geometric \cite{Shatah2008a}, directly connected to the regularity of the mean curvature (cf. Lemma \ref{l:bou reg}), and not entirely evident in the Lagrangian coordinates.  
\item[b)] We avoid addressing the issue of spatial regularity of the flow map in Lagrangian coordinates.
\item[c)] A more precise estimation of the pressure can be obtained by analyzing the normal velocity 
of the free boundary. In contrast, in Lagrangian coordinates, the normal velocity of the free boundary is implicit because the boundary is fixed.
\end{enumerate} 

The rest of this paper is organized as follows. In Section \ref{s:formu}, we calculate the commutators, the error terms, and additional terms to establish the energy estimates. In Section \ref{s:d/dt}, we compute the time derivative of the energy functional. In Section \ref{s:p est}, we will show that $\|p\|_{H^3(\Omega_t)}$ can be uniformly bounded within the time interval of existence. In Section \ref{s:est error}, we estimate the error terms that appeared in Section \ref{s:d/dt}. In Section \ref{s:energy est}, we close the energy estimates and prove our main theorems. Finally, in Section \ref{s:discuss}, we discuss the connection between the self-intersection and the curvature blow-up on the free boundary established in Theorem \ref{t:main t2}.
\section{Formulas for the energy estimates}\label{s:formu}

Throughout the paper, we will use the Einstein summation convention and the notation $S\star T$ from \cite{Hamilton1982} 
to denote a tensor formed by contracting certain indices of tensors $S$ and $T$ with constant coefficients. In particular, for $k,l\in\N=\{1,2,3,\cdots\}$ (we denote $\Nz =\{0,1,2,3,\cdots\}$), $\nabla^k f\star\nabla^l g$ represents a contraction of certain indices of tensors $\nabla^i f$ and $\nabla^j g$ for $0\le i\le k$ and $0\le j\le l$ with constant coefficients.  
Note that $f$ and $g$ can be vector fields, and we include the lower-order derivatives along with the function (or vector field) itself. However, we exclude the case of a single term $\nabla^i f$. Let $u:\parOmega \to \mathbb{R}$ and $F:\parOmega\to\mathbb{R}^3$ be a sufficiently regular function and vector field, respectively. Since the reference hypersurface $\parOmega$ (embedded in $\mathbb{R}^3$) has a natural metric $g$ induced by the Euclidean metric, $(\parOmega, g)$ is a Riemannian manifold with connection $\connec$. For a function $u\in C^\infty(\parOmega)$ and a vector field $F$, $\connec_F u=Fu$. 

We denote the normal part of $F$ by $F_n\coloneqq F\cdot\nu_\parOmega$, and the tangential part by  $F_\sigma\coloneqq F-F_n\nu_\parOmega$, where ``$\cdot$'' denotes the inner product. If $\parOmega$ is smooth, we can extend both $u$ and $F$ to $\mathbb{R}^3$ and define the tangential differential by $\tangrad u\coloneqq (\nabla u)_\sigma$, the tangential gradient of $F$ by
$\tangrad F\coloneqq \nabla F-(\nabla F\nu)\otimes\nu$, i.e., $ (\tangrad F)_{ij}=\partial_j F^i-\partial_lF^i\nu^l\nu_j$, and the tangential divergence by $\tandive F\coloneqq \trace(\tangrad F)$.     
The tangential gradient and covariant gradient are equivalent: for any vector field $\tilde{F}:\parOmega\to\mathbb{R}^3,\tilde{F}\cdot\nu=0$, we have $\connec_{\tilde{F}} u=\tangrad u\cdot\tilde{F}$.   
Additionally, the second fundamental form $B$ and the mean curvature $\meancurv$ can be written as $B=\tangrad \nu$ and $\meancurv=\tandive\nu$.
The Beltrami-Laplacian is defined by $\surflaplace u\coloneqq \tandive(\tangrad u)$, and it holds 
\begin{equation}\label{e:lap_B}
\surflaplace u=\laplace u-\left( \nabla^2u \nu \cdot \nu\right) -\meancurv \partial_\nu u,
\end{equation}
where $\partial_\nu$ denotes the outer normal derivative. We also recall the divergence theorem $\int_{\parOmega}\tandive FdS=\int_{\parOmega} \meancurv_\parOmega(F \cdot \nu_\parOmega)dS$,
and the differentiation formula (see, e.g., \cite{Shatah2008a}) 
\begin{equation}\label{e:reynold 2}
\frac{d}{dt}\int_{\parOmega_t}fdS=\int_{\parOmega_t}\matder f+f\tandive vdS.
\end{equation} 

We will fix our reference surface $\Gamma$, a boundary of a smooth, compact set $\Omega$ satisfying the uniform interior and exterior ball condition with radius $\radi>0$. We denote its tubular neighborhood $U(\radi,\Gamma)= \{ x \in \mathbb{R}^3: \operatorname{dist}(x,\Gamma) < \radi\}$.  We say that $\parOmega_t=\partial\Omega_t$ (or $\Omega_t$) is $H^s(\Gamma)$-regular, if $\parOmega_t=\{x+h(x,t)\nu_{\Gamma}(x):x\in\Gamma\}$, where $h(\cdot,t):\Gamma\to\mathbb{R}$ is $H^s(\Gamma)$-regular and $\|h(\cdot,t)\|_{L^\infty(\Gamma)}<\radi$. $\parOmega_t$ is called uniformly $H^s(\Gamma)$-regular if $\|h\|_{H^s(\Gamma)} \le C$ and $\|h\|_{L^\infty(\Gamma)} \leq c \radi$ for constants $C$ and $c <1$ (see \cite{Julin2024} for similar definitions).  We can express the unit outer normal and the second fundamental form by the height function (cf. \cite{Mantegazza2011})
\begin{equation}\label{e:nu by h} 
\nu_{\parOmega_t}=a_1\left( h(\cdot,t),\tangrad h(\cdot,t)\right),\quad  B_{\parOmega_t}=a_2\left( h(\cdot,t),\tangrad h(\cdot,t)\right) \tangrad^2 h(\cdot,t),
\end{equation} 
where $a_1,a_2\in C^\infty$.  We extend $\nu$ to $\Omega$ via harmonic extension and denote it as $\tilde{\nu}$. 
We sometimes still denote the extended one by $\nu$. From \eqref{e:a priori} and \eqref{e:nu by h}, $\|\tilde{\nu}\|_{H^{5/2+\delta}(\Omega_t)}\le C$ for $\delta>0$ small.

From the definition $\curl F=\nabla F-(\nabla F)^\top$, a straightforward calculation yields:
\begin{lemma}\label{l:formu 2}
Let $l,k\in\N, F$, and $G$ be smooth vector fields and $f$ be a smooth function. Then, we have: 
\begin{enumerate}[label={\rm (\arabic*)}]
\item $\curl (F\cdot\nabla G)=\nabla G\nabla F-\nabla F^{\top}\nabla G^{\top}+(F\cdot\nabla)\curl G$ and $[\matder, \curl]F=\nabla v^{\top}\nabla F^{\top}-\nabla F\nabla v$.
\item $[\matder^{l+1}, \nabla^k]f=\matder[\matder^{l},\nabla^k]f+[\matder,\nabla^k]\matder^{l}f$ and $[\matder^{l}, \nabla^{k+1}]f=[\matder^l,\nabla]\nabla^{k}f+\nabla[\matder^l,\nabla^{k}]f$. 
\end{enumerate}  
\end{lemma} 

To derive a general formula for the commutators, we apply the following results. It is easy to verify that $\matder a(\nu)=b(\nu)\tangrad v, 
\matder \nabla \matder^k v=\nabla \matder^{k+1}v+\nabla v\star \nabla \matder^k v,\matder \tangrad \matder^k v=\tangrad \matder^{k+1}v+\tangrad v\star \tangrad \matder^k v$ for $k\in \N $, where $a(\nu)$ and $b(\nu)$ denote the finite $\star$ product of $\nu$.  
\begin{lemma}\label{l:formu 3}
Let $l,k\in \N,l\ge 2$ and $k\ge 3$. Then, we have:
\begin{enumerate}[label={\rm (\arabic*)}]
\item $[\matder, \nabla^2] f =\nabla v\star \nabla^2f+\nabla^2 v\star\nabla f$.
\item $[\matder, \nabla^k] f=\sum_{|\alpha| \leq k-1}\nabla^{1+\alpha_1}  v  \star \nabla^{1+\alpha_{2}} f$.
\item $[\matder^l, \nabla] f=\sum_{2\le m\le l+1}\sum_{|\beta| \leq l+1-m} 
\nabla \matder^{\beta_1} v \star \cdots\star \nabla \matder^{\beta_{m-1}} v\star \nabla \matder^{\beta_{m}} f$.
\item $[\matder^l, \nabla^2] f =\sum_{2\le m\le l+1}\sum_{\substack{|\alpha|\le 1,|\beta| \leq l+1-m}}\nabla^{1+\alpha_1} \matder^{\beta_1} v \star \cdots\star \nabla^{1+\alpha_{m-1}} \matder^{\beta_{m-1}}v\star \nabla^{1+\alpha_{m}} \matder^{\beta_{m}} f$.
\end{enumerate} 
Roughly speaking, the leading term is $\nabla^k\matder^{l-1}$ in the commutator $[\matder^l,\nabla^k]$.
\end{lemma}
\begin{proof}
A direct calculation yields the first claim and the second claim can be found in \cite[Lemma 4.1]{Julin2024}. We prove the third one by induction, and it is easy to verify the case of $l=2$. For the case of $l\ge 3$, from Lemma \ref{l:formu 2} and the above formulas, it follows that
\begin{align*}
[\matder^l, \nabla]f={}&\matder[\matder^{l-1},\nabla]f+\nabla v\star \nabla \matder^{l-1}f\\
={}&\matder(\sum_{2\le m\le l}\sum_{|\beta| \leq l-m} 
\nabla \matder^{\beta_1} v \star \cdots \star \nabla \matder^{\beta_{m-1}} v \star \nabla \matder^{\beta_{m}} f)+\nabla v\star \nabla \matder^{l-1}f\\ 
={}&\sum_{2\le m\le l+1}\sum_{|\beta| \leq l+1-m} 
\nabla \matder^{\beta_1} v \star \cdots \star \nabla \matder^{\beta_{m-1}} v \star \nabla \matder^{\beta_{m}} f.
\end{align*} 
The last claim follows again by induction and we omit the proof.	 
\end{proof} 

Let $a_\beta(\nu)$ and $a_{\alpha, \beta}(\nu, B)$ denote the finite $\star$ product of the tensors.  
We provide a more precise formulation of the quantities than those in \cite[Lemma 4.2]{Julin2024}.
\begin{lemma}\label{l:formu 4}
Let $l\ge 1$ and we have the following results: 
\begin{enumerate}[label={\rm (\arabic*)}]
\item $[\matder^l, \tangrad] f=\sum_{2\le m\le l+1}\sum_{|\beta| \leq l+1-m} 
\tangrad \matder^{\beta_1} v \star \cdots \star \tangrad \matder^{\beta_{m-1}} v \star \tangrad \matder^{\beta_{m}} f$.
\item $\matder^l \nu=\sum_{1\le m\le l}\sum_{|\beta| \leq l-m} a_\beta(\nu) \tangrad \matder^{\beta_1} v \star \cdots \star \tangrad \matder^{\beta_m} v$.
\item $\matder^l B=\sum_{1\le m\le l}\sum_{\substack{|\beta|\le l-m,|\alpha|\le 1}}a_{\alpha,\beta}(\nu,B) \tangrad^{1+\alpha_1} \matder^{\beta_1} v \star \cdots \star \tangrad^{1+\alpha_m} \matder^{\beta_m} v$.
\item $[\matder^l, \tangrad^2]f 
=\sum_{2\le m\le l+1}\sum_{\substack{|\beta|\le l+1-m,|\alpha| \leq 1}} a_{\alpha,\beta}(\nu,B) \nabla^{1+\alpha_1} \matder^{\beta_1} v \star \cdots\star \nabla^{1+\alpha_{m-1}} \matder^{\beta_{m-1}} v \star \tangrad^{1+\alpha_m} \matder^{\beta_{m}} f$.
\end{enumerate} 
\end{lemma}
\begin{proof}
To prove the first claim, we recall $[\matder,\tangrad]f=-(\tangrad v)^\top\tangrad f$ in Lemma \ref{l:formu 1}. For the case of $l\ge 2$, we have by induction that 
\begin{align*}
[\matder^l,\tangrad] f
={}&\matder[\matder^{l-1}, \tangrad] f+[\matder,\tangrad]\matder^{l-1} f\\
={}&\matder(\sum_{2\le m\le l}\sum_{|\beta| \leq l-m} 
\tangrad \matder^{\beta_1} v \star \cdots \star \tangrad \matder^{\beta_{m-1}} v \star \tangrad \matder^{\beta_{m}} f)+\tangrad v\star\tangrad\matder^{l-1} f\\
={}&\sum_{2\le m\le l+1}\sum_{|\beta| \leq l+1-m} 
\tangrad \matder^{\beta_1} v \star \cdots \star \tangrad \matder^{\beta_{m-1}} v \star \tangrad \matder^{\beta_{m}} f.
\end{align*}
Similarly, we can obtain the last claim.  For the second claim, we recall $\matder \nu=\tangrad v\star\nu$, and for $l\ge 2$, it holds by induction.  
As for the third claim, we have for $l\ge 1$ that
\begin{align*}
\matder^l B
={}&[\matder^l, \tangrad]\nu +\tangrad\matder^l\nu\\
={}&\tangrad(\sum_{1\le m\le l}\sum_{|\beta| \leq l-m} a_\beta(\nu) \tangrad \matder^{\beta_1} v \star \cdots \star \tangrad \matder^{\beta_m} v)\\
&+\sum_{2\le m\le l+1}\sum_{|\beta| \leq l+1-m} \tangrad \matder^{\beta_1} v \star \cdots \star \tangrad \matder^{\beta_{m-1}} v \star \tangrad \matder^{\beta_{m}} \nu 
\eqqcolon I_1+I_2.
\end{align*}
It is clear that $I_1=\sum_{1\le m\le l}\sum_{\substack{|\beta|\le l-m,|\alpha|\le 1}}a_{\alpha,\beta}(\nu,B) \tangrad^{1+\alpha_1} \matder^{\beta_1} v \star \cdots \star \tangrad^{1+\alpha_m} \matder^{\beta_m} v$. 
For $I_2$, it follows that
\begin{align*}
I_2 
={}&\sum_{2\le m\le l+1}\sum_{|\beta| \leq l+1-m} \tangrad\matder^{\beta_1} v  \star \cdots \star \tangrad \matder^{\beta_{m-1}} v\\
&\star(\sum_{1\le n\le \beta_m}\sum_{\substack{|\lambda| \leq \beta_m-n,|\gamma|\le 1}}\tangrad^{1+\gamma_1} \matder^{\lambda_1} v \star \cdots\star \tangrad^{1+\gamma_n} \matder^{\lambda_n} v)\\
={}&\sum_{\substack{2\le m\le l+1,|\beta| \leq l+1-m}}\sum_{1\le n\le \beta_m}\sum_{\substack{|\lambda| \leq \beta_m-n,|\gamma|\le 1}}a_{\beta,\lambda,\gamma}(\nu,B) \tangrad \matder^{\beta_1} v \star \cdots\star \tangrad \matder^{\beta_{m-1}} v  \\
&\qquad\qquad\qquad\qquad\qquad\qquad\qquad\qquad\quad\star \tangrad^{1+\gamma_1} \matder^{\lambda_1} v\star \cdots \star \tangrad^{1+\gamma_n} \matder^{\lambda_n} v, 
\end{align*}
which is also contained in  $\sum_{1\le m\le l}\sum_{ |\alpha|\le 1,|\beta|\le l-m}a_{\alpha,\beta}(\nu,B) \tangrad^{1+\alpha_1} \matder^{\beta_1} v \star \cdots \star \tangrad^{1+\alpha_m} \matder^{\beta_m} v$. 
\end{proof}   

We denote the divergence of a matrix $A=(A_{ij})$ as $(\dive A)_i\coloneqq \sum_{j}\partial_j A_{ij}$ and recall $\curl F=\nabla F-(\nabla F)^\top$. For later use, we recall \cite[Lemma 3.3]{Julin2024}:
\begin{lemma}\label{l:3.3 jul}
Let $\Omega$ be a bounded domain with $C^{1,\alpha}$ boundary. For any smooth vector field $F$, we have $\|F\|_{L^2(\parOmega)}^2 \leq C(\|F_{\tau}\|_{L^2(\parOmega)}^2+\|F\|_{L^2(\Omega)}^2+\|\dive F\|_{L^2(\Omega)}^2+\|\curl F\|_{L^2(\Omega)}^2)$, where $ \tau=n,\sigma$.  
\end{lemma}

To estimate energy, we begin with the following basic results. By the divergence-free condition, it is clear that $\dive\matder v=\partial_i v^j\partial_j v^i$ and we have
\begin{equation}\label{e:laplace p}
-\laplace p=\partial_i v^j\partial_j v^i-\partial_i H^j\partial_j H^i.
\end{equation}  

A direct calculation produces the following identities.
\begin{lemma}\label{l:curl Dt}
For the velocity and magnetic fields, we have
\begin{enumerate}[label={\rm (\arabic*)}]
\item $\curl \matder v=(\nabla H)^{\top}\curl H+\curl H\nabla H+(H\cdot\nabla)(\curl H),[\matder,\curl] v=
-(\nabla v)^{\top}\curl v-\curl v\nabla v$.
\item $\curl \matder H=\nabla v\nabla H-(\nabla H)^{\top}(\nabla v)^{\top}+(H\cdot\nabla)(\curl v),[\matder, \curl] H=(\nabla v)^{\top}(\nabla H)^{\top}-\nabla H\nabla v$.
\end{enumerate} 
\end{lemma} 

Next, we introduce some errors associated with the magnetic field.   
Denote $R^0_{\nabla H,H}\coloneqq 0, R^0_{\nabla H,\nabla H}\coloneqq \nabla H\star\nabla H$, and for $k\ge 1$, we define
\begin{align*}
&R^k_{\nabla H,H}
\coloneqq \sum_{3\le m\le k+2}\sum_{\substack{|\alpha|\le 1,|\beta|\le k+2-m}}a_{\alpha,\beta}(\nabla v)\nabla^{1+\alpha_1}\matder^{\beta_1}v\star\cdots
\star \nabla^{1+\alpha_{m-2}}\matder^{\beta_{m-2}}v\star \nabla^{\alpha_{m-1}} H\star H,\\
&R^k_{\nabla H,\nabla H}\coloneqq \sum_{3\le m\le k+2}\sum_{\substack{|\alpha|\le 2, \alpha_i\le 1,|\beta|\le k+2-m}}\nabla^{1+\alpha_1}\matder^{\beta_1}v\star\cdots\star \nabla^{1+\alpha_{m-2}}\matder^{\beta_{m-2}}v
\star \nabla^{\alpha_{m-1}} H\star \nabla^{\alpha_{m}} H,  
\end{align*} 
where $a_{\alpha,\beta}(\nabla v)$ denotes the finite $\star$ product. In the case of $\beta_j=0,\nabla \matder^{\beta_j}$ can be absorbed into $a_{\alpha,\beta}(\nabla v)$. A direct calculation shows $\matder(\nabla H\star \nabla H)=\nabla^2 v\star H\star \nabla H+\nabla v\star\nabla H\star \nabla H$ and $\matder(\nabla H\star H)= \nabla^2 v\star H\star H+\nabla v\star \nabla H\star H$, and the following are the results for higher-order material derivatives. 
\begin{lemma}\label{l:nab H, nab H} 
Let $k\in \N $. We have $\matder^k(\nabla H\star \nabla H)=R^k_{\nabla H,\nabla H}$ and $\matder^k(\nabla H\star H)
=R^k_{\nabla H,H}$. 
\end{lemma}
\begin{proof}
It is sufficient to consider the  case of $k\ge 2$. We claim that given any $k\ge 2$, one has
\begin{align*}
&\matder^k(\nabla H\star \nabla H)=\sum_{2\le m\le k+2}\sum_{|\beta|\le k+2-m}\nabla\matder^{\beta_1}v\star\cdots\star \nabla\matder^{\beta_{m-2}}v 
\star \nabla\matder^{\beta_{m-1}}H\star \nabla\matder^{\beta_m}H,\\
&\matder^k(\nabla H\star H)
=\sum_{2\le m\le k+2}\sum_{|\beta|\le k+2-m}\nabla\matder^{\beta_1}v\star\cdots\star \nabla\matder^{\beta_{m-2}}v
\star \nabla\matder^{\beta_{m-1}}H \star \matder^{\beta_{m}}H.
\end{align*}
In fact, from Lemma \ref{l:formu 3}, we see that
\begin{align*}
\matder^k(\nabla H\star \nabla H)
={}&\nabla\matder^k H\star \nabla H+[\matder^k,\nabla] H\star \nabla H+\sum_{|\gamma|=k,\gamma_1,\gamma_2\ge 1}[\matder^{\gamma_1},\nabla] H\star[\matder^{\gamma_2},\nabla] H\\
&  +\nabla\matder^{\gamma_1} H\star[\matder^{\gamma_2},\nabla] H+\nabla\matder^{\gamma_1} H\star\nabla\matder^{\gamma_2} H\\
={}&\sum_{2\le m\le k+2}\sum_{|\beta| \leq  k+2-m} 
\nabla \matder^{\beta_1} v \star \cdots \star \nabla \matder^{\beta_{m-2}} v \star \nabla \matder^{\beta_{m-1}} H\star \nabla \matder^{\beta_{m}}H,\\
\matder^k(\nabla H\star H)
={}&\nabla\matder^k H\star H+[\matder^k,\nabla] H\star H+ \matder^k H\star \nabla H\\
&+\sum_{|\gamma|=k,\gamma_i\ge 1}[\matder^{\gamma_1},\nabla] H\star \matder^{\gamma_2} H+\nabla\matder^{\gamma_1} H\star \matder^{\gamma_2} H\\
={}&\sum_{2\le m\le k+2}\sum_{|\beta| \leq  k+2-m} 
\nabla \matder^{\beta_1} v \star \cdots \star \nabla \matder^{\beta_{m-2}} v \star \nabla \matder^{\beta_{m-1}} H\star \matder^{\beta_{m}}H.
\end{align*}
By substituting $\matder H=H\cdot \nabla v$ 
and by induction, it is readily verified that
\begin{align}
\matder^jH&=\sum_{1\le m\le j}\sum_{|\beta|\le j-m}\nabla\matder^{\beta_1}v\star\cdots\star \nabla\matder^{\beta_{m}}v\star H,\label{e:Dt^k H}\\
\nabla^i\matder^jH&=\sum_{1\le m\le j}\sum_{|\alpha|\le i,|\beta|\le j-m}\nabla^{1+\alpha_1}\matder^{\beta_1}v\star\cdots\star \nabla^{1+\alpha_{m}}\matder^{\beta_{m}}v\star \nabla^{\alpha_{m+1}}H,\label{e:nabDt^k H}
\end{align}
where $i,j\in \N $.  These conclude the proof of the lemma.
\end{proof}

The above lemma shows that $\matder^k(H\cdot\nabla H)=R^k_{\nabla H,H}$. Due to the divergence-free condition, it can be shown that taking the divergence does not increase the order of derivatives. 
\begin{lemma}\label{l:divDt}
We have the following results:
\begin{enumerate}[label={\rm (\arabic*)}]
\item $\dive\matder (H\cdot \nabla H)=\nabla^2 v\star \nabla H\star H+\nabla v\star \nabla H\star \nabla H+\nabla^2 H\star \nabla v\star H$.
\item For any integer $k\ge 2$, it holds $\dive\matder^k(H\cdot \nabla H)=\partial_j\partial_l\matder^{k-1}v^iH^l\partial_i H^j+\nabla^3\matder^{k-2}v\star\nabla v\star H\star H+\operatorname{L.O.T.}$, where $\operatorname{L.O.T.}$ stands for lower-order terms. 
\end{enumerate}   
\end{lemma}
\begin{proof}
By Lemma \ref{l:formu 1}, a direct calculation gives the first result. 
For $k\ge 2$, the divergence-free condition implies that $\partial_j\matder^{\gamma} \partial_i H^j=[\partial_j,\matder^{\gamma}] \partial_i H^j$, and therefore
\begin{align*}
\dive\matder^k(H\cdot \nabla H)
={}&\partial_j(\matder^k \partial_i H^j H^i)+\partial_j( \partial_i H^j\matder^kH^i)+\partial_j(\sum_{|\gamma|=k,\gamma_i<k}\matder^{\gamma_1} \partial_i H^j \matder^{\gamma_2}H^i)\\ 
={}&\partial_j\matder^kH^i\partial_i H^j+[\partial_j,\matder^k] \partial_i H^j H^i+[\matder^k,\nabla] H\star \nabla H+ \nabla \matder^{\gamma_1} H\star 
\nabla \matder^{\gamma_2}H\\
&+[\matder^{\gamma_1},\nabla] H\star\nabla\matder^{\gamma_2}H+\sum_{|\gamma|=k,\gamma_i<k}[\partial_j,\matder^{\gamma_1}] \partial_i H^j \matder^{\gamma_2}H^i.
\end{align*}
In the above, it suffices to consider the most challenging term  $\partial_j\matder^kH^i\partial_i H^j$. Note that $\partial_j\matder^kH^i=\partial_j\partial_l\matder^{k-1}v^iH^l+\sum_{|\gamma|=k-1,\gamma_1<k-1}\partial_j\partial_l\matder^{\gamma_1}v^i\matder^{\gamma_2}H^l+\sum_{|\gamma|=k-1}\partial_j[\matder^{\gamma_1},\partial_l]v^i\matder^{\gamma_2}H^l$, and we find that
\begin{align*}
\dive\matder^k(H\cdot \nabla H)
={}&\partial_j\partial_l\matder^{k-1}v^iH^l\partial_i H^j+[\nabla,\matder^k] \nabla H\star H+[\matder^k,\nabla] H\star \nabla H\\
&+\sum_{|\gamma|=k,\gamma_i<k}\big([\nabla,\matder^{\gamma_1}] \nabla H\star \matder^{\gamma_2}H+ \nabla \matder^{\gamma_1} H\star 
\nabla \matder^{\gamma_2}H+[\matder^{\gamma_1},\nabla] H\star 
\nabla \matder^{\gamma_2}H\big)\\ &+\sum_{|\gamma|=k-1,\gamma_1<k-1}\nabla^2\matder^{\gamma_1}v\star\matder^{\gamma_2}H\star\nabla H+\sum_{|\gamma|=k-1}\nabla[\matder^{\gamma_1},\nabla]v\star\matder^{\gamma_2}H\star\nabla H\\
\eqqcolon {}&\partial_j\partial_l\matder^{k-1}v^iH^l\partial_i H^j+R.
\end{align*}
Here, the highest-order term in $R$ is $\nabla^2\matder^{k-1}H\star\nabla v\star H$, resulting from $[\nabla,\matder^k]\nabla H\star H$. To complete the proof, we replace the material derivative with the spatial derivative, resulting in $\nabla^3\matder^{k-2}v\star\nabla v\star H\star H$, along with lower-order terms as shown in \eqref{e:nabDt^k H}.  
\end{proof} 

To derive the energy estimates by applying the div-curl estimates, it is inevitable to compute $\dive\matder^l v,\dive\matder^l H,\curl\matder^l v$, and $\curl\matder^l H$. The following lemma is crucial for computing $\curl \matder^l v$ (see Lemma \ref{l:formu 5}). 
\begin{lemma}\label{l:nab^2H,H}
It holds $\matder((H\cdot\nabla)(\curl H))
=\nabla^2\curl v\star H\star H+\nabla^2 H\star\nabla v\star  H+\nabla^2 v\star\nabla H\star H$, and 
\begin{align*}
&\matder^k((H\cdot\nabla)\curl H)\\
={}&\nabla^{k+1}\curl H\star \underbrace{H\star \cdots\star H}_{k\text{ times }}+\sum_{\substack{|\alpha|,m\le k+2,\alpha_i\le k+1,F_j=v,H}}\nabla^{\alpha_1}F_1\star\cdots\star\nabla^{\alpha_m}F_m\\
& +\!\!\sum_{\substack{|\alpha|+|\beta|\le k+2,\alpha_i+\beta_i\le k+1,\\m\le k+1,\beta_i\le k-1,F_j=v,H}}\!\!\!\!\!\!\!\!\!\!\!\!\nabla^{\alpha_1}\matder^{\beta_1}v\star\cdots\star \nabla^{\alpha_{k-1}}\matder^{\beta_{k-1}}v\star \nabla^{\alpha_k} F_k\star\cdots\star\nabla^{\alpha_m}F_m,
\end{align*}
if $k\ge 2$ is even. For odd $k\ge 3$, we replace $\nabla^{k+1}\curl H\star \underbrace{H\star \cdots\star H}_{k\text{ times }}$ by $\nabla^{k+1}\curl v\star \underbrace{H\star \cdots\star H}_{k\text{ times }}$. 
\end{lemma}  
\begin{proof}
First, we apply Lemma \ref{l:curl Dt} to obtain $\matder[(H\cdot\nabla)(\curl H)]
=\nabla^2\curl v\star H\star H+\nabla^2 H\star\nabla v\star  H+\nabla^2 v\star\nabla H\star H$. 
In the case of $k=2$, one has
\begin{align*}
\matder^{2}((H\cdot\nabla)(\curl H)) 
={}&\partial_i\matder^{2}\curl H H^i+[\matder^{2},\partial_i]\curl H H^i+\nabla^2 H\star \matder^{2}H\\
&+\matder\nabla^2 H \star \nabla v\star H 
\eqqcolon  I_1+I_2+I_3+I_4.
\end{align*}
We denote $I_1
=(\nabla\curl \matder (H\cdot\nabla v))\star H+\nabla ([\matder^{2},\nabla] H)  \star H\eqqcolon I_{11}+I_{12}$. 
By Lemma \ref{l:curl Dt}, it holds $\curl (H\cdot\nabla \matder v)=\nabla \matder v\star \nabla H+\nabla^2\curl H\star H\star H+\nabla^2 H\star \nabla H\star H$, and using Lemma \ref{l:formu 1}, it follows that
\begin{align*}
I_{11}={}&\nabla(\curl (\matder H\cdot\nabla v))\star H+\nabla(\curl  (H\cdot\matder\nabla v))\star H \\ 
={}&\nabla^3\curl H\star H\star H+\nabla^2 \matder v\star \nabla H+\nabla \matder v\star \nabla^2 H+\sum_{|\alpha|,m\le 4,\alpha_i\le 3,F_j=v,H}\nabla^{\alpha_1}F_1\star\cdots\star\nabla^{\alpha_m}F_m.
\end{align*}
Applying Lemma \ref{l:formu 3}, we have $I_{12}=\nabla^2\matder v\star\nabla H\star H+\nabla\matder v\star\nabla^2H\star H+\nabla^2 H\star\nabla v\star\nabla v\star H+\nabla^2 v\star\nabla H\star\nabla v\star H+\nabla^2 H\star\nabla v\star H+\nabla^2 v\star\nabla H\star H,$ and $
I_2=\nabla\matder v\star\nabla^2H\star H+\nabla^3 v\star\nabla v\star H\star H+\nabla^2 H\star\nabla v\star\nabla v\star H+ \nabla^2 v\star\nabla H\star\nabla v\star H+\nabla^2 H\star\nabla v\star H$. 

To control the last two terms, \eqref{e:Dt^k H} implies that $I_3=\nabla\matder v\star\nabla^2H\star H+\nabla^2H\star\nabla v\star\nabla v\star H+\nabla^2H\star\nabla v\star H$, and Lemma \ref{l:formu 3} together with \eqref{e:mhd} yields $I_4=\nabla^3 v\star\nabla v\star H\star H+\nabla^2 H\star \nabla v\star\nabla v\star H+\nabla^2 v\star \nabla H\star\nabla v\star H$. 
We arrive at the following 
\begin{align*}
\matder^2((H\cdot\nabla)\curl H) 
={}&\nabla^3\curl H\star H\star H +\sum_{|\alpha|,m\le 4,\alpha_i\le 3,F_j=v,H}\nabla^{\alpha_1}F_1\star\cdots\star\nabla^{\alpha_m}F_m\\ 
&+\sum_{\substack{|\alpha|+|\beta|\le 4, \alpha_i+\beta_i\le 3\\\beta_i\le 1,m\le 3,F_j=v,H}}\nabla^{\alpha_1}\matder^{\beta_1}v\star\nabla^{\alpha_2} F_{2}\star\cdots\star\nabla^{\alpha_m} F_m
\eqqcolon J_1+J_2+J_3.
\end{align*}
As for $k=3$, to calculate $\matder J_1$, we only focus on the most difficult term. Actually, it holds $\matder\nabla^3\curl  H=\nabla^4\curl v\star H+\sum_{|\alpha|\le 5,\alpha_i\le 4}\nabla^{\alpha_1}H\star\nabla^{\alpha_2}v$, 
from Lemmas \ref{l:formu 3} and \ref{l:curl Dt}.
With the help of Lemma \ref{l:formu 3}, $\matder J_2$ and $\matder J_3$ can be treated in the same fashion. Therefore, we obtain
\begin{align*}
\matder^3((H\cdot\nabla)\curl H)
={}&\nabla^4\curl v\star H\star H\star H +\sum_{|\alpha|,m\le 5,\alpha_i\le 4,F_j=v,H}\nabla^{\alpha_1}F_1\star\cdots\star\nabla^{\alpha_m}F_m\\ 
&+\sum_{\substack{|\alpha|+|\beta|\le 5, \beta_i\le 2,\alpha_i+\beta_i\le 4,m\le 4,F_j=v,H}}\nabla^{\alpha_1}\matder^{\beta_1}v\star \nabla^{\alpha_2}\matder^{\beta_2}v\star\nabla^{\alpha_3} F_{3}\star\cdots\star\nabla^{\alpha_m} F_m .
\end{align*}
The other cases can be shown in the same way. 
\end{proof}

From now on,  we denote $R^0_{\nabla^2 H, H}
\coloneqq (H\cdot \nabla )\curl H$, and $R^k_{\nabla^2 H, H}
\coloneqq \matder^k((H\cdot \nabla )\curl H)$ for $ k\ge 1$.
We proceed to introduce another two types of error terms.  
The first one is written in the form
\begin{equation}\label{e:def error 1}
R_{\Rdiv}^0=\nabla v\star\nabla v,\ R_{\Rdiv}^l=\sum_{2\le m\le l+1}\sum_{|\beta| \leq l+2-m} 
\nabla \matder^{\beta_1} v \star \cdots \star \nabla \matder^{\beta_{m-1}} v \star \nabla \matder^{\beta_{m}} v, 
\end{equation} 
for any $l\ge 1$. 
The second error term is denoted by
\begin{equation}\label{e:def error 2}
\begin{aligned}
R_{\Rbulk}^0&=\nabla v\star  \matder v+\nabla v\star\nabla v\star v,\\
R_{\Rbulk}^l&=\sum_{2\le m\le l+1,|\beta| \leq l,|\alpha|\le 1} a_{\alpha,\beta}(\nabla v) \nabla \matder^{\beta_1} v \star \cdots \star \nabla \matder^{\beta_{m-1}} v \star \nabla^{\alpha_1} \matder^{\alpha_2+\beta_{m}} v, 
\end{aligned}
\end{equation}	
where $l\ge 1$ and $a_{\alpha,\beta}(\nabla v)$ denotes the finite $\star$ product as before.

\begin{lemma}\label{l:[Dt^l,nab]p}
For $l\in\Nz$, we have $[\matder^{l+1},\nabla ]p= \sum_{i \leq l} \nabla  \matder^{i} v  \star  \nabla H\star H+R_{\Rbulk}^l+R^l_{\nabla H,H}$. 
\end{lemma}
\begin{proof}
We prove this claim by induction. The case of $l=0$ follows directly. As for $l\ge 1$, by Lemmas \ref{l:formu 2} and \ref{l:formu 3},   
\begin{equation*}
[\matder^{l+1}, \nabla ]p=\matder([\matder^{l},\nabla ]p)+[\matder,\nabla ]\matder^{l}p=\matder([\matder^{l},\nabla ]p)-(\nabla v)^{\top}\nabla \matder^{l}p,
\end{equation*} 
where $-\nabla \matder^{l}p=[\matder^l,\nabla]p+\matder^{l} (\matder v -H\cdot \nabla H)=[\matder^l,\nabla]p+\matder^{l+1} v-\matder^{l}(H\cdot \nabla H)$. 
A direct computation shows that $\matder R^{l-1}_{\Rbulk}=R^l_{\Rbulk}$ and $\matder R^{l-1}_{\nabla H,H}=R^l_{\nabla H,H}$. These, combined with $[\matder^{l},\nabla]p=\sum_{i\leq l-1}\nabla\matder^{i} v\star\nabla H\star H+R_{\Rbulk}^{l-1}+R^{l-1}_{\nabla H,H}$ (also obtained by induction), yield that
\begin{equation*}
[\matder^{l+1}, \nabla ]p=\matder(\sum_{ i \leq l-1}  \nabla  \matder^{i} v  \star  \nabla H\star H)+R^l_{\Rbulk}+R^l_{\nabla H,H}=\sum_{ i \leq l} \nabla  \matder^{i} v  \star  \nabla H\star H+R_{\Rbulk}^l+R^l_{\nabla H,H},
\end{equation*} 
where in the last step,  the lower-order terms have been absorbed into the terms $R^l_{\Rbulk}$ and $R^l_{\nabla H,H}$.  
\end{proof} 
\begin{lemma}\label{l:formu 5} Let $l\in \N $. We have 
\begin{align*}
\matder \nabla^l \curl v={}&
(H\cdot \nabla) \nabla^{l} \curl H+\nabla v  \star \nabla^{l} \curl v+\nabla^{l+1}v\star \curl v\\
&+\sum_{|\beta|=l }\nabla^{1+\beta_1}  H  \star \nabla^{\beta_{2}} \curl H +\sum_{|\alpha| \leq l-1,\alpha_2\le l-2 }\nabla^{1+\alpha_1}  v  \star \nabla^{1+\alpha_{2}} \curl v,\\
\matder\nabla^l \curl H={}&(H\cdot\nabla)\nabla^l(\curl v)+\nabla  v  \star \nabla^{l} \curl H\\
&+\sum_{|\beta| =l } 
\nabla^{1+\beta_1}  v  \star \nabla^{1+\beta_{2}}H +\sum_{|\alpha| \leq l-1,\alpha_2\le l-2 } 
\nabla^{1+\alpha_1}  v  \star \nabla^{1+\alpha_{2}} \curl H.
\end{align*}
Moreover, we can also write $\dive\matder^l v=R^{l-1}_{\Rdiv}, \curl\matder^l v=R^{l-1}_{\Rdiv}+   R^{l-1}_{\nabla H,\nabla H}+	R^{l-1}_{\nabla^2 H, H}$, and   $\dive\matder^{l+1} v=\dive\dive(v \otimes \matder^l v)+\dive R_{\Rbulk}^{l-1}$. 
\end{lemma}
\begin{proof}
The first two claims are  immediate consequences of Lemmas  \ref{l:formu 3} and \ref{l:curl Dt}.  
Regarding $\curl\matder^l v$ and $\dive\matder^l v$ for $l\ge 2$.  Noting that $(\matder^l \nabla u)^\top=\matder^l [ (\nabla u)^\top ]$ and applying Lemmas \ref{l:formu 3} and \ref{l:nab^2H,H}, together with Lemma \ref{l:curl Dt}, we obtain 
\begin{align*}
\curl \matder^l v 
={}&[\nabla, \matder^{l-1}](\matder v)-([\nabla,  \matder^{l-1}](\matder v))^\top+\matder^{l-1}\curl\matder v \\
={}&\sum_{2\le m\le l}\sum_{|\beta| \leq l-m} 
\nabla \matder^{\beta_1} v \star \cdots \star \nabla \matder^{\beta_{m-1}} v \star \nabla \matder^{\beta_{m}+1} v+\matder^{l-1}(\nabla H\star \nabla H)\\
&+\matder^{l-1}((H\cdot\nabla)(\curl H)) 
=R^{l-1}_{\Rdiv} + 	R^{l-1}_{\nabla H,\nabla H}+	R^{l-1}_{\nabla^2 H, H}.
\end{align*}
Similarly, one has $\dive\matder^l v=R^{l-1}_{\Rdiv}$ thanks to $\dive v=0$. 
For the last statement, we apply $[\matder,\dive]F=-\dive(\nabla vF)$ and $\dive \dive (v\otimes\matder^l v)=\dive(\nabla \matder^l vv),l\ge 1$ (both can be easily computed).
Then, we have 
\begin{equation*}
\dive \matder^{2}v=\matder\dive (\nabla v v)+\dive(\nabla v\matder v)=\dive\matder(\nabla v v)-\dive(\nabla v\nabla v v)+\dive R^0_{\Rbulk},
\end{equation*}
and therefore,
\begin{equation*}
\dive \matder^{2}v 
=\dive(\nabla \matder v v)+\dive([\matder,\nabla] v v)-\dive(\nabla v\nabla v v)+\dive R^0_{\Rbulk}=\dive \dive (v\otimes\matder v)+\dive  R^0_{\Rbulk}.
\end{equation*} 

For $l\ge 2$, we argue by induction, i.e., $\dive \matder^{l+1}v=\matder\dive\matder^l v-[\matder,\dive]\matder^l v=\matder\dive(\nabla\matder^{l-1} vv)+  \matder\dive R^{l-2}_{\Rbulk}+\dive(\nabla v\matder^l v)$. 
The proof is complete since $\matder\dive R^{l-2}_{\Rbulk}=\dive R^{l-1}_{\Rbulk}, \dive(\nabla v\matder^l v)=\dive R^{l-1}_{\Rbulk}$ (direct calculations), and 
\begin{align*}
\matder\dive(\nabla\matder^{l-1} vv)
={}&\dive\matder(\nabla\matder^{l-1} vv)+[\matder,\dive](\nabla\matder^{l-1} vv)\\ 
={}&\dive(\nabla v\star \nabla \matder^{l-1}v\star v)+\dive(\nabla\matder^{l} vv+\matder v\star \nabla\matder^{l-1} v +v\star\nabla v\star \nabla\matder^{l-1} v)\\ 
={}&\dive\dive(v\otimes \matder^{l} v)+\dive R^{l-1}_{\Rbulk}.
\end{align*} 
\end{proof}

\begin{lemma}\label{l:lapDtp} 
Let $l\ge 1$. We have 
\begin{align*}
-\laplace \matder p
={}&\dive\dive(v\otimes\matder v)+\dive( R^0_{\Rbulk}+\nabla v\star H\star \nabla H+H\cdot \nabla (H\cdot\nabla v))\\
={}&-\dive\dive(v\otimes\nabla p)+\dive R^0_{\Rbulk}+\nabla^2 v\star \nabla H\star H+\nabla^2 H\star \nabla v\star H\\
&+\nabla^2 H\star\nabla H\star v+\nabla v\star \nabla H\star \nabla H\\
-\laplace \matder^{l+1} p
={}&\dive\dive(v \otimes \matder^{l+1} v )-\dive R^{l+1}_{\nabla^2H, H}+\dive(\sum_{i\leq l}\nabla\matder^{i} v  \star\nabla H\star H+R_{\Rbulk}^l+R^l_{\nabla H,H}).
\end{align*}
\end{lemma}
\begin{proof} 
From the divergence-free condition, Lemmas \ref{l:formu 5} and \ref{l:[Dt^l,nab]p}, the first claim follows. 
The second claim follows by applying Lemma \ref{l:[Dt^l,nab]p} that
\begin{align*}
-\laplace \matder^{l+1}p
={}&-\dive \matder^{l+1} \nabla p+\dive [\matder^{l+1}, \nabla ] p\\
={}&\dive \matder^{l+2} v-\dive \matder^{l+1} (H\cdot\nabla H)+\dive R_{\Rbulk}^l+\dive( \sum_{i \leq l} \nabla  \matder^{i} v  \star  \nabla H\star H+R^l_{\nabla H,H}). 
\end{align*}	
\end{proof}

From $p= \meancurv$ and the identities (e.g., \cite[Section 3.1]{Shatah2008a})
\begin{equation}\label{e:geo fact}
\matder \meancurv =-\surflaplace v_n-|B|^2v_n+\tangrad \meancurv\cdot v,\quad \surflaplace \nu =-|B|^2\nu+\tangrad \meancurv,
\end{equation}
it holds on the free-boundary $\parOmega_t$ that 
\begin{equation}\label{e:Dtp}
\matder p=-\surflaplace v\cdot \nu-2 B:\tangrad v=-\surflaplace v_n-|B|^2v_n+\tangrad p\cdot v.
\end{equation} 

Finally, we introduce the error term $R^l_p$ as described in \cite{Julin2024}. We define
\begin{align*}
R^1_p={}&-|B|^2\matder v\cdot\nu +\tangrad p\cdot\matder v+a_1(\nu,\nabla v)\star \nabla^2 v+a_2(\nu,\nabla v)\star B,\\
R^2_p={}&- |B|^2\matder^2 v\cdot\nu+\tangrad p\cdot\matder^2v +a_3(\nu,\nabla v)\star\nabla^2\matder v+a_4(\nu,\nabla v)\star\nabla\matder v\star \nabla^2 v\\
&+a_5(\nu,\nabla v)\star\nabla\matder v\star B+a_6(\nu,\nabla v)\star\nabla^2 v+a_7(\nu,\nabla v)\star B,\\
R^3_p ={}&- |B|^2\matder^3 v\cdot\nu+\tangrad p\cdot\matder^3v+a_8(\nu,\nabla v)\star\nabla^2\matder^2v+a_{9}(\nu,\nabla v)\star\nabla\matder^2v\star \nabla^2v\\
&+a_{10}(\nu,\nabla v)\star\nabla\matder^2v\star B+a_{11}(\nu,\nabla v)\star\nabla^2\matder v\star \nabla\matder v+a_{12}(\nu,\nabla v)\star\nabla^2\matder v\star B\\
&+a_{13}(\nu,\nabla v)\star\nabla \matder v\star \nabla\matder v\star \nabla^2 v+a_{14}(\nu,\nabla v)\star\nabla \matder v\star \nabla\matder v\star B+\operatorname{L.O.T.},\\
R^l_p={}&-|B|^2\matder^{l} v\cdot\nu +\tangrad p\cdot\matder^l v+\sum_{|\alpha|\le 1,|\beta|\le l-1}a_{\alpha,\beta}(\nu,B)\nabla^{1+\alpha_1}\matder^{\beta_1}v\star\cdots\star\nabla^{1+\alpha_{l+1}}\matder^{\beta_{l+1}}v,
\end{align*} 
where $l\ge 4$, $a_i(\nu,\nabla v)$ and $a_{\alpha,\beta}(\nu,B)$  denote the finite $\star$ product.  
\begin{lemma}\label{l:Dt^lp}
On the free-boundary $\parOmega_t$, we have $\matder^{l+1} p= -\surflaplace (\matder^l v\cdot\nu)+R^l_p$ for $l\in\N$. 
\end{lemma}
\begin{proof}
For $l=1$, we differentiate \eqref{e:Dtp} to obtain $\matder^2p=-\matder\surflaplace v\cdot \nu  -\surflaplace v\cdot \matder\nu-2\matder B:\tangrad v-2 B:\matder\tangrad v$. Recalling the formulas for $[\matder,\surflaplace],\matder \nu$ and $\matder B$ in Lemma \ref{l:formu 1}, it holds $\matder^2p={}-\surflaplace \matder v\cdot \nu-2 B:\tangrad \matder v+ a_1(\nu,\nabla v)\star \nabla^2 v+a_2(\nu,\nabla v)\star B$. 

For $l=2$, we differentiate $\matder^2 p$ and calculate $[\matder,\surflaplace]\matder v =\tangrad^2\matder v\star\nabla v-\tangrad \matder v\cdot \surflaplace v+B_\parOmega\star \nabla v\star \tangrad \matder v,
\matder B=a_1(\nu,\nabla v)\star B+a_2(\nu,\nabla v)\star \nabla^2 v,
\matder\tangrad \matder v=\tangrad \matder^2 v+\tangrad v\star \tangrad \matder v,
\matder a(\nu,\nabla v)=b(\nu,\nabla v)\star \nabla\matder v,  
\matder \nabla^2 v=\nabla^2 v\star \nabla v+\nabla^2 \matder v$ to obtain $\matder^3p=-\surflaplace \matder^2 v\cdot \nu-2 B:\tangrad \matder^2 v+a_3(\nu,\nabla v)\star\nabla^2\matder v+a_4(\nu,\nabla v)\star\nabla\matder v\star \nabla^2 v+a_5(\nu,\nabla v)\star\nabla\matder v\star B+a_6(\nu,\nabla v)\star\nabla^2 v+a_7(\nu,\nabla v)\star B$. 
We can obtain the case of $l=3$ in the same way and the remaining proof is similar to \cite[Lemma 4.7]{Julin2024}. 
\end{proof}

\section{Time derivatives of the energy functionals}\label{s:d/dt}

In this section, we compute the time derivative of the energy functional $\energy_l(t)$ by applying Reynolds transport theorem and \eqref{e:reynold 2}. The main result in this section is the following proposition.

\begin{proposition}\label{p:d/dt} 
Assume that the a priori assumptions \eqref{e:a priori} hold for some $T>0$. Then, we have
\begin{align*}
\frac{d}{dt}\barenergy(t)
\le{}& C\sum_{l=1}^{3}\left( \|R^l_{\Rdiv}\|_{H^{1/2}(\Omega_t)}^2+\|R^l_{\Rbulk}\|_{L^2(\Omega_t)}^2+\|R^l_{\nabla H,H}\|_{L^2(\Omega_t)}^2+\|R^l_p\|_{H^{1/2}(\parOmega_t)}^2\right)\\
&+C \left( 1+\|\nabla^2 p\|_{L^2(\Omega_t)}^2\right) \barE(t), 
\end{align*}
where the constant $C$ depends on $T,\mathcal{N}_T$, and $\mathcal{M}_T$.

Moreover, we further assume that $\sup_{0\le t< T}E_{l-1}(t)\le C$ for $l\ge 4$. Then, it holds
\begin{equation*} 
\frac{d}{dt}\energy_l(t)
\le  C\left( E_l(t)+\|R^l_{\Rdiv}\|_{H^{1/2}(\Omega_t)}^2+\|R^l_{\Rbulk}\|_{L^2(\Omega_t)}^2+\|R^l_{\nabla H,H}\|_{L^2(\Omega_t)}^2+\|R^l_p\|_{H^{1/2}(\parOmega_t)}^2\right) , 
\end{equation*}
for $l\ge 4$, where the constant $C$ depends on $T,\mathcal{N}_T,\mathcal{M}_T$, and $\sup_{0\le t< T}E_{l-1}(t)$.
\end{proposition}

Denote $I^l_1(t)=\frac{1}{2}\|\matder^{l+1}v\|_{L^2(\Omega_t)}^2,I^l_2(t)=\frac{1}{2}\|\matder^{l+1}H\|_{L^2(\Omega_t)}^2,I^l_3(t)=\frac{1}{2}\|\tangrad(\matder^{l}v\cdot \nu)\|_{L^2(\parOmega_t)}^2,I^l_4(t)=\frac{1}{2}\|\nabla^{\lfloor \frac{3l+1}{2}\rfloor} \curl v\|_{L^2(\Omega_t)}^2$ and $I^l_5(t)=\frac{1}{2}\|\nabla^{\lfloor \frac{3l+1}{2}\rfloor} \curl H\|_{L^2(\Omega_t)}^2$. 
We will apply Reynolds transport theorem and  \eqref{e:reynold 2} several times and we start with $I^l_1(t)$. From \eqref{e:mhd} and the divergence theorem, 
\begin{align*}
\frac{d}{dt}I^l_1(t) 
={}&-\int_{\Omega_t}\matder^{l+1}\nabla p\cdot \matder^{l+1}vdx+\int_{\Omega_t}\matder^{l+1}(H\cdot \nabla H)\cdot \matder^{l+1}vdx\\
={}&-\int_{\Omega_t}\nabla\matder^{l+1} p\cdot \matder^{l+1}vdx-\int_{\Omega_t}[\matder^{l+1},\nabla] p\cdot \matder^{l+1}vdx+\int_{\Omega_t}\matder^{l+1}(H^j\partial_j  H_i) \matder^{l+1}v^idx\\
={}&-\int_{\Omega_t}\dive (\matder^{l+1} p \matder^{l+1}v)dx+\int_{\Omega_t} \matder^{l+1} p\dive \matder^{l+1}vdx\\
&-\int_{\Omega_t}[\matder^{l+1},\nabla] p\cdot \matder^{l+1}vdx+\int_{\Omega_t}\matder^{l+1}(H^j\partial_j  H_i) \matder^{l+1}v^idx\\
\le{}& \underbrace{\int_{\Omega_t}H^j\partial_j(\matder^{l+1}  H_i) \matder^{l+1}v^idx}_{\eqqcolon J^l_{1}(t)}\underbrace{-\int_{\parOmega_t} \matder^{l+1} p (\matder^{l+1}v \cdot \nu) dS}_{\eqqcolon K^l_{1}(t)}+\|\matder^{l+1}v\|_{L^2(\Omega_t)}^2\\
&+\underbrace{\int_{\Omega_t} \matder^{l+1} p\dive \matder^{l+1}vdx}_{\eqqcolon I^l_{11}(t)}+\underbrace{\|[\matder^{l+1},\nabla] p\|_{L^2(\Omega_t)}^2}_{\eqqcolon I^l_{12}(t)}\\ &+\underbrace{\sum_{k=0}^{l}\int_{\Omega_t}\matder^k  H^j[\matder^{l+1-k},\partial_j]H_i \matder^{l+1}v^idx}_{\eqqcolon I^l_{13}(t)}+\underbrace{\sum_{k=1}^{l+1}\int_{\Omega_t}\matder^k  H^j\partial_j\matder^{l+1-k}H_i \matder^{l+1}v^idx}_{\eqqcolon I^l_{14}(t)},
\end{align*}
where we have used the fact that 
\begin{align*}
&\matder^{l+1}(H^j\partial_j  H_i) \matder^{l+1}v^i\\
={}&H^j\partial_j(\matder^{l+1}  H_i) \matder^{l+1}v^i+\sum_{k=0}^{l}\matder^k  H^j[\matder^{l+1-k},\partial_j]H_i \matder^{l+1}v^i+\sum_{k=1}^{l+1}\matder^k  H^j\partial_j\matder^{l+1-k}H_i \matder^{l+1}v^i.
\end{align*}
Similarly, for the magnetic field, it follows that
\begin{align*}
\frac{d}{dt}I^l_2(t) 
={}&\underbrace{\int_{\Omega_t}H^j\partial_j(\matder^{l+1}v^i) \matder^{l+1}H_i dx}_{\eqqcolon J^l_{2}(t)}+\underbrace{\sum_{k=0}^{l}\int_{\Omega_t}\matder^k  H^j[\matder^{l+1-k},\partial_j]v^i \matder^{l+1}H_idx}_{\eqqcolon I^l_{21}(t)}\\
&
+\underbrace{\sum_{k=1}^{l+1}\int_{\Omega_t}\matder^k  H^j\partial_j\matder^{l+1-k}v^i\matder^{l+1}H_idx}_{\eqqcolon I^l_{22}(t)}.
\end{align*}
Recalling the divergence-free condition and $H\cdot\nu=0$ on $\parOmega_t$, it is clear that $J^l_{1}(t)+J^l_{2}(t)=0$, and we obtain $\frac{d}{dt}( I^l_1(t)+I^l_2(t)) \le K^l_{1}(t)+\sum_{i=1}^4I^l_{1i}(t)+I^l_{21}(t)+I^l_{22}(t)+\|\matder^{l+1}v\|_{L^2(\Omega_t)}^2$. 

To control the third term,  we apply  Lemma \ref{l:formu 1} to deduce
\begin{align*}
\frac{d}{dt}I^l_3(t) 
={}&\int_{\parOmega_t}-(\tangrad v)^\top \tangrad(\matder^{l}v\cdot \nu)\cdot \tangrad(\matder^{l}v\cdot \nu)dS
+\frac{1}{2}\int_{\parOmega_t}|\tangrad(\matder^{l}v\cdot \nu)|^2\tandive vdS\\
&+\int_{\parOmega_t}\tangrad(\matder^{l+1}v\cdot \nu)\cdot \tangrad(\matder^{l}v\cdot \nu)dS+\int_{\parOmega_t}\tangrad(\matder^{l}v\cdot \matder\nu)\cdot \tangrad(\matder^{l}v\cdot \nu)dS\\
\le{}&\underbrace{-\int_{\parOmega_t}(\matder^{l+1}v\cdot \nu)\cdot \surflaplace(\matder^{l}v\cdot \nu)dS}_{\eqqcolon K^l_3(t)}
+\underbrace{\|\tangrad(\matder^{l}v\cdot\matder\nu)\|_{L^2(\parOmega_t)}^2}_{\eqqcolon I^l_{31}(t)}\\ 
&+C(\|\tangrad v\|_{L^\infty(\parOmega_t)}+1)
\|\tangrad(\matder^{l}v\cdot \nu)\|_{L^2(\parOmega_t)}^2.
\end{align*}
Finally, to compute the last two terms involving the $\curl$, we denote $\mu_l\coloneqq \lfloor  (3l+1)/2\rfloor$. We then utilize the divergence-free condition and the fact that $H\cdot\nu=0$ on $\parOmega_t$ to obtain $\int_{\Omega_t}\sum_{|\alpha|=l}(H\cdot \nabla) (\nabla^{\alpha} \curl H:\nabla^{\alpha} \curl v+\nabla^{\alpha} \curl v:\nabla^{\alpha} \curl H)dx=0$. 
Therefore, from Lemma \ref{l:formu 5}, it follows that
\begin{align*}
&\frac{d}{dt}I^l_4(t)-	\int_{\Omega_t}\sum_{|\alpha|=l}(H\cdot \nabla)  \nabla^{\alpha} \curl H:\nabla^{\alpha} \curl vdx\\
\le{}& C(\|\nabla v\|_{L^\infty(\Omega_t)}+1)\|\nabla^{\mu_l+1}v\|_{L^2(\Omega_t)}^2+\|\nabla H  \|_{L^\infty(\Omega_t)}^2\|\curl H\|_{H^{\mu_l}(\Omega_t)}^2\\
&+\|\curl H  \|_{L^\infty(\Omega_t)}^2\|\nabla H\|_{H^{\mu_l}(\Omega_t)}^2+\|\nabla v \|_{L^\infty(\Omega_t)}^2\| \nabla v\|_{H^{\mu_l}(\Omega_t)}^2,\\
&\frac{d}{dt}I^l_5(t)-	\int_{\Omega_t}\sum_{|\alpha|=l}(H\cdot \nabla)  \nabla^{\alpha} \curl v:\nabla^{\alpha} \curl Hdx\\ 
\le{}& C(\|\nabla v\|_{L^\infty(\Omega_t)}+1)\|\nabla^{\mu_l} \curl H\|_{L^2(\Omega_t)}^2+\|\nabla v\|_{H^{\mu_l}(\Omega_t)}^2\|\nabla H\|_{L^{\infty}(\Omega_t)}^2\\
&+\|\nabla H\|_{H^{\mu_l}(\Omega_t)}^2\|\nabla v\|_{L^{\infty}(\Omega_t)}^2.
\end{align*} 
\begin{proof}[Proof of Proposition \ref{p:d/dt}]
By \eqref{e:a priori}, one has $\|\tangrad v\|_{L^\infty(\parOmega_t)}\le C\|\nabla v\|_{L^\infty(\Omega_t)}\le C$.  This, combined with the above calculations and applying Lemma \ref{l:Dt^lp}, $\|\nabla H\|_{L^\infty(\Omega_t)}  \le C$ by \eqref{e:a priori},  together with  the definition of $\barE(t)$,  we obtain $K^l_1(t)+K^l_3(t)=-\int_{\parOmega_t}R^{l}_p (\matder^{l+1}v \cdot \nu)dS$, and
\begin{align*}
\frac{d}{dt}\barenergy(t)
&\le C\barE(t)+C\sum_{l=1}^{3}\left(-\int_{\parOmega_t}R^{l}_p (\matder^{l+1}v \cdot \nu) dS+\sum_{i=1}^{4}I^l_{1i}(t)+I^l_{31}(t)+I^l_{21}(t)+I^l_{22}(t)\right),\\
\frac{d}{dt}\energy_l(t)&\le CE_l(t)+C\left(-\int_{\parOmega_t}R^{l}_p (\matder^{l+1}v \cdot \nu) dS+\sum_{i=1}^{4}I^l_{1i}(t)+I^l_{31}(t) +I^l_{21}(t)+I^l_{22}(t)\right),\quad l\ge 4.
\end{align*}
We divide the remaining proof into six steps. 

\textbf{Step 1.} We control $I^l_{14}(t)$ and $I^l_{22}(t)$. We omit the case of $l=1$, and assume $F=v,G=H$ or $F=H,G=v$ respectively.  
In the case of $l=2$, from the fact that 
\begin{align}
\|\nabla \matder H\|_{L^2(\Omega_t)}^2&\le  \|\nabla (H\cdot\nabla v)\|_{L^2(\Omega_t)}^2\le C,\label{e:nabDtH}\\
\|\nabla \matder v\|_{L^2(\Omega_t)}^2&\le  \|\nabla (H\cdot\nabla H)\|_{L^2(\Omega_t)}^2+\|\nabla^2p\|_{L^2(\Omega_t)}^2 \le C(1+\|\nabla^2 p\|_{L^2(\Omega_t)}^2),\label{e:nabDtv}
\end{align}
it follows that
\begin{align*}
&\sum_{k=1}^{3}\int_{\Omega_t}\matder^k  H^j\partial_j\matder^{3-k}F_i \matder^{3}G^idx\\ 
\le {}& C(E_2(t)+\|H\cdot\nabla v\|_{L^2(\Omega_t)}^2\|\matder^2F\|_{H^{3}(\Omega_t)}^2+\|\matder^2 H\|_{H^{2}(\Omega_t)}^2\|\nabla\matder  F\|_{L^{2}(\Omega_t)}^2\\
& +\|\matder^{3}  H \|_{L^2(\Omega_t)}^2\|\nabla F\|_{L^\infty(\Omega_t)}^2)\le  C(1+\|\nabla^2 p\|_{L^2(\Omega_t)}^2)\barE(t).
\end{align*}
As for $l=3$, again by \eqref{e:nabDtH} and \eqref{e:nabDtv}, we obtain
\begin{align*}
&\sum_{k=1}^{4}\int_{\Omega_t}\matder^k  H^j\partial_j\matder^{4-k}F_i \matder^{4}G^idx\\ 
\le{}& C(E_3(t)+\|H\cdot\nabla v\|_{L^6(\Omega_t)}^2 \| \matder^{3} F\|_{H^{3/2}(\Omega_t)}^2+\|\matder^2 H\|_{L^2(\Omega_t)}^2\|\matder^{2} F\|_{H^{3}(\Omega_t)}^2 \\
&\quad +\|\matder^2(H\cdot\nabla v)\|_{L^\infty(\Omega_t)}^2\|\nabla\matder F\|_{L^2(\Omega_t)}^2+\|\matder^{4}  H\|_{L^2(\Omega_t)}^2\|\nabla F\|_{L^\infty(\Omega_t)}^2)
\\
\le{}& C(1+\|\nabla^2 p\|_{L^{2}(\Omega_t)}^2)\barE(t),
\end{align*}
where we have used
\begin{align}
\|\matder^2 H\|_{L^2(\Omega_t)}^2\le{}& C(\|\matder H\star  \nabla v\|_{L^2(\Omega_t)}^2+\| H\star\matder\nabla v\|_{L^2(\Omega_t)}^2) \le C(1+\|\nabla^2 p\|_{L^2(\Omega_t)}^2),\label{e:Dt^2H}\\
\|\matder^3H\|_{L^\infty(\Omega_t)}^2 
\le{}&\|\matder^2 H\star \nabla v\|_{L^\infty(\Omega_t)}^2+\|\matder H\star\matder\nabla v\|_{L^\infty(\Omega_t)}^2+\|H\star\matder^2\nabla v\|_{L^\infty(\Omega_t)}^2\nonumber\\
\le{}&C(\|\matder^2 H\|_{L^\infty(\Omega_t)}^2+\|[\matder,\nabla] v\|_{L^\infty(\Omega_t)}^2+\|\nabla\matder v\|_{L^\infty(\Omega_t)}^2\nonumber\\
&\quad +\|[\matder^2,\nabla] v\|_{L^\infty(\Omega_t)}^2+\|\nabla\matder^2 v\|_{L^\infty(\Omega_t)}^2)\le C\barE(t),\nonumber
\end{align}
by utilizing \eqref{e:a priori}, Lemmas \ref{l:formu 1} and \ref{l:formu 3}. Additionally, one order material derivative has been substituted with the spatial derivative of the velocity field. As $l\ge 4 $, we use the hypotheses  $E_{l-1}(t)\le C$ to obtain
\begin{align*}
& \sum_{k=1}^{l+1}\int_{\Omega_t}\matder^k  H^j\partial_j\matder^{l+1-k}F_i \matder^{l+1}G^idx\\ 
\le{}& C(\sum_{k=2}^{l}\|\matder^k H\|_{H^1(\Omega_t)}^2 \|\matder^{l+1-k} F\|_{H^{3/2}(\Omega_t)}^2+\|\matder  H^j\partial_j\matder^{l}F\|_{L^2(\Omega_t)}^2+E_l(t))\\
\le{}& CE_l(t)E_{l-1}(t)+CE_l(t)+C\|\matder  H\|_{L^6(\Omega_t)}^2\|\nabla\matder^{l}F\|_{L^3(\Omega_t)}^2\le CE_l(t).
\end{align*} 

\textbf{Step 2.} We control $I^l_{13}(t)$ and $I^l_{21}(t)$. As before, we assume $F=v,G=H$ or $F=H,G=v$.  
We only consider the case of $l\ge 3$. 
In fact, from 
$[\matder^{j}, \nabla]$ in Lemma \ref{l:formu 3}, \eqref{e:nabDtH}, \eqref{e:nabDtv} and \eqref{e:Dt^2H}, it holds 
\begin{align*}
&\sum_{k=0}^{3}\int_{\Omega_t}\matder^k  H^j[\matder^{4-k},\partial_j]F_i \matder^{4}G^idx\\
\le{}&C(E_3(t)+ \|\matder^3 H^j\partial_jv^k\partial_kF\|_{L^2(\Omega_t)}^2\\
&+\|\matder^2 H\star(\nabla v \star \nabla F+
\nabla \matder v \star \nabla F+
\nabla v \star \nabla \matder F 
+\nabla v\star\nabla v\star\nabla F)\|_{L^2(\Omega_t)}^2\\
&+\|\matder H\star
(\nabla \matder^2 v \star \nabla F+\nabla \matder v \star \nabla \matder  F+\nabla  v \star \nabla \matder^2 F\\
&\quad \quad \quad \quad \quad +\nabla \matder v\star\nabla v\star\nabla F+\nabla v\star \nabla v\star \nabla \matder F+\operatorname{L.O.T.})\|_{L^2(\Omega_t)}^2\\
&+\|H\star(\nabla \matder^3 v\star\nabla F+\nabla \matder^2 v\star\nabla\matder F+\nabla\matder v\star\nabla\matder^2F\\
&\quad\quad\quad\ \ \ +\nabla v\star\nabla\matder^3 F+\operatorname{L.O.T.}) \|_{L^2(\Omega_t)}^2) 
\le C(1+\|\nabla^2 p\|_{L^2(\Omega_t)}^2)\barE(t).
\end{align*}
For $l\ge 4$, from Lemma \ref{l:formu 3} and the assumption $E_{l-1}(t)\le C$, we deduce
\begin{align*}
&\sum_{k=0}^{l}\int_{\Omega_t}\matder^k  H^j[\matder^{l+1-k},\partial_j]F_i \matder^{l+1}G^idx\\ 
\le{}&C(E_l(t)+\|\matder^l H^j\partial_jv^k\partial_kF\|_{L^2(\Omega_t)}^2 
\\
&\quad +\sum_{k=0}^{l-1} \|\matder^k  H\star\sum_{2\le m\le l+2-k}\sum_{|\beta| \leq l+2-k-m}\nabla \matder^{\beta_1} v \star \cdots \star \nabla \matder^{\beta_{m-1}} v\\
&\quad \quad \quad \quad \quad \quad \quad \quad \quad \quad \quad \quad \quad \quad \quad \quad \quad \ \ \ \star \nabla \matder^{\beta_{m}}F \|_{L^2(\Omega_t)}^2)\\ 
\le{}&CE_{l-1}(t)E_l(t)+CE_l(t)\le CE_l(t).
\end{align*} 

\textbf{Step 3.} To estimate $\int_{\parOmega_t} R^l_p (\matder^{l+1}v \cdot \nu) dS$, we apply Lemma \ref{l:formu 5} and the normal trace theorem (e.g., \cite[Theorem 3.1]{Cheng2007}) 
to obtain $\|\matder^{l+1}v\cdot\nu\|_{H^{-1/2}(\parOmega_t)}\le C(	\|\matder^{l+1}v\|_{L^{2}(\Omega_t)}+\|\dive \matder^{l+1}v\|_{H^{-1}(\Omega_t)})$.
Therefore,  it follows that
\begin{align*}
|\int_{\parOmega_t} R^l_p (\matder^{l+1}v \cdot \nu) dS| 
&\le C(\barE(t)+\|R^l_{\Rdiv}\|_{L^{2}(\Omega_t)}^2+\|R^l_p\|_{H^{1/2}(\parOmega_t)}^2),\  l\le 3,\\ 
|\int_{\parOmega_t} R^l_p (\matder^{l+1}v \cdot \nu) dS|&\le C(E_l(t)+\|R^l_{\Rdiv}\|_{L^{2}(\Omega_t)}^2+\|R^l_p\|_{H^{1/2}(\parOmega_t)}^2),\  l\ge 4.
\end{align*} 

\textbf{Step 4.} We estimate $I^l_{31}(t)$. 
We only present estimates for $l\ge 3$, and the cases of $l\le 2$ are easier. Actually, by the a priori assumptions \eqref{e:a priori} and the trace theorem, one has
\begin{align*}
&\|\tangrad(\matder^3 v\cdot\matder\nu)\|_{L^2(\parOmega_t)}^2\\
\le{}& \|\tangrad \matder^3 v\star\matder\nu \|_{L^2(\parOmega_t)}^2+\|\matder^3 v\star \tangrad\matder\nu\|_{L^2(\parOmega_t)}^2\\
\le{}& C (\|\matder  \nu\|_{L^\infty(\parOmega_t)}^2\|\matder^3 v\|_{H^{3/2}(\Omega_t)}^2+\underbrace{\|\matder^3 v\star \tangrad^2v\star\nu\|_{L^2(\parOmega_t)}^2}_{\eqqcolon L^3_{31}(t)}+\|\matder^3 v\star \tangrad v\star \tangrad \nu\|_{L^2(\parOmega_t)}^2)\le C \barE(t).
\end{align*} 
Above, we have applied the Sobolev embedding, i.e., for $p^{-1}+q^{-1}=2^{-1},p=2\delta^{-1}$ with $\delta>0$ small enough, it holds $L^3_{31}(t) 
\le C\|\matder^3 v\|_{H^{1-\delta}(\parOmega_t)}^2\| \tangrad^2v\|_{H^{\delta}(\parOmega_t)}^2$,
and $\|\matder^3 v\|_{H^{1-\delta}(\parOmega_t)}^2\| \tangrad^2v\|_{H^{\delta}(\parOmega_t)}^2\le \|\matder^3 v\|_{H^{3/2-\delta}(\Omega_t)}^2\|v\|_{H^{5/2+\delta}(\Omega_t)}^2 \le C\barE(t)$, by using the trace theorem.  As for $l\ge 4$,  it follows that
\begin{align*}
\|\tangrad(\matder^l v\cdot\matder\nu)\|_{L^2(\parOmega_t)}^2 
\le{}& C (\|\matder  \nu\|_{L^\infty(\parOmega_t)}^2\|\matder^l v\|_{H^1(\parOmega_t)}^2+\|\matder  \nu\|_{W^{1,4}(\parOmega_t)}^2\|\matder^l v\|_{L^4(\parOmega_t)}^2)\\
\le{}& C(\|\matder^l v\|_{H^{3/2}(\Omega_t)}^2+E_{l-1}(t)\|\matder^l v\|_{H^{1}(\Omega_t)}^2)\le C E_l(t),
\end{align*}
where we have used $\matder\nu= \tangrad v\star \nu$ from Lemma \ref{l:formu 1} and $\|\nu\|_{H^{2+\delta}(\parOmega_t)}\le C$ by \eqref{e:a priori} together with \eqref{e:nu by h}. 

\textbf{Step 5.} For $I^l_{12}(t)$, we recall that it holds $[\matder^{l+1},\nabla ]p=\sum_{\beta_1 \leq l}  \nabla  \matder^{\beta_1} v  \star  \nabla H\star H+R_{\Rbulk}^l+R^l_{\nabla H,H}$ by Lemma \ref{l:[Dt^l,nab]p}. Clearly, we have $\|\sum_{\beta_1 \leq l}  \nabla  \matder^{\beta_1} v  \star  \nabla H\star H\|_{L^2(\Omega_t)}^2 \le C\barE(t)$  for $l\le 3$, and $\|\sum_{\beta_1 \leq l}  \nabla  \matder^{\beta_1} v  \star  \nabla H\star H\|_{L^2(\Omega_t)}^2 \le CE_l(t)$ as $l\ge 4$. 

\textbf{Step 6.} Finally, controlling  $I^l_{11}(t)$ is trickier. Let $u$ be a solution to 
\begin{equation*}
\begin{cases}
-\laplace u=\dive \matder^{l+1}v, & \text{ in } \Omega_t,\\
u=0, &  \text{ on } \parOmega_t,
\end{cases}
\end{equation*}
where $l\ge 1$. We first recall the elliptic estimates (see, e.g., {\cite[Proposition 3.8]{Julin2024}})
\begin{equation}\label{e:elliptic}
\|\partial_\nu u\|_{H^1(\parOmega_t)}+\|\nabla u\|_{H^{3/2}(\Omega_t)} \leq C\|\dive \matder^{l+1}v\|_{H^{1/2}(\Omega_t)}.
\end{equation}
Then, we integrate by parts to obtain $I^l_{11}(t)
=-\int_{\Omega_t}\laplace\matder^{l+1} p udx-\int_{\parOmega_t}\matder^{l+1} p\partial_\nu udS\eqqcolon I^l_{111}(t)+I^l_{112}(t)$. 
Again by integration by parts, Lemma \ref{l:lapDtp} and the divergence theorem, it follows that
\begin{align*}
I^l_{111}(t)
={}&\int_{\Omega_t}(v \otimes \matder^{l+1} v ):\nabla^2 udx-\int_{\Omega_t}( R_{\Rbulk}^l+R^l_{\nabla H,H}+\sum_{\beta_1 \leq l}  \nabla  \matder^{\beta_1} v  \star  \nabla H\star H)\cdot \nabla u dx\\
&-\int_{\Omega_t}\dive\matder^{l+1}(H\cdot\nabla H)  udx-\int_{\parOmega_t}v^i\matder^{l+1} v^j\partial_i u\nu_jdS\\
\le{}& C(\|u\|_{H^2(\Omega_t)}^2+E_l(t)+\|R_{\Rbulk}^l\|_{L^2(\Omega_t)}^2+\|R^l_{\nabla H,H}\|_{L^2(\Omega_t)}^2)\\
&+\underbrace{\int_{\Omega_t}\dive (v^i\matder^{l+1} v\partial_i u)dx}_{\eqqcolon L^l_{1111}(t)}\underbrace{-\int_{\Omega_t}\dive\matder^{l+1}(H\cdot\nabla H)  udx}_{\eqqcolon L^l_{1112}(t)}.
\end{align*}
We estimate the first term by using  Lemma \ref{l:formu 5}. Indeed, it holds
\begin{align*}
|L^l_{1111}(t)|
={}&|\int_{\Omega_t}\nabla v\star\matder^{l+1} v\star\nabla u+ v\star\dive\matder^{l+1} v\star\nabla u+v\star \matder^{l+1} v\star \nabla^2 udx|\\
\le{}&C(\|u\|_{H^2(\Omega_t)}^2+E_l(t)+\|R^l_{\Rdiv}\|_{L^2(\Omega_t)}^2).
\end{align*}

To control $L^l_{1112}(t)$, it is important to note that the integration by parts method used previously is not applicable. However, as indicated in Lemmas \ref{l:nab H, nab H} and \ref{l:divDt}, a one-order material derivative can be substituted for a one-order spatial derivative due to the divergence-free condition.  
In fact, we have from Lemma \ref{l:divDt} that  
$\dive\matder^{l+1}(H\cdot \nabla H)=\partial_i\partial_m\matder^{l}  v^j \partial_jH^mH^i+\nabla^3\matder^{l-1}v\star H\star H+\operatorname{L.O.T.}$,
and 
\begin{align*}
|L^l_{1112}(t)|\le {}&|\int_{\Omega_t}\partial_i\partial_m\matder^{l}  v^j \partial_jH^mH^iudx|+C\|u\|_{L^2(\Omega_t)}^2+C\|\nabla^3\matder^{l-1}v\star H\star H\|_{L^2(\Omega_t)}^2+M^l_{1112}(t)\\
\le {}& C\|u\|_{H^1(\Omega_t)}^2+CE_l(t)+M^l_{1112}(t),
\end{align*}
where we have used $H\cdot\nu=0$, and 
\begin{align*}
|\int_{\Omega_t}\partial_i\partial_m\matder^{l}  v^j \partial_jH^mH^iudx| 
={}&|\int_{\Omega_t}\partial_m\matder^{l}  v^j \partial_i\partial_jH^mH^iu+\partial_m\matder^{l}  v^j \partial_jH^mH^i\partial_i udx|\\ 
\le{}& CE_l(t)+C\|u\|_{H^1(\Omega_t)}^2,
\end{align*}
by integration by parts.
Also, $M^l_{1112}(t)$ contains lower-order terms (at most $\nabla^2\matder^{l-1}$) which can be controlled in the same fashion as before.  These, together with the fact $\|u\|_{H^2(\Omega_t)}^2\le \|\dive \matder^{l+1}v\|_{L^2(\Omega_t)}^2\le C\|R^l_{\Rdiv}\|_{L^2(\Omega_t)}^2$,  
it holds $|I^l_{111}(t)|\le C(\barE(t)+\|R^l_{\Rdiv}\|_{L^2(\Omega_t)}^2+\| R^l_{\Rbulk}\|_{L^2(\Omega_t)}^2+\|R^l_{\nabla H,H}\|_{L^2(\Omega_t)}^2)$ 
for $l\le 3$, and for $l\ge 4,|I^l_{111}(t)|\le C(E_l(t)+\|R^l_{\Rdiv}\|_{L^2(\Omega_t)}^2+\| R^l_{\Rbulk}\|_{L^2(\Omega_t)}^2+\|R^l_{\nabla H,H}\|_{L^2(\Omega_t)}^2)$.

We are left with $I^l_{112}(t)$. Applying Lemma \ref{l:Dt^lp} and  integration by parts, one has $\int_{\parOmega_t} \matder^{l+1} p \partial_\nu dS u =\int_{\parOmega_t} \tangrad(\matder^l v\cdot\nu)\cdot\tangrad\partial_\nu udS+\int_{\parOmega_t}R^l_p\partial_\nu udS$. 
Then, we use \eqref{e:elliptic} to deduce 
\begin{align*}
|I^l_{112}(t)|&\le C(\|\tangrad(\matder^l v\cdot\nu)\|_{L^2(\parOmega_t)}^2+\|\partial_\nu u\|_{H^1(\parOmega_t)}^2+\|R^l_p\|_{L^2(\parOmega_t)}^2)\\ 
&\le C(\barE(t)+\|R^l_{\Rdiv}\|_{H^{1/2}(\Omega_t)}^2+\|R^l_p\|_{L^2(\parOmega_t)}^2),\quad l\le 3.
\end{align*}
Similarly, $|I^l_{112}(t)|\le C(E_l(t)+\|R^l_{\Rdiv}\|_{H^{1/2}(\Omega_t)}^2+\|R^l_p\|_{L^2(\parOmega_t)}^2)$ for $l\ge 4$. This completes the proof. 
\end{proof}
\section{Estimates for the pressure}\label{s:p est}

In this section, we treat the pressure and will show that
\begin{equation}\label{e:p H^3}
\sup_{t\in[0,T]}\|p\|_{H^3(\Omega_t)}\le C,
\end{equation}
where the constant $C$ depends on the time $T>0$, the a priori assumptions $\mathcal{N}_T,\mathcal{M}_T$, and the initial data $\|v_0\|_{H^6(\Omega_0)},\|H_0\|_{H^6(\Omega_0)}$ and $\|\meancurv_{\parOmega_0}\|_{H^5(\parOmega_0)}$. For this purpose, we assume the a priori assumptions \eqref{e:a priori} for some $T>0$. As a result, it follows that $\sup_{0\le t< T}\|h\|_{H^{3+\delta}(\Gamma)}\le C$ and $\sup_{0\le t< T}\|B\|_{H^{1+\delta}(\parOmega_t)}  \le C$. In particular, we have $\|p\|_{H^{1+\delta}(\parOmega_t)}\le C$ and
\begin{equation}\label{e:p H^1 bou}
\int_0^T \|p\|_{H^1(\parOmega_t)}^2dt\le C\left( \mathcal{N}_T,\mathcal{M}_T\right) T.
\end{equation}
Recalling we define  $H^{1/2}(\parOmega_t)$ via the harmonic extension. From Lemma \ref{l:3.3 jul} and \eqref{e:har ext 2}, we obtain
\begin{align}
\|\partial_\nu p\|_{L^2(\parOmega_t)}^2&\le C( \|  \tangrad p\|_{L^2(\parOmega_t)}^2+ \|\nabla p\|_{L^2(\Omega_t)}^2+ \|\laplace p\|_{L^2(\Omega_t)}^2 )\nonumber\\
&\le C( \|  \tangrad p\|_{L^2(\parOmega_t)}^2+\|p\|_{H^{1/2}(\parOmega_t)}^2+\|\laplace p\|_{L^2(\Omega_t)}^2 )\nonumber\\
&\le  C( \|  p\|_{H^1(\parOmega_t)}^2+\|\laplace p\|_{L^2(\Omega_t)}^2 )\le C(\mathcal{N}_T,\mathcal{M}_T)(1+T)\label{e:par nu p}.
\end{align}

For higher-order derivatives, we have the following results.
\begin{proposition}\label{p:nab^3 bou}
Assume that $\parOmega_t$ is uniformly $H^{3+\delta}(\Gamma)$-regular for $\delta>0$ sufficiently small. For smooth function $f$, it holds
\begin{align}
\|\nabla^2 f\|_{L^2(\parOmega_t)}^2&\le C\left( \|\laplace f\|_{H^1(\Omega_t)}^2+\|f\|_{H^2(\parOmega_t)}^2\right) ,\label{e:nab^2f bou}\\
\|\nabla^3 f\|_{L^2(\parOmega_t)}^2&\le C\left( \|\laplace f\|_{H^2(\Omega_t)}^2+\|f\|_{H^3(\parOmega_t)}^2\right) .\label{e:nab^3f bou}
\end{align} 
\end{proposition}
\begin{proof}
For any $k\in\{ 1,2,3 \}$, it holds $\|\nabla\partial_k f\|_{L^2(\parOmega_t)}^2 \le C(\|\tangrad\partial_k f\|_{L^2(\parOmega_t)}^2  +\|\nabla^2 f\|_{L^2(\Omega_t)}^2+\|\nabla\laplace f\|_{L^2(\Omega_t)}^2)$ 
by applying Lemma \ref{l:3.3 jul}. Recall that we extend the unit outer normal $\nu$ to $\Omega_t$ by the harmonic extension
and   $\|\tinu\|_{H^{5/2+\delta}(\Omega_t)}\le C$. This, combined with Lemmas \ref{l:formu 1} and \ref{l:3.3 jul} implies that
\begin{align*}
\|\tangrad\partial_k f\|_{L^2(\parOmega_t)}^2\le{}& C(\|\nabla\tangrad f\|_{L^2(\parOmega_t)}^2+\|\nabla f\star\nabla\tilde{\nu}\star\tilde{\nu}\|_{L^2(\parOmega_t)}^2)\\ 
\le{}& C(\|\tangrad^2 f\|_{L^2(\parOmega_t)}^2+\|\nabla\laplace f\|_{L^2(\Omega_t)}^2+\|\nabla f\|_{H^1(\Omega_t)}^2\\
&\quad +\|\nabla f\star\nabla\tilde{\nu}\star\nabla\tilde{\nu}\|_{L^2(\Omega_t)}^2+\|\nabla^2 f\star\nabla\tilde{\nu}\|_{L^2(\Omega_t)}^2)\\
\le{}& C(\|\tangrad^2 f\|_{L^2(\parOmega_t)}^2+\|\nabla\laplace f\|_{L^2(\Omega_t)}^2+\|\nabla f\|_{H^1(\Omega_t)}^2),
\end{align*}
and $\|\nabla\partial_k f\|_{L^2(\parOmega_t)}^2\le  C(\|\tangrad^2 f\|_{L^2(\parOmega_t)}^2+\|\nabla\laplace f\|_{L^2(\Omega_t)}^2+\|\nabla f\|_{H^1(\Omega_t)}^2)$ as a consequence. 
Next, we apply \eqref{e:har ext 2} and Lemma  \ref{l:3.5 jul} to find that
\begin{align*}
\|\nabla f\|_{H^1(\Omega_t)}^2&\le C(\|\partial_\nu f\|_{H^{1/2}(\parOmega_t)}^2+\|\nabla f\|_{L^{2}(\Omega_t)}^2+\|\laplace f\|_{L^2(\Omega_t)}^2)\\
&\le C(\|\partial_\nu f\|_{H^{1/2}(\parOmega_t)}^2+\|f\|_{H^{1/2}(\parOmega_t)}^2+\|\laplace f\|_{L^2(\Omega_t)}^2).
\end{align*}
To control $\|\partial_\nu f\|_{H^{1/2}(\parOmega_t)}^2$, using Lemma \ref{l:3.3 jul} and by interpolation, one has 
\begin{align*}
\|\partial_\nu f\|_{H^{1/2}(\parOmega_t)}^2 
\le{}& \vare( \|\nabla^2 f\|_{L^{2}(\parOmega_t)}^2+ \|\nabla f\|_{H^{1}(\Omega_t)}^2)+C_\vare (\|\tangrad f\|_{L^{2}(\parOmega_t)}^2+\| f\|_{H^{1/2}(\parOmega_t)}^2+\|\laplace f\|_{L^{2}(\Omega_t)}^2),
\end{align*}
where  $\vare>0$ is sufficiently small.
We conclude that 
\begin{equation}\label{e:nab f H1}
\|\nabla f\|_{H^{1}(\Omega_t)}^2\le \vare \|\nabla^2 f\|_{L^{2}(\parOmega_t)}^2+  C(\|f\|_{H^1(\parOmega_t)}^2+\|\laplace f\|_{L^2(\Omega_t)}^2),
\end{equation}
and then \eqref{e:nab^2f bou} follows. 

To prove the second claim, by Lemma \ref{l:3.3 jul}, it follows that $\|\nabla\partial_k\partial_l f\|_{L^2(\parOmega_t)}^2 \le C(\|\tangrad\partial_k\partial_l f\|_{L^2(\parOmega_t)}^2+\|\nabla^3 f\|_{L^2(\Omega_t)}^2+\|\nabla^2\laplace f\|_{L^2(\Omega_t)}^2),k\in\{ 1,2,3 \}$. 
To estimate $\|\nabla^3 f\|_{L^2(\Omega_t)}^2$, from Lemma \ref{l:3.5 jul}, we obtain $\|\partial_i f\|_{H^2(\Omega_t)}^2\le C(\|\partial_\nu\partial_i f\|_{H^{1/2}(\parOmega_t)}^2+\|\nabla f\|_{L^2(\Omega_t)}^2+\|\nabla\laplace f\|_{L^2(\Omega_t)}^2)$ 
for $i\in\{ 1,2,3 \}$.  Then, we obtain $\|\partial_\nu\partial_i f\|_{H^{1/2}(\parOmega_t)}^2\le\vare\|\tangrad\partial_\nu\partial_i f\|_{L^{2}(\parOmega_t)}^2+C_\vare \|\partial_\nu\partial_i f\|_{L^{2}(\parOmega_t)}^2$ by interpolation,  
where $\vare>0$ is small enough. These, combined with \eqref{e:nab^2f bou}, \eqref{e:har ext 2} and the fact that $\|\tilde{\nu}\|_{H^{5/2+\delta}(\Omega_t)}\le C$, yield
\begin{align}
\|\nabla f\|_{H^2(\Omega_t)}^2\le{}&\vare(\|\nabla^3 f\|_{L^{2}(\parOmega_t)}^2+\|\nabla^2 f\star\nabla\tilde{\nu}\|_{L^{2}(\parOmega_t)}^2)+\|f\|_{H^{2}(\parOmega_t)}^2+\|\laplace f\|_{H^1(\Omega_t)}^2\nonumber\\
\le{}&\vare\|\nabla^3 f\|_{L^{2}(\parOmega_t)}^2+\|f\|_{H^{2}(\parOmega_t)}^2+\|\laplace f\|_{H^1(\Omega_t)}^2.\label{e:nab f H2}
\end{align}
Then, we control $\|\tangrad\partial_k\partial_l f\|_{L^2(\parOmega_t)}^2$ by
Lemma \ref{l:3.3 jul} and the fact that $\laplace \tilde{\nu}=0$  
\begin{align*}
\|\tangrad\partial_k\partial_l f\|_{L^2(\parOmega_t)}^2 
\le{}& C(\|\tangrad^2\partial_l f\|_{L^2(\parOmega_t)}^2+\|\nabla\tangrad\nabla f\|_{L^2(\Omega_t)}^2+\|\laplace\tangrad\nabla f\|_{L^2(\Omega_t)}^2+\|\nabla^2f\|_{L^2(\parOmega_t)}^2)\\
\le{}& C(\|\tangrad^2\partial_l f\|_{L^2(\parOmega_t)}^2+\|\nabla^2 f\|_{H^1(\Omega_t)}^2+\|\laplace f\|_{H^2(\Omega_t)}^2+\|f\|_{H^2(\parOmega_t)}^2).
\end{align*}
Again by \eqref{e:nab^2f bou} and Lemma \ref{l:3.3 jul}, we obtain 
\begin{align*}
&\|\tangrad^2\partial_l f\|_{L^2(\parOmega_t)}^2\\ 
\le{}& \|\tangrad^3 f\|_{L^2(\parOmega_t)}^2+\|\nabla^3 f\|_{L^2(\Omega_t)}^2+\|\nabla^3 f\star\nabla\tilde{\nu}\|_{L^2(\Omega_t)}^2+\|\nabla^2 f\star\nabla\tilde{\nu}\|_{L^2(\Omega_t)}^2+\|\nabla^2\laplace f\|_{L^2(\Omega_t)}^2\\
&+\|\nabla f\star\nabla^2\tilde{\nu}\|_{L^2(\Omega_t)}^2+\|\nabla^2 f\|_{L^2(\parOmega_t)}^2 +\|\nabla^2 f\star\nabla^2\tilde{\nu}\|_{L^2(\Omega_t)}^2+\|\nabla f\|_{H^{3/2+\delta}(\Omega_t)}^2\\
\le{}& C (\| f\|_{H^3(\parOmega_t)}^2+\|\nabla f\|_{H^2(\Omega_t)}^2+\|\laplace f\|_{H^2(\Omega_t)}^2).
\end{align*}
Recalling \eqref{e:nab f H2}, we conclude that $\|\nabla^3 f\|_{L^{2}(\parOmega_t)}^2\le \vare \|\nabla^3 f\|_{L^{2}(\parOmega_t)}^2+  C(\|f\|_{H^3(\parOmega_t)}^2+\|\laplace f\|_{H^2(\Omega_t)}^2)$,
and this completes the proof.
\end{proof}	

We will proceed with the estimates for the pressure.
\begin{lemma}\label{l:p est 1}
Assume that \eqref{e:a priori} holds for some $T>0$. Then, we have
\begin{equation*}
\sup_{t\in[0,T]}\| \nabla p\|_{L^2(\Omega_t)}^2\le e^{C(\mathcal{N}_T,\mathcal{M}_T)(1+T)}\left( 1+\|\nabla  p\|_{L^2(\Omega_0)}^2\right) .
\end{equation*}
\end{lemma}
\begin{proof}
From Lemma \ref{l:formu 1}, Reynolds transport theorem, and the divergence-free condition, one has
\begin{equation*}
\frac{d}{dt}\frac{1}{2}\int_{\Omega_t} |\nabla p|^2dx=\int_{\Omega_t} \nabla \matder p\cdot\nabla pdx+\int_{\Omega_t}\nabla v\star \nabla p\star\nabla pdx\eqqcolon I_1(t)+I_2(t).
\end{equation*} 
Clearly, \eqref{e:a priori} implies $|I_2(t)|\le C \|\nabla p\|_{L^2(\Omega_t)}^2$. For $|I_1(t)|$, by \eqref{e:Dtp}, \eqref{e:par nu p},  and  the divergence theorem, we have $|I_1(t)| \le \int_{\parOmega_t}\matder p\partial_\nu pdS-\int_{\Omega_t}\matder p\laplace pdx \le C(1+\| p\|_{H^1(\parOmega_t)}^2)-\int_{\Omega_t}\matder p\laplace pdx$. 
To control $\int_{\Omega_t}\matder p\laplace p dx$, we consider the following elliptic equation
\begin{equation*}
\begin{cases}
-\laplace u=\laplace p,&\text{ in }\Omega_t,\\
u=0,&\text{ on }\parOmega_t.
\end{cases}
\end{equation*}
Then, we see that $-\int_{\Omega_t}\matder p\laplace pdx 
=\int_{\Omega_t}\laplace\matder p udx+\int_{\parOmega_t}\matder p \partial_\nu udS\eqqcolon I_{11}(t)+I_{12}(t)$. 
Note that \eqref{e:laplace p} implies $|\laplace p|\le C$, and we have $\|u\|_{H^1(\Omega_t)}\le C$. Also, we get $\|\nabla u\|_{L^2(\parOmega_t)}^2\le C$ and $|I_{12}(t)|\le \|\matder p\|_{L^2(\parOmega_t)}^2+	\|\partial_\nu u\|_{L^2(\parOmega_t)}^2  \le C(1+\| p\|_{H^1(\parOmega_t)}^2)$ from Lemma \ref{l:3.3 jul}. 
We are left with  $I_{11}(t)$, for which one can repeat the argument in \cite[Propsition 6.3]{Julin2024} to deduce $\| u\|_{H^2(\Omega_t)}^2\le C(1+\| p\|_{H^1(\parOmega_t)}^2)$. Then, by \eqref{e:mhd}, \eqref{e:a priori}, Lemma \ref{l:lapDtp},   \eqref{e:def error 2} and \eqref{e:par nu p}, we integrate by parts to obtain $I_{11}(t) \le C(1+\|p\|_{H^1(\parOmega_t)}^2+\|\nabla p\|_{L^2(\Omega_t)}^2)$. 
Combining the above calculations, it follows that $I_1(t)+I_2(t)\le C(1+\|p\|_{H^1(\parOmega_t)}^2+\|\nabla p\|_{L^2(\Omega_t)}^2)$. 
With the help of estimate \eqref{e:p H^1 bou}, the proof is complete. 
\end{proof}
\begin{lemma}\label{l:p est 2}
Assume that \eqref{e:a priori} holds for some $T>0$. Then, we have
\begin{equation*}
\int_0^T \|\tangrad ^2p\|_{L^2(\parOmega_t)}^2dt\le C(\mathcal{N}_T,\mathcal{M}_T)(1+T).
\end{equation*}
\end{lemma}
\begin{proof}
We define $I(t)\coloneqq \int_{\parOmega_t}\tangrad p\cdot\tangrad(\nabla v\nu\cdot\nu)dS$, 
and from the hypothesis \eqref{e:a priori} and \eqref{e:p H^1 bou}, we see that $|I(t)|\le C \|\tangrad p\|_{L^2(\parOmega_t)}^2+C \|\nabla^2 v\|_{L^2(\parOmega_t)}^2+C \|\nabla v\star B\|_{L^2(\parOmega_t)}^2\le C$. 
Again by \eqref{e:a priori}, the divergence theorem, Lemma \ref{l:formu 1} and \eqref{e:reynold 2}, we deduce for sufficiently small $\vare>0$ that
\begin{align*}
\frac{d}{dt}I(t) 
\le{}&C|I(t)|+\int_{\parOmega_t}\matder\tangrad p\cdot\tangrad(\nabla v\nu\cdot\nu)+\tangrad p\cdot\matder\tangrad(\nabla v\nu\cdot\nu)dS\\
\le{}& C_\vare +\vare\|\tangrad\matder p\|_{L^2(\parOmega_t)}^2+\int_{\parOmega_t}\tangrad p\cdot\tangrad\matder(\nabla v\nu\cdot\nu)dS\\
\le{}& C_\vare +\vare\|\tangrad\matder p\|_{L^2(\parOmega_t)}^2-\int_{\parOmega_t}\surflaplace p\matder(\nabla v\nu\cdot\nu)dS\eqqcolon C_\vare+\vare I_1(t)+I_2(t).
\end{align*}
By \eqref{e:a priori}, \eqref{e:Dtp} and \eqref{e:p H^1 bou}, it holds $|I_{1}(t)|\le C(1+\|v_n\|_{H^3(\parOmega_t)}^2+ \|\tangrad B\star B\star v_n\|_{L^2(\parOmega_t)}^2+\|\tangrad p\|_{H^1(\parOmega_t)}^2)\le C(1+\|\tangrad^2 p\|_{L^2(\parOmega_t)}^2)$. 
For $|I_2(t)|$,  from \eqref{e:a priori}, Lemma \eqref{l:formu 1} and the divergence theorem, we have
\begin{align*}
|I_{2}(t)|\le{}&- \int_{\parOmega_t}\surflaplace p(\nabla\matder v\nu\cdot\nu)dS+C\|\tangrad p\|_{L^1(\parOmega_t)}\\
={}&-\int_{\parOmega_t}\surflaplace p(\nabla(-\nabla p+H\cdot \nabla H)\nu\cdot\nu)dS+\vare\|\tangrad^2 p\|_{L^2(\parOmega_t)}^2 +C_\vare\\
\le{}&\int_{\parOmega_t}\surflaplace p(\nabla^2p\nu\cdot\nu)dS-\int_{\parOmega_t}\surflaplace p\star\nabla^2 H\star \nabla H\star \nu\star \nu dS+\vare\|\tangrad^2p\|_{L^2(\parOmega_t)}^2+C_\vare\\
\le{}&\int_{\parOmega_t}\surflaplace p(\nabla^2p\nu\cdot\nu)dS+\vare\|\tangrad^2 p\|_{L^2(\parOmega_t)}^2+C_\vare.
\end{align*}
Recalling $|\laplace p|\le C$ and by \eqref{e:lap_B}, \eqref{e:par nu p}, the divergence theorem, for $\vare>0$ small enough, we deduce
\begin{align*}
\int_{\parOmega_t}\surflaplace p(\nabla^2 p\nu\cdot\nu)dS
={}&\int_{\parOmega_t}\surflaplace p\laplace  p-\surflaplace p\surflaplace   p-\surflaplace p\meancurv \partial_\nu  pdS\\
\le{}& C+\frac{\vare}{2} \|\tangrad^2 p\|_{L^2(\parOmega_t)}^2+C_\vare\| p\|_{H^1(\parOmega_t)}^2-\int_{\parOmega_t}|\tangrad^2 p|^2dS\\
&+\|\surflaplace p\|_{L^2(\parOmega_t)}\|\partial_\nu  p\|_{L^2(\parOmega_t)}\|p\|_{L^\infty(\parOmega_t)}\le -\frac 34\|\tangrad^2 p\|_{L^2(\parOmega_t)}^2+C_\vare.
\end{align*}
Above, we have applied \cite[Remark 2.4]{Fusco2020} that $\|\tangrad^2p\|_{L^2(\parOmega_t)}^2\le \|\surflaplace p\|_{L^2(\parOmega_t)}^2+C\int_{\parOmega_t}|B|^2|\tangrad p|^2dS$.
Combining the above calculations, the proof is complete since $\frac{d}{dt}I(t)\le-\frac{1}{2}\|\tangrad^2 p\|_{L^2(\parOmega_t)}^2+C$.  
\end{proof}

\begin{lemma}\label{l:p est 3}
Assume that \eqref{e:a priori} holds for some $T>0$. Then, we have
\begin{equation*}
\sup_{t\in[0,T]}\| \nabla^2 p\|_{L^2(\Omega_t)}^2\le e^{C(\mathcal{N}_T,\mathcal{M}_T)(1+T)}(1+\|\nabla^2  p\|_{L^2(\Omega_0)}^2).
\end{equation*}
\end{lemma}
\begin{proof}
We differentiate and apply Lemma \ref{l:formu 3} to obtain
\begin{equation*}
\frac{d}{dt}\frac{1}{2}\int_{\Omega_t} |\nabla^2 p|^2dx=\int_{\Omega_t} \nabla^2 \matder p:\nabla^2 pdx+\int_{\Omega_t}\nabla^2 v\star \nabla p\star\nabla^2 p+\nabla v\star \nabla^2 p\star\nabla^2 pdx
\eqqcolon I_1(t)+I_2(t).
\end{equation*} 
From \eqref{e:a priori}, \eqref{e:laplace p} and using Lemma \ref{l:p est 1}, we have
\begin{align*}
I_1(t) 
&\le\int_{\Omega_t}\sum_{i,j}\partial_i(\partial_j\matder p\partial_i\partial_j p)dx-\int_{\Omega_t}\nabla\matder p\cdot\nabla\laplace pdx\\
&\le \int_{\parOmega_t}\sum_{j}\partial_j\matder p\partial_\nu\partial_j pdS+\int_{\Omega_t}\laplace\matder p\laplace pdx-\int_{\parOmega_t}\partial_\nu \matder p\laplace pdS\\
&\le C\sum_{j}\|\partial_\nu\partial_j p\|_{L^2(\parOmega_t)}^2+C\|\partial_\nu\matder p\|_{L^2(\parOmega_t)}^2+C\|\laplace\matder  p\|_{L^2(\Omega_t)}^2\eqqcolon I_{11}(t)+I_{12}(t)+I_{13}(t),\\
I_2(t)&\le C(\|v\|_{H^{7/2}(\Omega_t)}^2\|\nabla p\|_{L^6(\Omega_t)}^2+\|\nabla^2 p\|_{L^2(\Omega_t)}^2)\le C(1+\|\nabla^2 p\|_{L^2(\Omega_t)}^2).
\end{align*} 
We apply Lemmas \ref{l:lapDtp} and \ref{l:p est 1}, and \eqref{e:nab^2f bou} to obtain $|I_{13}(t)|\le C(1+\|\nabla^2 p\|_{L^2(\Omega_t)}^2)$ and $|I_{11}(t)|\le C(1+\|p\|_{H^2(\parOmega_t)}^2)$.
Finally, \eqref{e:a priori}, Lemmas \ref{l:lapDtp} and \ref{l:3.3 jul}, and  \eqref{e:har ext 2} imply that
\begin{align*}
|I_{12}(t)|&\le C( \| \tangrad\matder p\|_{L^2(\parOmega_t)}^2+ \|\nabla\matder p\|_{L^2(\Omega_t)}^2+ \|\laplace\matder p\|_{L^2(\Omega_t)}^2 )\\
&\le C( \| \tangrad\matder p\|_{L^2(\parOmega_t)}^2+ \|\matder p\|_{H^{1/2}(\parOmega_t)}^2+ \|\laplace\matder p\|_{L^2(\Omega_t)}^2 )\\
&\le C(1+\|p\|_{H^2(\parOmega_t)}^2+\|\nabla^2 p\|_{L^2(\Omega_t)}^2).
\end{align*}
Combined with \eqref{e:p H^1 bou} and Lemma \ref{l:p est 2}, the proof is complete. 
\end{proof}	

\begin{lemma}\label{l:p est 4}
Assume that \eqref{e:a priori} holds for some $T>0$. Then, we have
\begin{equation*}
\sup_{t\in[0,T]}\|\tangrad^2 p\|_{L^2(\parOmega_t)}^2+\int_0^T \|\tangrad ^3p\|_{L^2(\parOmega_t)}^2dt\le C\left( T,\mathcal{N}_T,\mathcal{M}_T,\|\tangrad^2p\|_{L^2(\parOmega_0)},\|\nabla p\|_{H^1(\Omega_0)}\right).
\end{equation*} 
\end{lemma}
\begin{proof}
We define 
\begin{equation*}
I(t)\coloneqq \int_{\parOmega_t}\tangrad^2 p: \tangrad^2(\nabla v\nu\cdot\nu)dS+\vare\int_{\parOmega_t}| \tangrad^2 p|^2dS\eqqcolon I_1(t)+\vare I_2(t),
\end{equation*}
where $\vare>0$ will be chosen later. From \eqref{e:a priori}, \eqref{e:p H^1 bou}, Lemmas  \ref{l:p est 2} and  \ref{l:2.12 jul}, we have
\begin{equation*}
|I_1(t)|\le C_\vare (\|\nabla^3 v\|_{L^2(\parOmega_t)}^2+\|\nabla^2 v\star B\|_{L^2(\parOmega_t)}^2 + \|\nabla v\star\tangrad B\|_{L^2(\parOmega_t)}^2)+\frac{\vare}{2} \|\tangrad^2 p\|_{L^2(\parOmega_t)}^2\le \frac{\vare}{2} \|\tangrad^2 p\|_{L^2(\parOmega_t)}^2+C_\vare,
\end{equation*} 
and therefore, $I(t)\ge -C_\vare+\dfrac \vare 2\|\tangrad^2 p(\cdot,t)\|_{L^2(\parOmega_t)}^2$.
We differentiate and use \eqref{e:a priori}, \eqref{e:p H^1 bou}, the divergence theorem, Lemmas \ref{l:formu 1} and \ref{l:2.12 jul} to obtain
\begin{align*}
\frac{d}{dt}I_1(t)
\le{}&C|I_1(t)|+\int_{\parOmega_t}\matder\tangrad^2 p:\tangrad^2(\nabla v\nu\cdot\nu)+\tangrad^2 p:\matder\tangrad^2(\nabla v\nu\cdot\nu)dS\\
\le{}& C_\vare +\vare(\|\tangrad^2 p\|_{L^2(\parOmega_t)}^2+\|\tangrad^2\matder p\|_{L^2(\parOmega_t)}^2)+\int_{\parOmega_t}\tangrad^2 p:\tangrad^2\matder(\nabla v\nu\cdot\nu)dS\\
\le{}& \vare\|\tangrad^2 p\|_{L^2(\parOmega_t)}^2+C_\vare+\vare\underbrace{\|\tangrad^2\matder p\|_{L^2(\parOmega_t)}^2}_{\eqqcolon I_{11}(t)}\underbrace{-\int_{\parOmega_t}\tangrad \surflaplace p\cdot\tangrad\matder(\nabla v\nu\cdot\nu)dS}_{\eqqcolon I_{12}(t)}.
\end{align*} 
The first term can be controlled by \eqref{e:a priori}, \eqref{e:Dtp}, \eqref{e:p H^1 bou}  and Lemma \ref{l:2.12 jul}, i.e., $|I_{11}(t)|\le C(1+\|\tangrad^2  p\|_{L^2(\parOmega_t)}^2+\|\tangrad^3 p\|_{L^2(\parOmega_t)}^2)$. 
As for $I_{12}(t)$, applying  \eqref{e:a priori}, Lemma \eqref{l:formu 1} and the divergence theorem, it follows that
\begin{align*}
|I_{12}(t)|
\le{}&- \int_{\parOmega_t}\tangrad\surflaplace p\cdot\tangrad(\nabla\matder v\nu\cdot\nu)dS+C(\|\tangrad^2 p\|_{L^2(\parOmega_t)}^2+1)\\
={}&-\int_{\parOmega_t}\tangrad\surflaplace p\cdot\tangrad(\nabla(-\nabla p+H\cdot \nabla H)\nu\cdot\nu)dS+C(\|\tangrad^2 p\|_{L^2(\parOmega_t)}^2+1)\\
\le{}&\int_{\parOmega_t}\tangrad\surflaplace p\cdot\tangrad(\nabla^2p\nu\cdot\nu)dS-\int_{\parOmega_t}\tangrad\surflaplace p\cdot\tangrad(\nabla^2H\star H\star\nu\star\nu)dS+C(\|\tangrad^2 p\|_{L^2(\parOmega_t)}^2+1)\\
\le{}&\int_{\parOmega_t}\tangrad\surflaplace p\cdot\tangrad(\nabla^2p\nu\cdot\nu)dS+\vare\|\tangrad^3 p\|_{L^2(\parOmega_t)}^2+C(\|\tangrad^2 p\|_{L^2(\parOmega_t)}^2+1).
\end{align*}
To estimate $\int_{\parOmega_t}\tangrad\surflaplace p\cdot\tangrad(\nabla^2p\nu\cdot\nu)dS$, by \eqref{e:a priori}, \eqref{e:lap_B}, \eqref{e:laplace p},  \eqref{e:par nu p}, Lemma \ref{l:3.5 jul} and the divergence theorem, it holds
\begin{align*}
\int_{\parOmega_t}\tangrad\surflaplace p\cdot\tangrad(\nabla^2 p\nu\cdot\nu)dS
={}&\int_{\parOmega_t}\tangrad\surflaplace p\cdot\tangrad\laplace  p-\tangrad\surflaplace p\cdot\tangrad\surflaplace   p-\tangrad\surflaplace p\cdot\tangrad(\meancurv \partial_\nu p)dS\\
\le{}& C_\vare\|\laplace p\|_{H^{3/2}(\Omega_t)}^2+\vare\|\tangrad^3 p\|_{L^2(\parOmega_t)}^2+C_\vare\|p\|_{H^2(\parOmega_t)}^2-\frac 78\int_{\parOmega_t}|\tangrad^3 p|^2dS\\
&+\|\tangrad^3 p\|_{L^2(\parOmega_t)}(\|\tangrad\partial_\nu  p\|_{L^2(\parOmega_t)}\|p\|_{L^\infty(\parOmega_t)}+\|\partial_\nu  p\|_{L^4(\parOmega_t)}\|\tangrad p\|_{L^4(\parOmega_t)})\\
\le{}& C_\vare-\frac 34\int_{\parOmega_t}|\tangrad^3 p|^2dS+C_\vare\|\nabla\partial_\nu  p\|_{L^2(\parOmega_t)}^2+\|\tangrad^3 p\|_{L^2(\parOmega_t)}\|\nabla p\|_{H^1(\Omega_t)}^2\\
\le{}& C_\vare-\frac 12\|\tangrad^3 p\|_{L^2(\parOmega_t)}^2+C \|\tangrad^2 p\|_{L^2(\parOmega_t)}^2. 
\end{align*}
Above, we have used Lemma  \ref{l:formu 1} and \eqref{e:nab^2f bou} to deduce $\|\nabla\partial_\nu p\|_{L^2(\parOmega_t)}^2 
\le C(1+\|p\|_{H^2(\parOmega_t)}^2+\|\laplace p\|_{H^1(\Omega_t)}^2)$, 
and the result \cite[Lemma 2.3]{Fusco2020}, i.e., $\|\tangrad^3 p\|_{L^2(\parOmega_t)}^2\le \|\tangrad\surflaplace p\|_{L^2(\parOmega_t)}^2+C\|p\|_{H^2(\parOmega_t)}^2$.
Similarly, we can obtain $\frac{d}{dt}I_2(t)\le C(1+\|\tangrad^2 p\|_{L^2(\parOmega_t)}^2+\|\tangrad^3 p\|_{L^2(\parOmega_t)}^2)$.

Combined the above calculations and by choosing suitable $\vare>0$, one has $\frac{d}{dt}I(t)\le -\frac{1}{4}\|\tangrad^3 p\|_{L^2(\parOmega_t)}^2+C \|\tangrad^2 p\|_{L^2(\parOmega_t)}^2+C$. 
Integrating the above over $[0,t]$ with $ 0<t\le T$ and recalling  \eqref{e:p H^1 bou} together with $I(t)\ge -C_\vare+\frac \vare 2\|\tangrad^2 p(\cdot,t)\|_{L^2(\parOmega_t)}^2$, the lemma follows.  
\end{proof} 

\begin{lemma}\label{l:p est 5}
Assume that \eqref{e:a priori} holds for some $T>0$. Then, we have 
\begin{equation*}
\sup_{t\in[0,T]}\|p\|_{H^3(\Omega_t)}^2\le C\left( \mathcal{N}_T,\mathcal{M}_T,\|\tangrad^2p\|_{L^2(\parOmega_0)},\|\nabla p\|_{H^2(\Omega_0)},T\right) .
\end{equation*}
\end{lemma} 
\begin{proof}
We differentiate and apply Lemma \ref{l:formu 3} to obtain
\begin{align*}
&\ \quad  \frac{d}{dt}\frac{1}{2}\int_{\Omega_t} |\nabla^3 p|^2dx\\  
&=\int_{\Omega_t} \sum_{ijk}\partial_{ijk}\matder p\partial_{ijk} pdx+\int_{\Omega_t}\nabla^3 v\star \nabla p\star\nabla^3 p+\nabla^2 v\star \nabla^2 p\star\nabla^3 p+\nabla v\star \nabla^3 p\star\nabla^3 pdx\\ 
&\eqqcolon   I_1(t)+I_2(t).
\end{align*}
From \eqref{e:a priori}, \eqref{e:laplace p} and Lemma \ref{l:lapDtp}, we have
\begin{align*}
|I_1(t)|\le{}&\int_{\Omega_t}\sum_{i,j,k}\partial_i(\partial_{jk}\matder p\partial_{ijk} p)dx-\int_{\Omega_t}\sum_{j,k}\partial_{jk}\matder p\partial_{jk}\laplace pdx\\
\le{}& \int_{\parOmega_t}\sum_{j,k}\partial_{jk}\matder p\partial_\nu\partial_{jk} pdS+\int_{\Omega_t}\sum_{k}\partial_{k}\laplace\matder p\partial_{k}\laplace pdx -\int_{\parOmega_t}\sum_{k}\partial_\nu\partial_{k}\matder p\partial_{k}\laplace pdS\\ 
\le{}& C \sum_{j,k}\|\partial_\nu\partial_{jk} p\|_{L^2(\parOmega_t)}^2+C\sum_{j,k}\|\partial_{jk}\matder p\|_{L^2(\parOmega_t)}^2  +C(1+\|\nabla p\|_{H^2(\Omega_t)}^2), 
\end{align*}
and $|I_2(t)|
\le C(1+\|\nabla p\|_{H^2(\Omega_t)}^2)$. Applying  \eqref{e:nab^2f bou} and \eqref{e:nab^3f bou}, we obtain
\begin{align*}
\|\partial_\nu\partial_{jk} p\|_{L^2(\parOmega_t)}^2+\|\partial_{jk}\matder p\|_{L^2(\parOmega_t)}^2 
&\le C(\|\laplace p\|_{H^2(\Omega_t)}^2+\| \nabla p\|_{H^2(\Omega_t)}^2+\|\matder p\|_{H^2(\parOmega_t)}^2+\|p\|_{H^3(\parOmega_t)}^2)\\
&\le C(1+\|\nabla p\|_{H^2(\Omega_t)}^2+\|p\|_{H^3(\parOmega_t)}^2),
\end{align*}
for any indices $j,k$.
The claim follows from Lemma \ref{l:3.5 jul} and the previous pressure estimates (\eqref{e:p H^1 bou}, Lemmas  \ref{l:p est 1},  \ref{l:p est 2}, \ref{l:p est 3} and \ref{l:p est 4}), since 
\begin{equation*}
\frac{d}{dt}\frac{1}{2}\int_{\Omega_t} |\nabla^3 p|^2 
\le C(1+\|\nabla p\|_{H^2(\Omega_t)}^2+\|  p\|_{H^3(\parOmega_t)}^2).
\end{equation*} 
\end{proof}	

We conclude this section by controlling the initial quantities $\barE(0)$ and $\sum_{k=0}^{3}\|\matder^{3-k}p\|_{H^{3k/2+1}(\Omega_0)}^2$. 
\begin{proposition}\label{p:initial cond}
Assume that $\Omega_0$ is a smooth and $\mathcal{M}_0\coloneqq \radi-\|h_0\|_{L^\infty(\parOmega)}>0$. Then, we have  
\begin{equation*}
\barE(0)+\sum_{k=0}^{3}\|\matder^{3-k}p\|_{H^{3k/2+1}(\Omega_0)}^2\le C\left(\mathcal{M}_0,\|v_0\|_{H^6(\Omega_0)},\|H_0\|_{H^6(\Omega_0)},\|\meancurv\|_{H^5(\parOmega_0)}\right) .
\end{equation*} 
The result remains valid when the initial time is replaced with any $t\in (0,T)$, provided $\|h(\cdot,t)\|_{L^\infty(\Gamma)}< \radi$.
\end{proposition}
\begin{proof}
We only need to consider the case when $t=0$ and we divide the proof into three steps.

\textbf{Step 1.} We control $\|\matder^{4-k}H\|_{H^{3k/2}(\Omega_0)}^2$ by the lower-order velocity terms using  \eqref{e:Dt^k H} and \eqref{e:nabDt^k H}. For $k=0$, from $\| H\|_{L^\infty(\Omega_0)}\le C\| H\|_{H^6(\Omega_0)}$, we apply \eqref{e:Dt^k H} to obtain
\begin{align*} 
\|\matder^4 H\|_{L^2(\Omega_0)}^2 
\le{}& C 
(\sum_{|\beta|\le 3}\|\nabla\matder^{\beta_1}v\|_{L^2(\Omega_0)}^2+\sum_{|\beta|\le 2}\|\nabla\matder^{\beta_1}v\|_{L^3(\Omega_0)}^2\| \nabla\matder^{\beta_{2}}v\|_{L^6(\Omega_0)}^2\\
&+\sum_{|\beta|\le 1}\|\nabla\matder^{\beta_1}v\|_{L^6(\Omega_0)}^2\| \nabla\matder^{\beta_{2}}v\|_{L^6(\Omega_0)}^2\| \nabla\matder^{\beta_{3}}v\|_{L^6(\Omega_0)}^2+\|v\|_{H^3(\Omega_0)}^8)\\
\le{}& C\|\matder^3v\|_{H^{1}(\Omega_0)}^2+C(1+\|\matder^2 v\|_{H^{2}(\Omega_0)}^2)(1+\|\matder v\|_{H^{2}(\Omega_0)}^2).
\end{align*}

We claim that 
\begin{align}
&\sum_{k=1}^3 \|\matder^{4-k}H\|_{H^{3k/2}(\Omega_0)}^2\nonumber\\
\le{}& C(\| v\|_{H^{4}(\Omega_0)}, \| H\|_{H^{4}(\Omega_0)})(1+\sum_{k=1}^3 \|\matder^{4-k}v\|_{H^{(3k-1)/2 }(\Omega_0)}^2+\| v\|_{H^{11/2}(\Omega_0)}^2+\| H\|_{H^{9/2}(\Omega_0)}^2).\label{e:claim}
\end{align} 
Indeed, by \eqref{e:Dt^k H}, it follows that
\begin{align*} 
\|\matder^3 H\|_{H^{3/2}(\Omega_0)}^2 
\le{}& C\sum_{|\beta|\le 2}\|\nabla\matder^{\beta_1}v\|_{H^{3/2}(\Omega_0)}^2\| H\|_{H^{2 }(\Omega_0)}^2+
C \sum_{|\beta|\le 1}\|\nabla\matder^{\beta_1}v\|_{H^{2 }(\Omega_0)}^2\\
&\quad \cdot\| \nabla\matder^{\beta_{m}}v\|_{H^{2}(\Omega_0)}^2\| H\|_{H^{2 }(\Omega_0)}^2+C\|v\|_{H^3(\Omega_0)}^6\| H\|_{H^{2}(\Omega_0)}^2\\
\le{}& C(\| v\|_{H^{4}(\Omega_0)}, \| H\|_{H^{4}(\Omega_0)})(1+\|\matder^2v\|_{H^{5/2}(\Omega_0)}^2+\|\matder v\|_{H^{5/2}(\Omega_0)}^2).
\end{align*}
Again from \eqref{e:Dt^k H} and Lemma \ref{l:2.10 jul}, we see that
\begin{align*}
\|\matder^2H\|_{H^{3}(\Omega_0)}^2 
&\le C\|\nabla\matder v\|_{H^{3}(\Omega_0)}^2\| H\|_{H^{3}(\Omega_0)}^2+C\|\nabla v\|_{H^{3}(\Omega_0)}^4 \| H\|_{H^{3}(\Omega_0)}^2\\
&\le C(\| v\|_{H^{4}(\Omega_0)}, \| H\|_{H^{4}(\Omega_0)})(1+\|\matder v\|_{H^{4}(\Omega_0)}^2),\\
\|\matder H\|_{H^{9/2}(\Omega_0)}^2&\le C(\|H\|_{L^{\infty}(\Omega_0)}^2\|v\|_{H^{11/2}(\Omega_0)}^2+\|H\|_{H^{9/2}(\Omega_0)}^2\|v\|_{L^{\infty}(\Omega_0)}^2)\\
&\le C(\| v\|_{H^{4}(\Omega_0)}, \| H\|_{H^{4}(\Omega_0)})(\|v\|_{H^{11/2}(\Omega_0)}^2+\|H\|_{H^{9/2}(\Omega_0)}^2).
\end{align*} 

\textbf{Step 2.} We control $ \|\matder^{4-k}v\|_{H^{3k/2}(\Omega_0)}^2$ by the pressure terms.  
Note that 
\begin{equation*}
\|\matder v\|_{H^{9/2}(\Omega_0)}^2\le C(\| p\|_{H^{11/2}(\Omega_0)}^2+\| H\|_{H^{11/2}(\Omega_0)}^2\| H\|_{H^{9/2}(\Omega_0)}^2)\le C(\| p\|_{H^{11/2}(\Omega_0)}^2  +1),
\end{equation*} 
and by Lemma \ref{l:nab H, nab H}, we have
\begin{equation*}
\|\matder^2 v\|_{H^{3}(\Omega_0)}^2
\le\|\nabla\matder p\|_{H^{3}(\Omega_0)}^2+\|[\matder,\nabla] p\|_{H^{3}(\Omega_0)}^2+C\le C(\|\nabla\matder p\|_{H^{3}(\Omega_0)}^2+\|p\|_{H^{4}(\Omega_0)}^2+1).
\end{equation*} 
Similarly, applying Lemmas \ref{l:nab^2H,H} and \ref{l:[Dt^l,nab]p}, we  obtain 
\begin{align*}
\|\matder^3 v\|_{H^{3/2}(\Omega_0)}^2\le{}& C(\|\nabla\matder^2 p\|_{H^{3/2}(\Omega_0)}^2+\|\nabla\matder p\|_{H^{3/2}(\Omega_0)}^2+\|p\|_{H^{9/2}(\Omega_0)}^2+1),\\
\|\matder^4 v\|_{L^2(\Omega_0)}^2\le{}& C(\|\nabla\matder^3 p\|_{L^2(\Omega_0)}^2+\|\nabla\matder^2 p\|_{L^2(\Omega_0)}^2+\|\nabla\matder  p\|_{H^2(\Omega_0)}^2+\|p\|_{H^{9/2}(\Omega_0)}^2).
\end{align*}  

\textbf{Step 3.} We show that $\sum_{k=0}^{3}\|\matder^{3-k}p\|_{H^{3k/2+1}(\Omega_0)}^2\le C$.  Consider the following elliptic equation
\begin{equation*}
\begin{cases}
-\laplace p=\partial_i v^j\partial_j v^i-\partial_i H^j\partial_j H^i,\quad & \text{ in } \Omega_0,\\
p=\meancurv_{\Gamma_0},\quad & \text{ on } \parOmega_0.		
\end{cases}
\end{equation*}
We find that $\|p\|_{H^{11/2}(\Omega_0)}\le C (\|\partial_i v^j\partial_j v^i-\partial_i H^j\partial_j H^i\|_{H^{7/2}(\Omega_0)}+\|\meancurv\|_{H^5(\parOmega_0)})\le C$ from the standard elliptic estimates.
Again by the elliptic estimates, 
it holds $\|\matder p\|_{H^{4}(\Omega_0)}\le C(\|\laplace \matder p\|_{H^{2}(\Omega_0)}+\|\matder p\|_{H^{7/2}(\parOmega_0)})$, 
and $\|\matder^2p\|_{H^{5/2}(\Omega_0)}\le C(\|\laplace \matder^2p\|_{H^{1/2}(\Omega_0)}+\|\matder^2p\|_{H^{2}(\parOmega_0)})$. 
Also,  by \eqref{e:har ext 3}, $\|\matder^3p\|_{H^1(\Omega_0)}\le C(\|\laplace \matder^3p\|_{L^2(\Omega_0)}+\|\matder^3p\|_{H^{1/2}(\parOmega_0)})$. 
The calculations of the remaining terms on the right-hand side are direct applications of Lemmas \ref{l:lapDtp} and \ref{l:Dt^lp}, and \eqref{e:Dtp}, since we have $\|p\|_{H^{11/2}(\Omega_0)}\le C$.

Finally, for $1\le j\le 3,\|\tangrad(\matder^j v\cdot\nu)\|_{L^2(\parOmega_0)}^2$  can be estimated by the trace theorem due to the regularity of the boundary. Using the mean curvature bound, we apply Lemma \ref{l:2.12 jul} to obtain $\|B\|_{H^2(\parOmega_0)}\le C$ and therefore $\|\tangrad(\matder^j v\cdot\nu)\|_{L^2(\parOmega_0)}^2 \le C(\|\tangrad\matder^j v\star\nu\|_{L^2(\parOmega_0)}^2+\|\matder^j v\star B\|_{L^2(\parOmega_0)}^2)  
\le C$.
This concludes the proof of the proposition. 
\end{proof} 
\section{Estimates for the error terms}\label{s:est error}

In this section, we estimate the error terms by the energy functional and the pressure. We start with the following results. 

\begin{lemma}\label{l:B by p}
Assume that  \eqref{e:a priori} holds for $T>0$. Then, we have $\|B\|_{H^{5/2}(\parOmega_t)}\le C$, and  $\|B\|_{H^k(\parOmega_t)}\le  C( 1+\|p\|_{H^k(\parOmega_t)}) $ for $k\in \N /2 ,k \leq 9/2$.  Assume further that  $\sup_{0\le t< T}E_{l-1}(t) \leq C$ for $l \geq 4$. Then, it holds $\|B\|_{H^{3l/2-1}(\parOmega_t)} \leq C$, and $\|B\|_{H^k(\parOmega_t)}\le  C(1+\|p\|_{H^k(\parOmega_t)})$ for $k\in \N /2 ,k \leq 3l/2+1$. 
\end{lemma}
\begin{proof}
We recall \eqref{e:p H^3} that $\|p\|_{H^3(\Omega_t)}\le C$ by the results in Section \ref{s:p est}. Since $\parOmega_t$ is uniformly $H^{3+\delta}(\Gamma)$-regular, it holds $\|B\|_{L^\infty(\parOmega_t)}+\|B\|_{H^1(\parOmega_t)}\le C$.
Applying Lemma \ref{l:2.12 jul}, for $k\in\N /2 ,k\le 3$, we see that $\|B\|_{H^k(\parOmega_t)}\le C(1+\|\meancurv\|_{H^k(\parOmega_t)})\le C(1+\|p\|_{H^k(\parOmega_t)})$, 
and $\|B\|_{H^{5/2}(\parOmega_t)}\le C$. Again by Lemma \ref{l:2.12 jul}, the first claim follows. As for $l\ge 4$, the assumption implies that
\begin{align*}
\|p\|_{H^{3l/2-1}(\parOmega_t)}^2&\le C(1+\|\nabla p\|_{H^{3(l-1)/2}(\Omega_t)}^2)\\
&\le C(1+\|\matder v\|_{H^{3(l-1)/2}(\Omega_t)}^2+\|H\cdot\nabla H\|_{H^{\lfloor3l/2-1\rfloor}(\Omega_t)}^2)\le C.
\end{align*}
For $l=4$,  we have $\|p\|_{H^5(\parOmega_t)}\le C$ and  $\|B\|_{H^{9/2}(\parOmega_t)}\le C$ by the first claim. Moreover, by Lemma \ref{l:2.12 jul}, it implies $\|B\|_{H^5(\parOmega_t)}\le C(1+\|p\|_{H^5(\parOmega_t)})\le C$, i.e., $\|B\|_{H^{3l/2-1}(\parOmega_t)}\le C$ in this case. Therefore, it holds $\|B\|_{H^k(\parOmega_t)}\le C(1+\|\meancurv\|_{H^k(\parOmega_t)})\le(1+\|p\|_{H^k(\parOmega_t)}), k\in \N /2 , k\le 3l/2+1$.
Using a similar argument, the second claim follows for $l\ge 5$. 
\end{proof}

\begin{lemma}\label{l:err est 1}
Assume that  \eqref{e:a priori} holds for $T>0$. We have $\|R_{\Rdiv}^l\|_{H^{1/2}(\Omega_t)}^2 \leq C( 1+\|\nabla^2 p\|_{H^{1/2}(\Omega_t)}^2)  \barE(t)$ for $l\le 3$. 
Assume further that $\sup_{0\le t< T}E_{l-1}(t)\le C$ for $l\ge 4$. Then, we have $\|R_{\Rdiv}^l\|_{H^{1/2}(\Omega_t)}^2 \leq C E_l(t)$,
and there exists a constant $\vare>0$ small enough such that $\|R_{\Rdiv}^{l-k}\|_{H^{3k/2-1}(\Omega_t)}^2 \leq \vare E_l(t)+C_\vare$ for $k\in \N ,1\le k\le l$.
\end{lemma} 
\begin{proof}
Thanks to the regularity of the free boundary in Lemma \ref{l:B by p},  
it is feasible to extend functions in $H^2(\Omega_t)$ to the entire space $\mathbb{R}^3$ (e.g., {\cite[Proposition 2.1]{Julin2024}}) 
and then apply Lemma \ref{l:Kato Ponce}. To simplify the notation, we will not distinguish between the original function and its extension. 

It suffices to estimate $R_{\Rdiv}^3=\sum_{2\le m\le 4}\sum_{|\beta| \leq 5-m} 
\nabla \matder^{\beta_1} v \star \cdots \star \nabla \matder^{\beta_{m-1}} v \star \nabla \matder^{\beta_{m}} v$ defined in \eqref{e:def error 1} since $R_{\Rdiv}^{1}$ and $R_{\Rdiv}^{2}$ are easier to handle.  
We deal with the case of $m=2$, i.e., $\sum_{|\beta|\le 3}\nabla\matder^{\beta_1}v\star\nabla\matder^{\beta_2} v$ and we only show the estimates when $|\beta|=3$. From \eqref{e:a priori} and Lemma \ref{l:Kato Ponce}, we see that
\begin{align*}
\|\nabla v\star \nabla \matder^3 v\|_{H^{1/2}(\Omega_t)}
\le{}& C(\|\nabla v\|_{L^\infty(\Omega_t)}\|\nabla \matder^3v\|_{H^{1/2}(\Omega_t)}+ \|\nabla v\|_{W^{1/2,6}(\Omega_t)}\|\nabla \matder^3v\|_{L^{3}(\Omega_t)})\\
\le{}& C(\|\nabla v\|_{L^\infty(\Omega_t)}\|\nabla \matder^3v\|_{H^{1/2}(\Omega_t)}+ \|  v\|_{H^{5/2}(\Omega_t)}\|  \matder^3v\|_{H^{3/2}(\Omega_t)}) \le C\barE(t)^{1/2},\\
\|\nabla\matder^2 v\star \nabla \matder  v\|_{H^{1/2}(\Omega_t)}
\le{}&   C(\|\nabla \matder  v\|_{H^{1/2}(\Omega_t)}\|\nabla \matder^2v\|_{L^{\infty}(\Omega_t)}+ \|\nabla\matder  v\|_{L^{3}(\Omega_t)}\|  \nabla\matder^2v\|_{W^{1/2,6}(\Omega_t)})\\
\le{}& C(1+\|\nabla^2 p\|_{H^{1/2}(\Omega_t)})\barE(t)^{1/2}.
\end{align*}
If $l\ge 4$, the assumption $E_{l-1}(t)\le C$ also ensures that the functions in $H^{3l/2+1}(\Omega_t)$ can be extended by Lemma \ref{l:B by p} and the extension theorem (e.g., {\cite[Proposition 2.1]{Julin2024}}). Then, it follows that $\|\nabla v\star \nabla \matder^l v\|_{H^{1/2}(\Omega_t)} 
\le C(\|\nabla v\|_{L^\infty(\Omega_t)}\|\nabla \matder^lv\|_{H^{1/2}(\Omega_t)}+ \|  v\|_{H^{5/2}(\Omega_t)}\|  \matder^lv\|_{H^{3/2}(\Omega_t)}) \le CE_l(t)^{1/2}$. 
For $1\le j\le l-j\le l-1$, we have $j\le \lfloor l/2\rfloor\le l-2$ due to $l\ge 4$, and obtain
\begin{align*}
&\|\nabla \matder^j v\star \nabla \matder^{l-j} v\|_{H^{1/2}(\Omega_t)}\\
\le{}& C(\|\nabla \matder^j v\|_{L^\infty(\Omega_t)}\|\nabla \matder^{l-j} v\|_{H^{1/2}(\Omega_t)}+ \|\matder^j v\|_{H^{5/2}(\Omega_t)}\|  \nabla \matder^{l-j} v\|_{H^{3/2}(\Omega_t)})\le CE_l(t)^{1/2},
\end{align*}
where we have used the fact that $\|\matder^j v\|_{H^{5/2+\vare}(\Omega_t)}\le E_{l-1}(t)\le C$. Again from the hypothesis $E_{l-1}(t)\le C$, the terms involving the product of more than three items can be controlled since we will have fewer material derivatives in this case. 

To prove the last claim, we first estimate that   
$\|R_{\Rdiv}^{0}\|_{H^{3l/2-1}(\Omega_t)}^2 
\leq C\|\nabla v\|_{L^{\infty}(\Omega_t)}^2\|\nabla v\|_{H^{3l/2-1}(\Omega_t)}^2 \le C\|\nabla v\|_{H^{3l/2-1}(\Omega_t)}^2$.  
By interpolation, we have $\|R_{\Rdiv}^{0}\|_{H^{3l/2-1}(\Omega_t)}^2 \le \vare E_l(t)+C_\vare$ for $ l=5,7,\cdots,$ 
and
$\|R_{\Rdiv}^{0}\|_{H^{3l/2-1}(\Omega_t)}^2\le C E_{l-1}(t)\le C$ for $ l=4,6,\cdots$  
Then we control the case of $k=1$. When $l\ge5$, applying the previous estimates, it holds $\|R^{l-1}_{\Rdiv}\|_{H^{1/2}(\Omega_t)}^2\le CE_{l-1}(t)\le C$. If $l=4$, one has $\|R_{\Rdiv}^{l-1}\|_{H^{1/2}(\Omega_t)}^2 \leq C(1+\|\nabla^2 p\|_{H^{1/2}(\Omega_t)}^2) E_{l-1}(t)\le C$, since $\|\nabla^2 p\|_{H^{1/2}(\Omega_t)}^2\le \|\nabla (H\cdot\nabla H-\matder v)\|_{H^{1/2}(\Omega_t)}^2\le C$. 

We are left with the case of $2\le k\le l-1$. Note that  $R_{\Rdiv}^{l-k}= \sum_{2\le m\le l-k+1}\sum_{|\beta| \leq l-k+2-m} 
\nabla \matder^{\beta_1} v \star \cdots \star \nabla \matder^{\beta_{m-1}} v \star \nabla \matder^{\beta_{m}} v$.
We only estimate the case of $k=m=2$, i.e., $\nabla \matder^{l-2-j} v \star  \nabla \matder^j v$.  
As before, we assume that $0\le j\le l-2-j\le l-2$ and it holds $j\le\lfloor(l-2)/2\rfloor \le l-2,l=4$, and $
j\le\lfloor(l-2)/2\rfloor \le l-3\text{, }l\ge 5$.
We deal with the first case, i.e.,  $\|\nabla \matder v \star  \nabla \matder v\|_{H^{2}(\Omega_t)}^2+\|\nabla \matder^2 v \star  \nabla v\|_{H^{2}(\Omega_t)}^2$, since the same arguments work for $l\ge 5$ ($j\le l-3$ in this case). We deduce that
\begin{align*}
\|\nabla v\star \nabla \matder^{2} v\|_{H^{2}(\Omega_t)}^2 
\le{}& C(\|\nabla v\|_{L^\infty(\Omega_t)}^2\|\nabla \matder^{2}v\|_{H^{2}(\Omega_t)}^2+ \|\nabla  v\|_{H^{3}(\Omega_t)}^2\|\nabla  \matder^{2}v\|_{H^{5/2}(\Omega_t)}^2)\\
\le{}&  
\vare\|\nabla  \matder^{2}v\|_{H^{3}(\Omega_t)}^2+C \|\nabla  \matder^{2}v\|_{H^{2}(\Omega_t)}^2 \le \vare E_l(t)+C_\vare,\\
\|\nabla \matder v\star \nabla \matder v\|_{H^{2}(\Omega_t)}^2 \le{}& C\|\nabla \matder v\|_{L^\infty(\Omega_t)}^2\|\nabla \matder v\|_{H^{2}(\Omega_t)}^2 \le CE_{l-1}(t)\le C.
\end{align*}
The proof is complete.
\end{proof}
 
\begin{lemma}\label{l:err est 2}
Assume that  \eqref{e:a priori} holds for $T>0$. For $l\le 3$, we have
\begin{equation*}
\|R_{\nabla H,H}^l\|_{H^{1/2}(\Omega_t)}^2+\|R_{\nabla H,\nabla H}^l\|_{H^{1/2}(\Omega_t)}^2+\|R_{\nabla^2 H,H}^l\|_{H^{1/2}(\Omega_t)}^2\leq C(1+\|\nabla^2p\|_{H^{1/2}(\Omega_t)}^2)  \barE(t) .
\end{equation*}
Assume further that $\sup_{0\le t< T}E_{l-1}(t)\le C$ for $l\ge 4$, then we have  
\begin{align}
&\|R_{\nabla H,H}^l\|_{H^{1/2}(\Omega_t)}^2+\|R_{\nabla H,\nabla H}^l\|_{H^{1/2}(\Omega_t)}^2+\|R_{\nabla^2 H,H}^l\|_{H^{1/2}(\Omega_t)}^2 \leq C E_l(t),\nonumber\\
&\|R_{\nabla H,H}^0\|_{H^{3l/2-1}(\Omega_t)}^2+\|R_{\nabla H,\nabla H}^0\|_{H^{3l/2-1}(\Omega_t)}^2\leq \vare E_l(t)+C_\vare,\nonumber\\
&\|R_{\nabla^2 H,H}^0\|_{H^{3k/2-1}(\Omega_t)}^2\leq C\|\curl H\|_{H^{\lfloor3l/2+1/2\rfloor}(\Omega_t)}^2,\nonumber\\
&\|R_{\nabla H,H}^{l-k}\|_{H^{3k/2-1}(\Omega_t)}^2+\|R_{\nabla H,\nabla H}^{l-k}\|_{H^{3k/2-1}(\Omega_t)}^2+\|R_{\nabla^2 H,H}^{l-k}\|_{H^{3k/2-1}(\Omega_t)}^2\le \vare E_l(t)+C_\vare,\label{e:err est 2-5}
\end{align} 
for $k\in \N ,1\le k< l$. In the above, $\vare>0$ is a constant small enough. 
\end{lemma} 
\begin{proof}
We note that $R^l_{\nabla^2 H,H}$ contains all the highest-order terms in $R^l_{\nabla H,H}$ and $R^l_{\nabla H,\nabla H}$ and we focus on the estimate for $R^l_{\nabla^2 H,H}$. To control $R^3_{\nabla^2 H,H}$ in the case of $l\le 3$, we recall $R^3_{\nabla^2 H, H}$ in Lemma \ref{l:nab^2H,H}. 
From \eqref{e:a priori}, we have $\|\nabla^{4}\curl H\star H\star \cdots\star H\|_{H^{1/2}(\Omega_t)}^2\le C\|H\|_{H^6(\Omega_t)}\le C\barE(t)$, 
and $\|\nabla^2\matder^2 v\star \nabla F_2\star F_3\|_{H^{1/2}(\Omega_t)}^2\le C\barE(t)$,  
as in Lemma \ref{l:err est 1}. The leading terms in $R^3_{\nabla^2 H, H}$ have been controlled, and the estimates of the lower-order terms follow from the same arguments as in Lemma \ref{l:err est 1}.

As for $l\ge 4$, to prove the first result, it is sufficient to bound $\nabla^{l+1}\curl v\star H\star \cdots\star H$ and $\nabla^{l+1}\curl H\star H\star \cdots\star H$ since the other terms are either simpler or have already been estimated in Lemma \ref{l:err est 1}. From the assumption,  $\|v\|_{H^{\lfloor3l/2\rfloor}(\Omega_t)}+\|H\|_{H^{\lfloor3l/2\rfloor}(\Omega_t)}\le C$. As before, we extend the functions and estimate as in Lemma \ref{l:err est 1} to obtain 
\begin{align*}
&\|\nabla^{l+1}\curl v\star H\star \cdots\star H\|_{H^{1/2}(\Omega_t)}\\
\le{}& C(\|H\star \cdots\star H\|_{L^\infty(\Omega_t)}\|\nabla^{l+1}\curl v\|_{H^{1/2}(\Omega_t)}+ \|H\star \cdots\star H\|_{W^{1/2,6}(\Omega_t)}\|\nabla^{l+1}\curl v\|_{L^{3}(\Omega_t)})\\
\le{}& C\| v\|_{H^{l+5/2}(\Omega_t)}\le CE_l(t)^{1/2}.
\end{align*}
In the last step,  the condition $l\ge 4$ implies $5/2+l\le \lfloor3(l+1)/2\rfloor$ and therefore, it holds $\| v\|_{H^{l+5/2}(\Omega_t)}\le \| v\|_{H^{\lfloor3(l+1)/2\rfloor}(\Omega_t)}$.

Next, to verify \eqref{e:err est 2-5}, we shall control $\|R_{\nabla^2 H,H}^{l-k}\|_{H^{3k/2-1}(\Omega_t)}^2$ for $1\le k\le l-1$. We concentrate on the estimate for  $\|\nabla^{l-k+1}\curl v\star H\star \cdots\star H\|_{H^{3k/2-1}(\Omega_t)}^2$. This time we obtain for $1\le k<l$ that  
\begin{equation*}
\|\nabla^{l-k+1}\curl v\star H\star \cdots\star H\|_{H^{3k/2-1}(\Omega_t)}^2 
\le C\| v\|_{H^{3l/2+1/2}(\Omega_t)}.
\end{equation*}
By interpolation, it holds  $\|v\|_{H^{3l/2+1/2}(\Omega_t)}^2\le \vare \|v\|_{H^{\lfloor3(l+1)/2\rfloor}(\Omega_t)}^2+C_\vare \| v\|_{H^{\lfloor3l/2\rfloor}(\Omega_t)}^2\le \vare E_l(t)+C_\vare$.
Finally, to obtain the last two estimates, we only need to bound the most difficult term  $R_{\nabla^2 H,H}^0=(H\cdot\nabla)\curl H$. Since $l\ge 4$, we have $\|R_{\nabla^2 H,H}^0\|_{H^{3l/2-1}(\Omega_t)}^2
\le C\|H\|_{H^{\lfloor3l/2\rfloor}(\Omega_t)}^2\|\curl H\|_{H^{\lfloor3l/2+1/2\rfloor}(\Omega_t)}^2\le C\|\curl H\|_{H^{\lfloor3l/2+1/2\rfloor}(\Omega_t)}^2$, 
and the proof is complete.
\end{proof}
 
\begin{lemma}\label{l:err est 3}
Assume that  \eqref{e:a priori} holds for $T>0$. We have $\|R_{\Rbulk}^l\|_{L^{2}(\Omega_t)}^2 \leq C( 1+\|\nabla p\|_{H^{3/2}(\Omega_t)}^2)  \barE(t)$ for $l\le 3$. Assume further that $\sup_{0\le t< T}E_{l-1}(t)\le C$ for $l\ge 4$. Then, it follows that $\|R_{\Rbulk}^l\|_{L^{2}(\Omega_t)}^2 \leq C E_l(t)$,
and $\|R_{\Rbulk}^{l-k}\|_{H^{3(k-1)/2}(\Omega_t)}^2 \le C$ for $ k\in \N ,1\le k\le l-1$.
\end{lemma}
\begin{proof}
To prove the first claim, we  estimate  
\begin{equation*} 
R_{\Rbulk}^3=\sum_{1\le m\le 4}\sum_{\substack{|\beta| \leq 3,|\alpha|\le 1\\\beta_1,\cdots,\beta_{m-1}\ge 1}} a_{\alpha,\beta}(\nabla v) \nabla \matder^{\beta_1} v \star \cdots \star \nabla \matder^{\beta_{m-1}} v \star \nabla^{\alpha_1} \matder^{\alpha_2+\beta_{m}} v.
\end{equation*}
If $m=1$, we consider the case of $|\beta|=\beta_1=3$ and $|\alpha|=1$. We should control $a(\nabla v)  \matder^{4} v+b(\nabla v)\nabla  \matder^{3} v$. From the hypothesis \eqref{e:a priori}, it is clear that $\|a(\nabla v)  \matder^{4} v\|_{L^2(\Omega_t)}^2+\|b(\nabla v)\nabla  \matder^{3} v\|_{L^2(\Omega_t)}^2\le C\barE(t)$.
For $m=2,|\beta|=3$ and $|\alpha|=1$, we show the estimates of $a(\nabla v)\nabla \matder v\star\matder^3v$ and $b(\nabla v)\nabla \matder^2 v\star\matder^2v$. Choosing $1/p+1/q=1/2,p=3/\delta$ with $\delta>0$ small enough, we see that $\|\nabla^2H\|_{L^{q}(\Omega_t)}^2\le C\|H\|_{H^{5/2+\delta}(\Omega_t)}^2$,
\begin{align*}
\|a(\nabla v)\nabla \matder v\star\matder^3v\|_{L^2(\Omega_t)}^2 
\le{}& C\|\nabla^2p+\nabla H\star \nabla H+H\star\nabla^2H\|_{L^{q}(\Omega_t)}^2\|\matder^3v\|_{H^{3/2}(\Omega_t)}^2\\
\le{}&  
C(1+\|\nabla^2p\|_{H^{1/2}(\Omega_t)}^2)\barE(t),
\end{align*} 
and $
\|a(\nabla v)\nabla \matder^2 v\star\matder^2v\|_{L^2(\Omega_t)}^2
\le{} 
C\|\matder^2v\|_{L^2(\Omega_t)}^2\barE(t)$.
To control    $\|\matder^2v\|_{L^2(\Omega_t)}^2$, from the boundedness $\|\laplace p\|_{H^1(\Omega_t)}\le C$ and using \eqref{e:laplace p}, \eqref{l:lapDtp}, \eqref{e:def error 2},  together with \eqref{e:har ext 2}, we obtain
\begin{align*}
\|\matder^2v\|_{L^2(\Omega_t)}^2 
\le{}&\|\nabla \matder  p\|_{L^2(\Omega_t)}^2+\|[\matder,\nabla] p\|_{L^2(\Omega_t)}^2+\|\matder H\star \nabla H+H\star\matder\nabla H\|_{L^2(\Omega_t)}^2\\
\le{}&\|\laplace\matder  p\|_{L^2(\Omega_t)}^2+\|\matder p\|_{H^{1/2}(\parOmega_t)}^2+\|\nabla v\star\nabla p\|_{L^2(\Omega_t)}^2\\
&+\|H\star \nabla v\star \nabla H+H\star\nabla v\star\nabla H+H\star\nabla^2 v\star H\|_{L^2(\Omega_t)}^2\\
\le{}&\|\dive\dive(v\otimes\nabla p)\|_{L^2(\Omega_t)}^2+\|\nabla p\|_{L^2(\Omega_t)}^2+C\\
&+\|\dive R^0_{\Rbulk}+\nabla^2 v\star \nabla H\star H+\nabla v\star \nabla H\star \nabla H\\
&\quad \ \ +\nabla^2 H\star \nabla v\star H+v\star\nabla^2 H\star\nabla H\|_{L^2(\Omega_t)}^2 \\
\le{}&\|\partial_j\partial_i(v^i\partial_j p)\|_{L^2(\Omega_t)}^2+\|\nabla p\|_{H^1(\Omega_t)}^2+C 
\le C(1+\|\nabla p\|_{H^1(\Omega_t)}^2).
\end{align*}
In the case of $m=3$ and $m=4$, we estimate in the same fashion, and obtain $\|R_{\Rbulk}^l\|_{L^{2}(\Omega_t)}^2 \leq C(1+\|\nabla p\|_{H^{3/2}(\Omega_t)}^2) \barE(t)$,
as desired.

To control $R_{\Rbulk}^l$ for $l\ge 4$,  we focus on the case of $|\beta|=l$ and $|\alpha|=1$. If $m=1$, it holds $\|a(\nabla v) \matder^{l+1} v+b(\nabla v)\nabla\matder^{l}v\|_{L^{2}(\Omega_t)}^2 \le CE_{l-1}(t)\le C$. 
Next, we handle the product of functions as follows. We simply assume  $\alpha_2=1$ since the material derivative $\matder$ is $1/2$-higher than the spatial derivative. If $ 1\le j\le l+1-j\le l$, it follows that $1\le j\le\lfloor (l+1)/2 \rfloor\le l-2$, and we have $\|a(\nabla v) \nabla \matder^{j} v \star  \matder^{l+1-j} v\|_{L^{2}(\Omega_t)}^2 \le C \|  \nabla \matder^{j} v  \|_{L^{\infty}(\Omega_t)}^2\|  \matder^{l+1-j} v  \|_{L^{2}(\Omega_t)}^2\le CE_l(t)$. 
If $1\le l+1-j<j \le l$, we find that $\lfloor(l+1)/2\rfloor+1\le j$ and $1\le l+1-j\le l-2$. Then, we obtain $\|a(\nabla v) \nabla \matder^{j} v \star  \matder^{l+1-j} v\|_{L^{2}(\Omega_t)}^2
\le C \|\matder^{j} v  \|_{H^{1}(\Omega_t)}^2\|  \matder^{l+1-j} v  \|_{H^{3/2+\vare}(\Omega_t)}^2\le CE_l(t)$. 
The others can be estimated in the same way.

We are left with the last claim. For $k=1$, it follows by applying the above estimates with $l-1$ if $l\ge 5$. As $k=1$ and $l=4$, it follows from the hypothesis that $E_3(t)\le C$. Therefore, $\|\nabla p\|_{H^1(\Omega_t)}^2\le C\|H\cdot\nabla H-\matder v\|_{H^1(\Omega_t)}^2\le C$. This concludes the proof for $k=1$. Assume that $2\le k\le l-1$ and we shall control $\|R_{\Rbulk}^{l-k}\|_{H^{3(k-1)/2}(\Omega_t)}^2$ defined in \eqref{e:def error 2}: 
\begin{equation*} 
R_{\Rbulk}^{l-k}=\sum_{1\le m\le l-k+1}\sum_{\substack{|\beta| \leq l-k,|\alpha|\le 1,\beta_1,\cdots,\beta_{m-1}\ge 1}} a_{\alpha,\beta}(\nabla v) \nabla \matder^{\beta_1} v \star \cdots \star \nabla \matder^{\beta_{m-1}} v \star \nabla^{\alpha_1} \matder^{\alpha_2+\beta_{m}} v. 
\end{equation*}
If $m=1,|\beta|=l-k$ and $|\alpha|=1$, it is clear that $\|a(\nabla v)\matder^{l+1-k}v+b(\nabla v)\nabla\matder^{l-k}v\|_{H^{3(k-1)/2}(\Omega_t)}^2\le CE_{l-1}(t)\le C$.
To bound the product of functions, e.g.,  $m=2,|\beta|=l-k,|\alpha|=1$ and $1\le j\le l-k-j\le l-k-1$, we note that $1\le j\le\lfloor (l-k)/2\rfloor$ and 
\begin{equation*}
\|a(\nabla v)\|_{H^{3k/2-1/2}(\Omega_t)}^2\le C\|\nabla v\|_{L^{\infty}(\Omega_t)}^2 \cdots \|\nabla v\|_{L^\infty(\Omega_t)}^2\|v\|_{\lfloor3l/2\rfloor(\Omega_t)}^2 
\le C.
\end{equation*}
This, combined with the Sobolev embedding and Lemma \ref{l:Kato Ponce}, we deduce  that
\begin{align*}
&\|a(\nabla v) \nabla \matder^{j} v \star \nabla^{\alpha_1} \matder^{\alpha_2+l-k-1}v\|_{H^{3(k-1)/2}(\Omega_t)}^2\\
\le{}& C\|a(\nabla v)\|_{W^{3(k-1)/2,6}(\Omega_t)}^2\| \nabla \matder^{j} v \star \nabla^{\alpha_1} \matder^{\alpha_2+l-k-1}v\|_{L^{3}(\Omega_t)}^2\\
& +C\|a(\nabla v)\|_{L^{\infty}(\Omega_t)}^2\| \nabla \matder^{j} v \star \nabla^{\alpha_1} \matder^{\alpha_2+l-k-1}v\|_{H^{3(k-1)/2}(\Omega_t)}^2\\ 
\le{}& C\|\nabla \matder^{j} v \|_{H^{3k/2-1/2}(\Omega_t)}^2\| \nabla^{\alpha_1} \matder^{\alpha_2+l-k-1}v\|_{L^{3}(\Omega_t)}^2\\
& + C\|\nabla \matder^{j} v \|_{L^{\infty}(\Omega_t)}^2\| \nabla^{\alpha_1} \matder^{\alpha_2+l-k-1}v\|_{H^{3(k-1)/2}(\Omega_t)}^2  
\le C,
\end{align*}
where we have used the fact that $\|\nabla \matder^{j} v \|_{H^{3k/2-1/2}(\Omega_t)}^2+\|\nabla \matder^{j} v \|_{L^{\infty}(\Omega_t)}^2\le C(\|\matder^{j} v \|_{H^{3k/2+1/2}(\Omega_t)}^2+\| \matder^{j} v \|_{H^{5/2+\vare}(\Omega_t)}^2)$ 
for $\vare>0$ small enough.   Thus, the proof is complete since the other terms can be estimated by using similar arguments.
\end{proof} 
 
\begin{lemma}\label{lem-err-est-3}
Assume that  \eqref{e:a priori} holds for $T>0$. We have $\|R^l_p\|_{H^{1/2}(\parOmega_t)}^2\le C\left( 1+\|\nabla p\|_{H^2(\Omega_t)}^2\right) \barE(t)$ for $l\le 3$.
Assume further that $\sup_{0\le t< T}E_{l-1}(t)\le C$ for $l\ge 4$.  Then it follows that $\|R^l_p\|_{H^{1/2}(\parOmega_t)}^2 \leq C E_l(t)$, 
and  $\|R^{l-k}_p\|_{H^{3k/2-1}(\parOmega_t)}^2 \leq \vare E_l(t)+C_\vare$ for $k\in \N ,1\le k\le l-1$ with $\vare>0$ small enough.
\end{lemma}
\begin{proof}
It is sufficient to show the estimate for $l=3$ since the other cases are easier. Recall the definition of $R^3_p$, 
we have 
\begin{align*}
\|\tangrad p\cdot\matder^3 v\|_{H^{1/2}(\parOmega_t)}^2&\le C(\|\tangrad p\|_{W^{1/2,4}(\parOmega_t)}^2\|\matder^3 v\|_{L^{4}(\parOmega_t)}^2+\|\tangrad p\|_{L^{4}(\parOmega_t)}^2\|\matder^3 v\|_{W^{1/2,4}(\parOmega_t)}^2)\\
&\le C\|\nabla p\|_{H^{3/2}(\Omega_t)}^2\barE(t),
\end{align*} 
we have used the fact that
$\|\tangrad^2 p\|_{L^2(\parOmega_t)}^2 
\le C(\|\nabla p\|_{H^1(\parOmega_t)}^2+\|\nabla p\star B\|_{L^2(\parOmega_t)}^2)\le C\|\nabla p\|_{H^1(\parOmega_t)}^2$,
and the trace theorem. Similarly, to deal with the term $-|B|^2\matder^3 v\cdot\nu$, we have $\|\matder^3 v\cdot\nu\|_{L^{2}(\parOmega_t)}^2\le C\|\matder^3 v\|_{H^{1}(\Omega_t)}^2\le C\barE(t)$, and $\|-|B|^2\matder^3 v\cdot\nu\|_{H^{1/2}(\parOmega_t)}^2\le C\||B|^2\|_{H^1(\parOmega_t)}^2\|\matder^3 v\cdot\nu\|_{H^{1}(\parOmega_t)}^2 \le C(1+\|\nabla p\|_{H^{1}(\Omega_t)}^2)\barE(t)$ by \eqref{e:a priori}. 
Again from \eqref{e:a priori}, it follows that  
\begin{align*}
&\|a_8(\nu,\nabla v)\star\nabla^2\matder^2v\|_{H^{1/2}(\parOmega_t)}^2 \le C(\|\nabla^2\matder^2v\|_{H^{1/2}(\parOmega_t)}^2+\|a_8(\nu,\nabla v)\|_{W^{1/2,4}(\parOmega_t)}^2\|\nabla^2\matder^2v\|_{H^{1/2}(\parOmega_t)}^2),\\ 
&\|a_9(\nu,\nabla v)\star\nabla\matder^2v\star B\|_{H^{1/2}(\parOmega_t)}^2  
\le C( \|\nabla\matder^2v\|_{W^{1/2 ,4}(\parOmega_t)}^2\|B\|_{L^4(\parOmega_t)}^2+\|B\|_{H^{1/2 }(\parOmega_t)}^2\|\nabla\matder^2v\|_{L^\infty(\parOmega_t)}^2),\\ 
&\|a_{10}(\nu,\nabla v)\star\nabla\matder^2v\star \nabla^2v\|_{H^{1/2 }(\parOmega_t)}^2  
\le C\|\nabla^2 v\|_{H^{1/2 }(\parOmega_t)}^2(\|\nabla\matder^2v\|_{W^{1/2 ,4}(\parOmega_t)}^2+\|\nabla\matder^2v\|_{L^\infty(\parOmega_t)}^2),  
\end{align*}
and they can be controlled by $C\barE(t)$. Moreover,
\begin{align*}
&\|a_{11}(\nu,\nabla v)\star\nabla^2\matder v\star \nabla\matder v\|_{H^{1/2 }(\parOmega_t)}^2\\ 
\le{}& C(\|\nabla^2\matder v\|_{L^\infty(\Omega_t)}^2\|\nabla\matder v\|_{H^1(\Omega_t)}^2 +\|\nabla^2\matder v\|_{W^{1,6}(\Omega_t)}^2\|\nabla\matder v\|_{L^3(\Omega_t)}^2)\\ \le{}& C\|\nabla(-\nabla p+H\cdot\nabla H)\|_{H^{1}(\Omega_t)}^2\|\matder v\|_{H^{4}(\Omega_t)}^2 \le C(1+\|\nabla^2p\|_{H^1(\Omega_t)}^2)\barE(t),
\end{align*}
and the other terms can be estimated in the same way. For $l\ge 4$, the proof is similar to \cite[Lemma 5.8]{Julin2024}, and we omit the details.
\end{proof}

Applying the above error estimates and recalling Proposition \ref{p:d/dt} as well as \eqref{e:p H^3}, we conclude this section by presenting the following improved version of Proposition \ref{p:d/dt}.

\begin{proposition}\label{p:d/dt 2}
Assume that  \eqref{e:a priori} holds for $T>0$. Then, we have $\frac{d}{dt}\barenergy(t)
\le C\barE(t)$, 
where $C$ depends on $T,\mathcal{N}_T,\mathcal{M}_T,\|v_0\|_{H^6(\Omega_0)},\|H_0\|_{H^6(\Omega_0)}$, and $\|\meancurv_{\parOmega_0}\|_{H^5(\parOmega_0)}$. For $l\ge 4$, assume further that $\sup_{0\le t< T}E_{l-1}(t)\le C$, then we have $\frac{d}{dt}\energy_l(t)
\le CE_l(t)$, where the constant $C$ depends on $T,\mathcal{N}_T,\mathcal{M}_T$, and $\sup_{0\le t< T}E_{l-1}(t)$.
\end{proposition}

\section{Closing the energy estimates and proving the main theorems}\label{s:energy est}

In this section, we close the energy estimates and prove Theorem \ref{t:main t}. We  introduce the energy functional
\begin{align*}
\tilde{\energy}(t)\coloneqq {}&\frac{1}{2}\sum_{k=1}^{3}\left( \|\matder^{k+1}v\|_{L^2(\Omega_t)}^2+\|\matder^{k+1}H\|_{L^2(\Omega_t)}^2+ \|\tangrad(\matder^{k}v\cdot \nu)\|_{L^2(\parOmega_t)}^2\right) \\
&+\frac{1}{2}\left( \|\curl v\|_{H^{5}(\Omega_t)}^2+\|\curl H\|_{H^{5}(\Omega_t)}^2\right) +1,\\
\tilde{\energy}_l(t)\coloneqq {}&\frac{1}{2}\left( \|\matder^{l+1}v\|_{L^2(\Omega_t)}^2+\|\matder^{l+1}H\|_{L^2(\Omega_t)}^2+\|\tangrad(\matder^{l}v\cdot \nu)\|_{L^2(\parOmega_t)}^2\right) \\
&+\frac{1}{2}\left(\|\curl v\|_{H^{\lfloor (3l+1)/2\rfloor}(\Omega_t)}^2+\|\curl H\|_{H^{\lfloor(3l+1)/2\rfloor}(\Omega_t)}^2\right) +1,\quad l\ge 4.
\end{align*}
Note that from the a priori assumptions \eqref{e:a priori}, it holds $\|\curl v\|_{L^2(\Omega_t)}^2+\|\curl H\|_{L^2(\Omega_t)}^2\le C$.
By interpolation, we have $\tilde{\energy}(t) \leq C(\barenergy(t)+1)$ and $\tilde{\energy}_l(t) \leq C(\energy_l(t)+1)$ for $l\ge 4$. 

We will apply the following div-curl estimates in \cite[Section 3.1]{Julin2024}.
\begin{lemma}
Let the integer $l \geq 2$ and assume that $\|B_{\parOmega} \|_{H^{3l/2-1}(\parOmega)} \leq C$. Let  $j\in \{ 5/2, 3, 7/2, 4,\cdots, 3l/2 \}$ and $k\in \{ 3/2, 5/2, 3, 7/2, 4 ,\cdots, 3l/2 \}$. 
Then, for all smooth vector fields $F$,  it holds
\begin{align}
&\|F\|_{H^k(\Omega)} \leq C \Big(\|F_n \|_{H^{k-1/2}(\parOmega)}+\|F\|_{L^2(\Omega)}+\|\dive  F\|_{H^{k-1}(\Omega)}+\|\curl F\|_{H^{k-1}(\Omega)} \Big),\label{e:div-curl 1}\\
&\|F\|_{H^j(\Omega)} \leq  C\Big(  \|\laplace_{\parOmega} F_n \|_{H^{j-5/2}(\parOmega)}+\|F\|_{L^2(\Omega)}+\|\dive F\|_{H^{j-1}(\Omega)}+\|\curl F\|_{H^{j-1}(\Omega)} \Big) ,\label{e:div-curl 2}\\
&\|F\|_{H^{\lfloor(3(l+1))/2\rfloor}(\Omega)}
\leq  C\Big(\|\surflaplace F_n \|_{H^{\lfloor(3l-2)/2\rfloor}(\parOmega)}+ (1+\|B\|_{H^{3l/2}(\parOmega)})\|F\|_{L^{\infty}(\Omega)}\nonumber\\
&\qquad\qquad\qquad\qquad\qquad+\|\dive F\|_{H^{\lfloor(3l+1)/2\rfloor-1}(\Omega)}+\|\curl F\|_{H^{\lfloor(3l+1)/2\rfloor-1}(\Omega)}\Big).\label{e:div-curl 3}
\end{align} 
  
\end{lemma} 
\begin{proposition}\label{p:ene est 1}
Assume that $\parOmega_t\in H^{3+\delta}(\Gamma)$ with $\delta>0$ small enough. Assume that  $\|p\|_{H^{3}(\Omega_t)}+\|v\|_{H^{4}(\Omega_t)}+\|H\|_{H^{4}(\Omega_t)} \le C_0$.
Then we have $\barE(t)+\|B\|_{H^{9/2}(\parOmega_t)}^2\le C(1+\barenergy(t))$,
where the constant $C$ depends on $\mathcal{M}_t,\|h(\cdot,t)\|_{H^{3+\delta}(\Gamma)},\|p\|_{H^{3}(\Omega_t)},\|v\|_{H^{4}(\Omega_t)}$, and $\|H\|_{H^{4}(\Omega_t)}$.
\end{proposition} 
\begin{proof}
We shall show that $\barE(t)\le C\tilde{\energy}(t)$. We need to control $\|\matder^{4-k}v\|_{H^{3k/2}(\Omega_t)}^2,\|\matder^{4-k}H\|_{H^{3k/2}(\Omega_t)}^2, k\le 3,\|  v\|_{H^6(\Omega_t)}^2$, and $ \| H\|_{H^6(\Omega_t)}^2$.   Recalling \eqref{e:claim},  
it is sufficient to control $\|\matder^{3}v\|_{H^{3/2}(\Omega_t)}^2$, $\|\matder^{2}v\|_{H^{3}(\Omega_t)}^2$, $\|\matder v\|_{H^{9/2}(\Omega_t)}^2$, $\|v\|_{H^6(\Omega_t)}^2$, and $\| H\|_{H^6(\Omega_t)}^2$. 
We divide the proof into three steps.

\textbf{Step 1.} We control $\|\matder^{3}v\|_{H^{3/2}(\Omega_t)}^2$. Recalling that  $\|\tilde\nu\|_{H^{5/2+\delta}(\Omega_t)}\le C$, we have
\begin{align*}
&\|\matder^3v\cdot \nu\|_{L^2(\parOmega_t)}^2\\ 
\le{}&|\int_{\Omega_t} (\matder^3v\cdot \nu)\dive\matder^3vdx|+|\int_{\Omega_t} \nabla\matder^3v\star\matder^3vdx|+|\int_{\Omega_t} \matder^3v\star\nabla\nu\star\matder^3vdx|\\
\le{}& C(\|\matder^3v\|_{L^{2}(\Omega_t)}^2+\|\dive\matder^3v\|_{L^{2}(\Omega_t)}^2 +\|\nabla\matder^3v\|_{L^{2}(\Omega_t)}\|\matder^3v\|_{L^{2}(\Omega_t)})\\
\le{}& \vare\|\nabla\matder^3v\|_{L^{2}(\Omega_t)}^2+C_\vare\tilde{\energy}(t)+C\|\dive\matder^3v\|_{L^{2}(\Omega_t)}^2.
\end{align*}
This, combined with Lemmas \ref{l:formu 5} and \ref{l:err est 1}, and \eqref{e:div-curl 1}, it follows that $\|\matder^{3}v\|_{H^{3/2}(\Omega_t)}^2\le C(\|\matder^{3}v\cdot\nu\|_{H^{1}(\parOmega_t)}^2+\|\matder^{3}v\|_{L^{2}(\Omega_t)}^2+\|\dive \matder^{3}v\|_{H^{1/2}(\Omega_t)}^2+\|\curl\matder^{3}v\|_{H^{1/2}(\Omega_t)}^2)$, 
and therefore,
\begin{align*}
\|\matder^{3}v\|_{H^{3/2}(\Omega_t)}^2 \le C(&\tilde{\energy}(t)+\|R^2_{\Rdiv}\|_{H^{1/2}(\Omega_t)}^2+\|R^2_{\nabla H,\nabla H}\|_{H^{1/2}(\Omega_t)}^2+\|R^2_{\nabla^2H,H}\|_{H^{1/2}(\Omega_t)}^2).
\end{align*}
To control $\|R^2_{\nabla H,\nabla H}\|_{H^{1/2}(\Omega_t)}^2$, we estimate as follows. Indeed, by the assumption, applying  Young's inequality and Lemma \ref{l:Kato Ponce}, we  obtain $\|\nabla^2\matder v\star\nabla H\star H\|_{H^{1/2}(\Omega_t)}^2+\|\nabla\matder v\star\nabla^2 H\star H\|_{H^{1/2}(\Omega_t)}^2 \le  C\|\matder v\|_{H^{3}(\Omega_t)}^2\|H\|_{H^{3}(\Omega_t)}^4$ and 
\begin{equation*}
\|\matder v\|_{H^{3}(\Omega_t)}^2\|H\|_{H^{3}(\Omega_t)}^4\le \vare \barE(t)+C_\vare\|p\|_{H^{1}(\Omega_t)}^2+C_\vare\|H\cdot\nabla H\|_{L^{2}(\Omega_t)}^2\le \vare \barE(t)+C_\vare.
\end{equation*}  
As for $\|R^2_{\nabla^2H,H}\|_{H^{1/2}(\Omega_t)}^2$, we recall Lemma \ref{l:nab^2H,H}, and we handle the most difficult term, i.e., $\|\nabla^3\curl H\star H\star H\star H\|_{H^{1/2}(\Omega_t)}^2 
\le C\|\curl H\|_{H^{4}(\Omega_t)}^2\le C \tilde{\energy}(t)$. 
Again by the Young's inequality and  Lemma \ref{l:Kato Ponce}, we can control $\|R^2_{\Rdiv}\|_{H^{1/2}(\Omega_t)}^2$. In fact, we have
\begin{align*}
&\|\nabla\matder^2 v\star \nabla v\|_{H^{1/2}(\Omega_t)}^2+\|\nabla\matder v\star \nabla\matder v\|_{H^{1/2}(\Omega_t)}^2\\
\le{}& C\|v\|_{H^{2}(\Omega_t)}^2\|\matder^2 v\|_{H^{2}(\Omega_t)}^2+C\|\nabla\matder v\|_{L^{3}(\Omega_t)}^2\|\nabla\matder v\|_{H^{3/2}(\Omega_t)}^2\\
\le{}
&(\vare\|\matder v\|_{H^{9/2}(\Omega_t)}^2+C_\vare\|\matder v\|_{L^{2}(\Omega_t)}^2)(\|\nabla^2p\|_{L^{3}(\Omega_t)}^2+\|\nabla(H\cdot\nabla H)\|_{L^{3}(\Omega_t)}^2)\\
&+\vare\|\matder^2 v\|_{H^{3}(\Omega_t)}^2+C_\vare\|\matder^2 v\|_{L^{2}(\Omega_t)}^2\le C_\vare \tilde{\energy}(t)+\vare \barE(t).
\end{align*}
Combining the above estimates, we obtain $\|\matder^{3}v\|_{H^{3/2}(\Omega_t)}^2
\le \vare \barE(t)+ C_\vare\tilde{\energy}(t)$.

\textbf{Step 2.} We estimate $\|\matder^{2}v\|_{H^{3}(\Omega_t)}^2$ and $\|\matder v\|_{H^{9/2}(\Omega_t)}^2$. Applying Lemma \ref{l:formu 5} and \eqref{e:div-curl 2}, it holds
\begin{align*}
\|\matder v\|_{H^{9/2}(\Omega_t)}^2 
\le{}& C\tilde{\energy}(t)+C( \|\surflaplace(\matder v\cdot \nu)\|_{H^{2}(\parOmega_t)}^2+\|\nabla v\star\nabla v\|_{H^{7/2}(\Omega_t)}^2\\
&\quad \quad \quad \quad \ \ +\|\nabla H\star \nabla H\|_{H^{7/2}(\Omega_t)}^2+\|H\star \nabla\curl H\|_{H^{7/2}(\Omega_t)}^2)\\
\le{}& C\tilde{\energy}(t)+C \|\surflaplace(\matder v\cdot \nu)\|_{H^{2}(\parOmega_t)}^2,\\
\|\matder^{2}v\|_{H^{3}(\Omega_t)}^2 
\le{}& C( \|\surflaplace(\matder^{2}v\cdot \nu)\|_{H^{{1}/{2}}(\parOmega_t)}^2+\|R^1_{\Rdiv}\|_{H^{2}(\Omega_t)}^2+\|R^1_{\nabla H,\nabla H}\|_{H^{2}(\Omega_t)}^2\\
&\quad +\|R^1_{\nabla^2 H,H}\|_{H^{2}(\Omega_t)}^2)+C\tilde{\energy}(t).
\end{align*}
We control $\|R^1_{\Rdiv}\|_{H^{2}(\Omega_t)}^2$ by the bilinear inequality,  $\|\nabla\matder v\star \nabla v\|_{H^{2}(\Omega_t)}^2\le C \|\matder v\|_{H^{3}(\Omega_t)}^2\|v\|_{H^{3}(\Omega_t)}^2\le \vare \barE(t)+ C_\vare\tilde{\energy}(t)$.
For $\|R^1_{\nabla H,\nabla H}\|_{H^{2}(\Omega_t)}^2$, it holds that $\|\nabla^2 v\star \nabla H\star H\|_{H^{2}(\Omega_t)}^2+\|\nabla v\star \nabla H\star\nabla H\|_{H^{2}(\Omega_t)}^2\le C\|v\|_{H^{4}(\Omega_t)}^2\|H\|_{H^{3}(\Omega_t)}^4$
from the assumption. 
Then, the estimate for $\|R^1_{\nabla^2 H,H}\|_{H^{2}(\Omega_t)}^2$  follows since $\|\nabla^2\curl v\star H\star H\|_{H^{2}(\Omega_t)}^2\le C\tilde{\energy}(t)$. 

We are left with $\|\surflaplace(\matder^{2}v\cdot \nu)\|_{H^{1/2}(\parOmega_t)}^2$ and $\|\surflaplace(\matder v\cdot \nu)\|_{H^{2}(\parOmega_t)}^2$. We focus on the estimate of $\|\surflaplace(\matder^{2}v\cdot \nu)\|_{H^{1/2}(\parOmega_t)}^2$. 
Recalling that from Lemma \ref{l:Dt^lp}, we have $\matder^3p=-\surflaplace(\matder^{2}v\cdot \nu)+R^2_p$. Since $\|R^2_p\|_{H^{1/2}(\parOmega_t)}^2$ is easier to control than $\|\matder^3 p\|_{H^{1/2}(\parOmega_t)}^2$, we only bound $\|\matder^3 p\|_{H^{1/2}(\parOmega_t)}^2$. By the definition of $H^{1/2}(\parOmega)$, it holds $\|\matder^3 p\|_{H^{1/2}(\parOmega_t)}^2\le C\|\matder^3 p\|_{L^{2}(\parOmega_t)}^2+C\|\nabla\matder^3 p\|_{L^{2}(\Omega_t)}^2$.
Applying (3) in Lemma \ref{l:formu 4}, for the first term, we have
\begin{equation*}
\|\matder^3 p\|_{L^{2}(\parOmega_t)}^2 
\le C \|\sum_{1\le m\le 3}\sum_{\substack{|\beta|\le 3-m,|\alpha|\le 1}}a_{\alpha,\beta}(\nu,B)\tangrad^{1+\alpha_1}\matder^{\beta_1}v\star\cdots\star \tangrad^{1+\alpha_m}\matder^{\beta_m}v \|_{L^{2}(\parOmega_t)}^2.
\end{equation*}
For $m=1$, from $\|B\|_{L^{\infty}(\parOmega_t)}\le C$, we control $a(\nu,B)\tangrad^{2}\matder^{2}v$ by the trace theorem and interpolation: $\|a(\nu,B)\tangrad^{2}\matder^{2}v\|_{L^{2}(\parOmega_t)}^2\le C \|\matder^{2}v\|_{H^{5/2}(\Omega_t)}^2\le \vare \barE(t)+C_\vare \tilde{\energy}(t)$.
The other cases are either simpler or similar. As for $\|\nabla\matder^3 p\|_{L^{2}(\Omega_t)}^2$, it follows that $\|\nabla\matder^3 p\|_{L^{2}(\Omega_t)}^2 \le C\tilde{\energy}(t)+C\|\matder^3(H\cdot\nabla H)\|_{L^{2}(\Omega_t)}^2+C\|[\nabla,\matder^3] p\|_{L^{2}(\Omega_t)}^2$. 
To control $\|\matder^3(H\cdot\nabla H)\|_{L^{2}(\Omega_t)}^2$, again by interpolation, we see that $\|\nabla^2\matder^2v\star H\star H\|_{L^{2}(\Omega_t)}^2+\|\nabla^2\matder v\star H\star H\|_{L^{2}(\Omega_t)}^2  
\le \vare \barE(t)+C_\vare \tilde{\energy}(t)$, 
and we estimate $\|[\nabla,\matder^3] p\|_{L^{2}(\Omega_t)}^2$ as follows
\begin{align*}
&\|\nabla\matder^2 v\star\nabla p\|_{L^{2}(\Omega_t)}^2+\|\nabla\matder  v\star\nabla\matder p\|_{L^{2}(\Omega_t)}^2+\|\nabla v\star\nabla\matder^2 p\|_{L^{2}(\Omega_t)}^2\\
\le{}& C(\|\matder^2 v\|_{H^{2}(\Omega_t)}^2\|p\|_{H^{3/2}(\Omega_t)}^2+\|\nabla(H\cdot\nabla H)\star\nabla\matder p\|_{L^{2}(\Omega_t)}^2\\
&\quad +\|\nabla^2p\star\nabla\matder p\|_{L^{2}(\Omega_t)}^2+\|\nabla\matder^2 p\|_{L^{2}(\Omega_t)}^2)\\ 
\le{}& C\|\matder^2 v\|_{H^{2}(\Omega_t)}^2+C\|\nabla\matder p\|_{L^3(\Omega_t)}^2+C\|\nabla\matder^2 p\|_{L^{2}(\Omega_t)}^2.
\end{align*}
Note that   $\|\nabla\matder^2 p\|_{L^{2}(\Omega_t)}^2$ and $\|\nabla\matder p\|_{L^3(\Omega_t)}^2$ have fewer material derivatives than $\|\nabla\matder^3 p\|_{L^{2}(\Omega_t)}^2$. Therefore, it can be estimated as $I_8(t)$ in the same fashion, and we can obtain $\|\matder^3 p\|_{H^{1/2}(\parOmega_t)}^2\le C\tilde{\energy}(t)+\vare \barE(t)$.
Similarly, it holds $\|\matder^2 p\|_{H^{2}(\parOmega_t)}^2\le C\tilde{\energy}(t)+\vare \barE(t)$.
Combining the above estimates, we conclude that
$\|\matder^{2}v\|_{H^{3}(\Omega_t)}^2+\|\matder v\|_{H^{9/2}(\Omega_t)}^2\le C\tilde{\energy}(t)+\vare \barE(t)$.  

\textbf{Step 3.} Finally, we bound $\|  v\|_{H^6(\Omega_t)}^2$ and $\| H\|_{H^6(\Omega_t)}^2$. From  \eqref{e:div-curl 3}, we see that $\|v\|_{H^6(\Omega_t)}^2\le C(\tilde{\energy}(t)+\|\surflaplace  v_n\|_{H^{7/2}(\parOmega_t)}^2+\|B\|_{H^{9/2}(\parOmega_t)}^2)$ and $\|H\|_{H^6(\Omega_t)}^2 \le C(\tilde{\energy}(t)+\|B\|_{H^{9/2}(\parOmega_t)}^2)$. 
Recalling Lemma \ref{l:B by p} and by the trace theorem, it follows that $\|B\|_{H^{9/2}(\parOmega_t)}^2\le  C(1+\|p\|_{H^{9/2}(\parOmega_t)}^2)\le \vare\barE(t)+\|H\|_{H^{5}(\Omega_t)}^2+C_\vare$. 

Again by \eqref{e:div-curl 3}, we can estimate in $H^5(\Omega_t)$ and deduce $\|H\|_{H^5(\Omega_t)}^2\le C\tilde{\energy}(t)+\|B\|_{H^{7/2}(\parOmega_t)}^2$. 
Similarly, it holds $\|B\|_{H^{7/2}(\parOmega_t)}^2 \le \vare \barE(t)+\|H\|_{H^4(\Omega_t)}^2+C_\vare\le \vare \barE(t)+C_\vare$.
Thus, $\|B\|_{H^{9/2}(\parOmega_t)}^2\le \vare \barE(t)+C_\vare,\|p\|_{H^{9/2}(\parOmega_t)}^2\le \vare \barE(t)+C_\vare$,
and $\|H\|_{H^6(\Omega_t)}^2\le \vare \barE(t)+C_\vare$. 

We are left with the term $\|\surflaplace  v_n\|_{H^{7/2}(\parOmega_t)}^2$. From \eqref{e:Dtp} and by the above calculations, it follows that
\begin{align*}
\|\surflaplace  v_n\|_{H^{7/2}(\parOmega_t)}^2&\le C\|\matder p\|_{H^{7/2}(\parOmega_t)}^2+C\||B|^2v_n\|_{H^{7/2}(\parOmega_t)}^2+C\|\tangrad p\cdot v\|_{H^{7/2}(\parOmega_t)}^2\\ 
&\le C\|v\|_{H^{4}(\Omega_t)}^2\|B\|_{L^{\infty}(\parOmega_t)}^2\|B\|_{H^{7/2}(\parOmega_t)}^2+\frac{\vare}{2} \barE(t)+C_\vare\tilde{\energy}(t)\le C_\vare\tilde{\energy}(t)+\vare \barE(t),
\end{align*}
where we have used the fact that
\begin{align*}
\|\matder p\|_{H^{7/2}(\parOmega_t)}^2 
\le{}& C(\|\matder p\|_{L^{2}(\parOmega_t)}^2+\|\nabla\matder p\|_{H^{3}(\Omega_t)}^2)\\
\le{}& C(1+\|\matder^2v\|_{H^{3}(\Omega_t)}^2+\|\matder(H\cdot\nabla H)\|_{H^{3}(\Omega_t)}^2 +\|\nabla v\star (H\cdot\nabla H-\matder v)\|_{H^{3}(\Omega_t)}^2)\\ 
\le{}&  C_\vare\tilde{\energy}(t)+\frac{\vare}{2} \barE(t),
\end{align*}
and interpolation arguments since $\|\matder H\|_{H^{4}(\Omega_t)}^2$ and  $\|\matder^2v\|_{H^{3}(\Omega_t)}^2$ have already been controlled. This completes the proof.
\end{proof}
 
\begin{proposition}\label{p:ene est 2}
Let $l\ge 4$. Assume that \eqref{e:a priori} holds for some $T>0$ and $\sup_{0\le t< T}E_{l-1}(t)\le C$. Then, we have $E_l(t)\le C(1+\energy_l(t))$,
where the constant $C$ depends on $l,T,\mathcal{N}_T,\mathcal{M}_T $ and $\sup_{0\le t< T}E_{l-1}(t)$.
\end{proposition} 
\begin{proof}
We will show that $E_l(t)\le C \tilde{\energy}_l(t)$ and we divide the proof into three steps.

\textbf{Step 1.} We claim that it is sufficient to bound $\|\matder^{l+1-k}v\|_{H^{3k/2}(\Omega_t)}^2,k\in\{1,2,\cdots,l\},\| v\|_{H^{\lfloor3(l+1)/2\rfloor}(\Omega_t)}^2$ and $\| H\|_{H^{\lfloor3(l+1)/2\rfloor}(\Omega_t)}^2$.  Indeed,  
$\|\matder^{l+1-k}H\|_{H^{3k/2}(\Omega_t)}^2$ can be controlled by these quantities. Starting with the case of $2\le k\le l-1$, from the hypothesis, \eqref{e:Dt^k H} and \eqref{e:nabDt^k H}, we have  
\begin{equation*}
\|\matder^{l+1-k}H\|_{H^{3k/2}(\Omega_t)}^2
\le C\sum_{\substack{1\le m\le l+1-k\\|\beta|\le l+1-k-m}}\|\nabla\matder^{\beta_1}v\|_{H^{3k/2}(\Omega_t)}^2\cdots\| \nabla\matder^{\beta_{m}}v\|_{H^{3k/2}(\Omega_t)}^2\|H\|_{H^{3k/2}(\Omega_t)}^2.
\end{equation*} 
If $m=1$, we see that $\|\matder^{l+1-k}H\|_{H^{3k/2}(\Omega_t)}^2 
\le C(\|\matder^{l+1-(k+1)}v\|_{H^{3 (k+1)/2}(\Omega_t)}^2+1)$,
since $\|H\|_{H^{3k/2}(\Omega_t)}^2\le CE_{l-1}(t)\le C$. For $m\ge 2$, it holds $\|\matder^{l+1-k}H\|_{H^{3k/2}(\Omega_t)}^2\le CE_{l-1}(t)\cdots E_{l-1}(t)\le C$.

Next, we deal with the case of $k=1$, and it follows that
\begin{align*} 
\|\matder^l H\|_{H^{3/2}(\Omega_t)}^2  
\le{}& C\sum_{\beta_1\le l-1}\|\nabla\matder^{\beta_1}v\|_{H^{3/2 }(\Omega_t)}^2\| H\|_{H^{2 }(\Omega_t)}^2\\
&+
C\sum_{\substack{2\le m\le l,|\beta|\le l-m}}\|\nabla\matder^{\beta_1}v\|_{H^{2 }(\Omega_t)}^2\cdots\| \nabla\matder^{\beta_{m}}v\|_{H^{2 }(\Omega_t)}^2\| H\|_{H^{2 }(\Omega_t)}^2\\
\le{}& C(\|\nabla\matder^{l-1}v\|_{H^{3/2 }(\Omega_t)}^2+1)\le C(\|\matder^{l+1-2}v\|_{H^3(\Omega_t)}^2+1).
\end{align*}

Finally, for an even integer $k=l$, one has $\|\matder H\|_{H^{3l/2}(\Omega_t)}^2\le C\|H\|_{H^{\lfloor3l/2\rfloor}(\Omega_t)}^2\|v\|_{H^{\lfloor3l/2+1\rfloor}(\Omega_t)}^2\le C\|v\|_{H^{\lfloor3l/2+1\rfloor}(\Omega_t)}^2$ from $E_{l-1}(t)\le C$, 
and  if $k=l$ is odd, we have by Lemma \ref{l:2.10 jul} that
\begin{align*}
\|\matder H\|_{H^{3l/2}(\Omega_t)}^2&\le C(\|H\|_{L^{\infty}(\Omega_t)}^2\|v\|_{H^{3l/2+1}(\Omega_t)}^2+\|H\|_{H^{3l/2}(\Omega_t)}^2\|v\|_{L^{\infty}(\Omega_t)}^2)\\
&\le C\|v\|_{H^{\lfloor3(l+1)/2\rfloor}(\Omega_t)}^2+C\|H\|_{H^{\lfloor3(l+1)/2\rfloor}(\Omega_t)}^2.
\end{align*}  

\textbf{Step 2.} We claim that $\|\matder^{l}v\|_{H^{3/2}(\Omega_t)}^2\le \vare E_l(t)+C\tilde{\energy}_l(t)$.  Due to the fact that $\|\nu\|_{H^{5/2+\delta}(\Omega_t)}\le C$ and the assumption $E_{l-1}(t)\le C$, we have
\begin{align*}
\|\matder^{l}v\cdot \nu\|_{L^2(\parOmega_t)}^2 
\le{}&|\int_{\Omega_t} (\matder^{l}v\cdot \nu)\dive\matder^{l}vdx|+|\int_{\Omega_t} \nabla\matder^{l}v\star\matder^{l}vdx|+|\int_{\Omega_t} \matder^{l}v\star\nabla\nu\star\matder^{l}vdx|\\
\le{}& C(\|\matder^{l}v\|_{L^{2}(\Omega_t)}^2+\|\dive\matder^{l}v\|_{L^{2}(\Omega_t)}^2+\|\nabla\matder^{l}v\|_{L^{2}(\Omega_t)}\|\matder^{l}v\|_{L^{2}(\Omega_t)})\\ 
\le{}& \vare\|\nabla\matder^{l}v\|_{L^{2}(\Omega_t)}^2+ C(1+\|\dive\matder^{l}v\|_{L^{2}(\Omega_t)}^2).
\end{align*}
This, combined with  \eqref{e:div-curl 1}, we see that  
\begin{align*}
\|\matder^{l}v\|_{H^{3/2}(\Omega_t)}^2\le{}& C(\|\matder^{l}v\cdot\nu\|_{H^{1}(\parOmega_t)}^2+\|\matder^{l}v\|_{L^{2}(\Omega_t)}^2+\|\dive\matder^{l}v\|_{H^{1/2}(\Omega_t)}^2+\|\curl\matder^{l}v\|_{H^{1/2}(\Omega_t)}^2)\\ \le{}&C(\vare \|\matder^{l}v\|_{H^{1}(\Omega_t)}^2+1+  E_{l-1}(t)+\|\tangrad(\matder^{l}v\cdot\nu)\|_{L^{2}(\parOmega_t)}^2\\
&\quad +\|\dive\matder^{l}v\|_{H^{1/2}(\Omega_t)}^2+\|\curl\matder^{l}v\|_{H^{1/2}(\Omega_t)}^2).
\end{align*} 
Then, it follows that $\|\matder^{l}v\|_{H^{3/2}(\Omega_t)}^2\le C(\tilde{e}_l(t)+\|\dive\matder^{l}v\|_{H^{1/2}(\Omega_t)}^2+\|\curl\matder^{l}v\|_{H^{1/2}(\Omega_t)}^2)$.
Applying Lemmas \ref{l:formu 5}, \ref{l:err est 1} and \ref{l:err est 2}, we arrive at
\begin{align*}
&\|\dive\matder^{l}v\|_{H^{1/2}(\Omega_t)}^2+\|\curl\matder^{l}v\|_{H^{1/2}(\Omega_t)}^2\\\le{}& C(\|R^{l-1}_{\Rdiv}\|_{H^{1/2}(\Omega_t)}^2+\|R^{l-1}_{\nabla H,\nabla H}\|_{H^{1/2}(\Omega_t)}^2+\|R^{l-1}_{\nabla^2 H, H}\|_{H^{1/2}(\Omega_t)}^2)\le \vare E_l(t)+C_\vare,
\end{align*}
where $\vare>0$ is sufficiently small. This concludes the claim.

\textbf{Step 3.} We claim that for $2 \le k \leq l$, it holds
\begin{equation}\label{e:ene est 2}
\|\matder^{l+1-k} v\|_{H^{3k/2}(\Omega_t)}^2 \leq C  \| \matder^{l+3 -k} v \|_{H^{3k/2-3}(\Omega_t)}^2+ \vare E_l(t) + C_\vare \tilde{\energy}_l(t).
\end{equation}
Once we have these estimates, it follows that $\|\matder^{l-1} v\|_{H^{3}(\Omega_t)}^2\le  \vare E_l(t)+C_\vare \tilde{\energy}_l$. This, combined with Step 2, will control $  \|\matder^{l+1-k}v\|_{H^{3k/2}(\Omega_t)}^2$ for any $3\le k\le l$. To prove \eqref{e:ene est 2}, from Lemmas  \ref{l:formu 5}, \ref{l:err est 1} and \ref{l:err est 2}, and \eqref{e:div-curl 2},  it holds
\begin{align*}
\|\matder^{l+1-k}v\|_{H^{3k/2}(\Omega_t)}^2 \le{}& C(\|\surflaplace(\matder^{l+1-k}v\cdot \nu)\|_{H^{(3k-5)/2}(\parOmega_t)}^2+\|R^{l-k}_{\Rdiv}\|_{H^{(3k-2)/2}(\Omega_t)}^2\\
&\quad +\|R^{l-k}_{\nabla H,\nabla H}\|_{H^{(3k-2)/2}(\Omega_t)}^2+\|R^{l-k}_{\nabla^2 H,H}\|_{H^{(3k-2)/2}(\Omega_t)}^2+E_{l-1}(t))\\
\le{}& C \|\surflaplace(\matder^{l+1-k}v\cdot \nu)\|_{H^{(3k-2)/2}(\parOmega_t)}^2+\vare E_l(t)+C_\vare.
\end{align*}
Lemmas \ref{l:Dt^lp} and \ref{lem-err-est-3} give $\matder^{l+2-k} p= -\surflaplace (\matder^{l+1-k} v\cdot\nu)+R^{l+1-k}_p$, and  $\|R^{l+1-k}_p\|_{H^{(3k-5)/2}(\parOmega_t)}^2\le\vare E_l(t)+C_\vare$. Then, we obtain $\|\matder^{l+1-k}v\|_{H^{3k/2}(\Omega_t)}^2\le C\|\matder^{l+2-k} p\|_{H^{(3k-5)/2}(\parOmega_t)}^2+\vare E_l(t)+C_\vare$. 
By \eqref{e:har ext 2}, we see that $\|\matder^{l+2-k} p\|_{H^{(3k-5)/2}(\parOmega_t)}^2 \le\|\matder^{l+2-k} p\|_{H^{(3k-6)/2}(\parOmega_t)}^2+\|\nabla\matder^{l+2-k} p\|_{H^{(3k-6)/2}(\Omega_t)}^2$.

The first term can be controlled by Lemma \ref{l:formu 4} as in Proposition \ref{p:ene est 1}, i.e., $\|\matder^{l+2-k} p\|_{H^{(3k-6)/2}(\parOmega_t)}^2\le \vare E_l(t)+C_\vare$. For the second term, by \eqref{e:mhd}, Lemmas \ref{l:err est 2} and \ref{l:err est 3}, it holds
\begin{align*}
\|\nabla\matder^{l+2-k} p\|_{H^{(3k-6)/2}(\Omega_t)}^2 
\le{}&\|\matder^{l+3-k} v\|_{H^{(3k-6)/2}(\Omega_t)}^2+ \|\sum_{\beta\le l+1-k}\nabla\matder^{\beta}v\star\nabla H\star H\|_{H^{(3k-6)/2}(\Omega_t)}^2\\
&+\|R^{l+1-k}_{\Rbulk}
\|_{H^{(3k-6)/2}(\Omega_t)}^2+\|R^{l+2-k}_{\nabla H,H}\|_{H^{(3k-6)/2}(\Omega_t)}^2\\
\le{}&\|\matder^{l+3-k} v\|_{H^{(3k-6)/2}(\Omega_t)}^2+\vare E_l(t)+C_\vare.
\end{align*}
Combining the above estimates,  \eqref{e:ene est 2} follows.

It remains to verify that $\| v\|_{H^{\lfloor3(l+1)/2\rfloor}(\Omega_t)}^2+\| H\|_{H^{\lfloor3(l+1)/2\rfloor}(\Omega_t)}^2\le \vare E_l(t)+C_\vare \tilde{\energy}_l(t)$. Note that from Lemma \ref{l:B by p} with $l\ge 4$, one has $\|B\|_{H^{3l/2-1}(\parOmega_t)}\le C$ and  $\|B\|_{H^k(\parOmega_t)}\le  C (1+\|p\|_{H^k(\parOmega_t)})$ for $k\in\N/2,k \leq 3l/2$. Then, we can apply the argument as in Proposition \ref{p:ene est 1}. This completes the proof.
\end{proof}

We are ready to prove the main results.
\begin{proof}[Proof of Theorem \ref{t:main t}]
We divide the proof into three parts.

\textbf{Step 1.} We prove the first two statements in Theorem \ref{t:main t}. Assume that the a priori assumptions \eqref{e:a priori} hold for some $T>0$.  

Recalling the estimates in Section \ref{s:p est} that $\barE(0)+\sup_{0\le t< T}\|p\|_{H^3(\Omega_t)}^2\le C$, 
where $C$ depends on $T,\mathcal{N}_T,\mathcal{M}_T, \|v_0\|_{H^6(\Omega_0)},\|H_0\|_{H^6(\Omega_0)}$, and $\|\meancurv_{\parOmega_0}\|_{H^5(\parOmega_0)}$. Then, the assumptions of Proposition \ref{p:ene est 1} hold for any $0\le t<T$, and   Propositions \ref{p:d/dt 2} and \ref{p:ene est 1} allow us to obtain
\begin{equation}\label{e:main th 2}
\frac{d}{dt}\barenergy(t)\le C\barE(t)\le C(1+\barenergy(t)),\quad 0\le t<T. 
\end{equation} 
Integrating over $(0,t)$, we have $\sup_{0\le t< T}\barenergy(t)\le C(1+\barenergy(0))e^{CT}$.
Again by Proposition \ref{p:ene est 1}, we see that 
\begin{equation}\label{e:main th 3}
\sup_{0\le t< T}\barE(t)\le C+C(1+\barenergy(0))e^{CT}\le C+ C\barE(0)e^{CT}\le\bar{C}_0,
\end{equation}
where $\bar{C}_0=\bar{C}_0\left( T, \mathcal{N}_T,\mathcal{M}_T,\|v_0\|_{H^6(\Omega_0)},\|H_0\|_{H^6(\Omega_0)},\|\meancurv_{\parOmega_0}\|_{H^5(\parOmega_0)}\right) $. 

With $\sup_{0\le t< T}(\barE(t)+\|p\|_{H^3(\Omega_t)}^2)\le \bar{C}_0$, applying Lemma
\ref{l:B by p} and  the trace theorem, it follows that $\|B\|_{H^{9/2}(\parOmega_t)}^2 
\le C(1+\|H\cdot\nabla H-\matder v\|_{H^{4}(\Omega_t)}^2)
\le C(\bar{C}_0)$,
giving $\|B\|_{H^{9/2}(\parOmega_t)}^2+\|p\|_{H^{5}(\Omega_t)}^2\le C(\bar{C}_0)$. We proceed to find that $$\|p\|_{H^{11/2}(\Omega_t)}^2\le C(1+\|\nabla p\|_{H^{9/2}(\Omega_t)}^2) \le C(1+\|H\cdot\nabla H-\matder v\|_{H^{9/2}(\Omega_t)}^2)\le C(\bar{C}_0),$$
and we utilize Lemma
\ref{l:2.12 jul} to obtain $\|B\|_{H^{5}(\parOmega_t)}^2\le C(1+\|p\|_{H^{5}(\parOmega_t)}^2) 
\le C(\bar{C}_0)$.
In particular, it follows that $\|\meancurv\|_{H^{5}(\parOmega_t)}^2\le C(\bar{C}_0)$, and Proposition \ref{p:initial cond} yields $\sum_{k=0}^{3}\|\matder^{3-k}p\|_{H^{3k/2+1}(\Omega_t)}^2\le C$,
where $C=C\left( \radi-\|h(\cdot,t)\|_{L^\infty(\Gamma)},\|v\|_{H^6(\Omega_t)},\|H\|_{H^6(\Omega_t)},\|\meancurv\|_{H^5(\parOmega_t)}\right)$. Combining the above estimates,  \eqref{e_t1_1} follows. Then, from the definitions of the material derivative and $\barE(t)$, we can also verify \eqref{e_t1_2}. 

To prove the second result, for $l\ge 4$, we apply Propositions \ref{p:d/dt 2} and \ref{p:ene est 2} by induction to find that: if $\sup_{0\le t< T}E_{l-1}(t)\le C$, then it follows that $\frac{d}{dt}\energy_l(t)\le CE_l(t)\le C(1+\energy_l(t))$.
Similarly, we integrate over $(0,t)$ and use Proposition \ref{p:ene est 2} again to obtain $\sup_{0\le t< T}\energy_l(t)\le C(1+\energy_l(0))e^{CT}$,	and
\begin{equation}\label{e:main th 4}
\sup_{0\le t< T}E_l(t)\le C+C(1+\energy_l(0))e^{CT}\le C_l\left(l, T,  \mathcal{N}_T,\mathcal{M}_T,\sup_{0\le t< T}E_{l-1}(t),e_l(0)\right).
\end{equation}
However, the induction argument implies that \eqref{e:main th 4} holds for all $l$ and the constant $C_l$ which depends on $l,T, \mathcal{N}_T, \mathcal{M}_T,e_l(0)$ and $\barenergy(0)$ from \eqref{e:main th 3}. Note that $\barenergy(0)+\energy_l(0)\le CE_l(0)$, and the constant $C_l$ in fact depends on $l,T, \mathcal{N}_T, \mathcal{M}_T$, and $E_l(0)$. This completes the proof of our claim. Again by the definition of the material derivative, \eqref{e_t1_4} follows. 

\textbf{Step 2.} We prove the last statement in Theorem \ref{t:main t}, i.e., the a priori assumptions \eqref{e:a priori} hold for some time $T_0\ge c_0>0$, where $c_0$ depends on $\mathcal{M}_0,\|v_0\|_{H^6(\Omega_0)},\|H_0\|_{H^6(\Omega_0)}$ and $\|\meancurv_{\parOmega_0}\|_{H^5(\parOmega_0)}$.  	To this aim, we define 
$$I(t) \coloneqq \|B\|_{H^3(\parOmega_t)}^2+\|p\|_{H^3(\Omega_t)}^2+\|v\|_{H^4(\Omega_t)}^2+\|H\|_{H^4(\Omega_t)}^2+1,\quad t\ge 0.$$
Suppose that it holds $I(t)\leq 2I(0)$ and $\mathcal{M}_t \geq \mathcal{M}_0/2$ for some $t>0$, where $\mathcal{M}_0=\radi-\|h_0\|_{L^\infty(\Gamma)}$. Then we have  $\|\meancurv_{\parOmega_t}\|_{H^3(\parOmega_t)}^2\le  C(I(0))$. Thus, applying Lemma \ref{l:bou reg}, one has $\|h(\cdot,t)\|_{H^{3+\delta}(\Gamma)}\le C$, for $\delta>0$ small enough, where the constant $C$ depends on $\|\meancurv_{\parOmega_t}\|_{H^{1+\delta}(\parOmega_t)}$, and hence on $I(0)$. An application of Proposition \ref{p:ene est 1} allows us to obtain that there exists a constant $C$, depending on $I(0)$ and $\mathcal{M}_0$ such that 
\begin{equation}\label{e:main th 5}
\barE(t)\le C(1+\barenergy(t)).
\end{equation}
From the above argument, we define $T_0 \in (0, 1]$ to be the largest number such that
\begin{equation}\label{e:time T0}
[0,T_0]\subset\left\lbrace t\in[0,1]: I(t)\ge I(0)/2,\mathcal{M}_t \ge \mathcal{M}_0/2 ,\text{ and } \barenergy(t)\le 1+\barenergy(0)\right\rbrace .
\end{equation}
Here, we assume that $T_0<1$, since the claim would be trivial otherwise. We note that the last condition together with  \eqref{e:main th 5} implies that 
\begin{equation}\label{e:main th 6}
\sup_{0\le t\le T_0}\barE(t)\le C(1+\barenergy(t))\le C(2+\barenergy(0))\le C\barE(0).
\end{equation}
Also, we observe that 
satisfies $\mathcal{N}_{T_0}^2 \leq C \sup_{0\le t <T_0} \barE(t)$,
thanks to the curvature bound $\|B\|_{H^3(\parOmega_t)}\le 2I(0)$. Indeed, from $\tangrad v_n=\tangrad v\cdot \nu-v\star B$, we can bound $\|v_n\|_{H^4(\parOmega_t)}$ by using $\|v\|_{H^4(\parOmega_t)}$ and $\|B\|_{H^3(\parOmega_t)}$. 

The estimate \eqref{e:main th 6} ensures that the a priori assumptions \eqref{e:a priori} hold for time $T=T_0$, and the claim follows once we show that $T_0$ specified in \eqref{e:time T0} has a lower bound $c_0>0$.   
From the definition of $T_0$, at least one of the three conditions has equality. Assume that $I(T_0) = 2I(0)$. Then, it holds $\barE(t) \leq  C \barE(0)$, for all $t \leq T_0$ by \eqref{e:main th 6}. We will show that
\begin{equation}\label{e:main th 7}
\frac{d}{dt}I(t)\le C\barE(t) I(t) \le C\barE(0) I(t).
\end{equation}  
We focus on the computation of the highest-order terms. In fact, Lemma \ref{l:formu 3} yields
\begin{align*}
\frac{d}{dt}\left(  \|\nabla^4v\|_{L^2(\Omega_t)}^2+\|\nabla^4H\|_{L^2(\Omega_t)}^2\right)  
\le{}& \int_{\Omega_t}\nabla^4\matder v\star\nabla^4v+\sum_{|\alpha|\le 3}\nabla^{1+\alpha_1}v\star\nabla^{1+\alpha_2}v \star\nabla^4vdx\\
&+\int_{\Omega_t}\nabla^4\matder H\star\nabla^4H+\sum_{|\alpha|\le 3}\nabla^{1+\alpha_1}v\star\nabla^{1+\alpha_2}H \star\nabla^4Hdx\\ 
\le{}& C\barE(t) I(t).
\end{align*}
Applying Lemmas \ref{l:formu 3} and \ref{l:nab H, nab H},  it is easy to deduce $\frac{d}{dt}\|\nabla^3p\|_{L^2(\Omega_t)}^2\le C\barE(t) I(t)$.
Similarly, we can obtain by Lemma \ref{l:formu 1} that $\frac{d}{dt} \|\tangrad^3B\|_{L^2(\parOmega_t)}^2\le C\barE(t) I(t)$. 
By integrating \eqref{e:main th 7} over $(0,T_0)$ and using $I(T_0) = 2 I(0)$, we obtain $\ln 2=\ln I(T_0)-\ln I(0)\le CT_0\barE(0)$.
Then we have $T_0\ge C/\barE(0)=c_0,$
where the constant $c_0$  depends on $I(0),\mathcal{M}_0$, and $\barE(0)$. Moreover, by Lemma \ref{l:B by p} and Proposition \ref{p:initial cond}, the constant $c_0$ depends only on $\mathcal{M}_0,\|v_0\|_{H^6(\Omega_0)},  \|H_0\|_{H^6(\Omega_0)}$ and $\|\meancurv_{\parOmega_0}\|_{H^5(\parOmega_0)}$. 

A similar argument applies if we have an equality in the third condition, i.e., $\barenergy(T_0)= 1+\barenergy(0)$.  In fact, it follows that $\frac{d}{dt}\barenergy(t)\le C\barE(t)\le C\barE(0)$
by \eqref{e:main th 2} and \eqref{e:main th 6}, and we integrate over $(0,T_0)$ to obtain $1=\barenergy(T_0)-\barenergy(0)\le C\barE(0)T_0$.
This results in $T_0 \ge c_0>0$ again. 

Finally, we assume that  $\mathcal{M}_{T_0} = \mathcal{M}_0/2$. Recalling that $\mathcal{M}_T = \radi - \sup_{0\le t<T}\|h(\cdot, t)\|_{L^\infty(\Gamma)}$,
and $\mathcal{M}_0>0$, we define $0<T_1 \le T_0$ by $\mathcal{M}_{T_0} =  \radi- \|h(\cdot,T_1)\|_{L^\infty(\Gamma)}$.
It is clear that $\|v_n\|_{L^\infty(\Omega_t)}^2\le C\barE(t) \le C \barE(0)$ by using \eqref{e:main th 6}.  Recalling the fact that $\partial_t h=v_n$,  we have by the fundamental Theorem of calculus that 
 \begin{equation*}
 \mathcal{M}_{T_0}=\radi-\|h(\cdot,T_1)\|_{L^\infty(\parOmega)}\ge  \radi-\|h_0\|_{L^\infty(\parOmega)}-\int_0^{T_1} \|v_n\|_{L^\infty(\Omega_t)} dt\ge \mathcal{M}_{0}- C\barE(0)^{\frac 12} T_1,
 \end{equation*} 
which means $T_0\ge T_1\ge C\mathcal{M}_{0}/\barE(0)^{1/2}>0$. This concludes the claim.  

\textbf{Step 3.} Finally, we prove that the smooth solution does not develop singularities at time $T$. According to the a priori assumptions, the estimates \eqref{e_t1_1} and \eqref{e_t1_3} hold. In particular, we conclude by Lemmas \ref{l:B by p} and \ref{l:bou reg} that the regularity of the curvature implies the regularity of the free boundary, i.e.,   $\parOmega_T\in C^\infty$. Additionally, the quantitative regularity estimates show that the time derivatives of arbitrary order of the velocity and magnetic field are smooth, i.e., belong to $\in C^\infty(\Omega_T)$.  
This completes the proof of the theorem. 
\end{proof}
Finally, we prove the blow-up classification in Theorem \ref{t:main t2}.
\begin{proof}[Proof of Theorem \ref{t:main t2}]
We prove this by contradiction. Assume that $T_*<\infty$, i.e., $v(\cdot,T_*),H(\cdot,T_*)\notin H^6(\Omega_{T_*})$ or $\parOmega_{T_*}\notin H^7$. Assume further that none of (1)–(4) hold. That is, $\inf_{0\le t<T_*}\radi(\Omega_t)>0,\parOmega_t\in H^{3+\delta},0\le t\le T_*$, and $\sup_{0\le t< T_*}
(\|\nabla v\|_{H^{3}(\Omega_t)}+\|\nabla H\|_{H^{3}(\Omega_t)}+\|v_n\|_{H^{4}(\parOmega_t)})<\infty$, where we have applied Lemma \ref{l:bou reg} and the fact that $v_n=V_{\parOmega_t}$. In particular, $\radi(\Omega_{T_*})>0$ and we choose $\parOmega_{T_*}=\partial\Omega_{T_*}$ as the reference surface to represent the free boundary over a short time interval before $T_*$. More precisely, the height function $h(\cdot,t)$ is well-defined on $\left[T_*-\vare,T_*\right)$ for sufficiently small $\vare>0$ and one has $\sup_{\left[T_*-\vare,T_*\right)}\|h\|_{H^{3+\delta}(\parOmega_{T_*})}<\infty$. Therefore, it holds that 
\begin{equation*}
\sup_{T_*-\vare\le t< T_*}(\|h\|_{H^{3+\delta}(\parOmega_{T_*})}+
\|\nabla v\|_{H^{3}(\Omega_t)}+\|\nabla H\|_{H^{3}(\Omega_t)}+\|v_n\|_{H^{4}(\parOmega_t)})<\infty.
\end{equation*}
Applying the low-order estimates in Theorem \ref{t:main t}, it follows that $v(\cdot,T_*),H(\cdot,T_*)\in H^6(\Omega_{T_*})$ and $ \parOmega_{T_*}\in H^7$ and the solution can be extended for some time. This leads to a contradiction and the proof is complete. 
\end{proof}
\section{Further discussions of Theorem \ref{t:main t2}}\label{s:discuss}
In the blow-up classification given in Theorem \ref{t:main t2}, the first two scenarios concern the geometric behavior of the free boundary. In the final section of this manuscript, we explore the connection between the self-intersection singularity in case (1) and the curvature blow-up in case (2). 

To quantitatively characterize how close the free boundary is to self-intersection, we adopt the concept of the injectivity radius $\iota_0$ of the normal exponential map, as introduced in \cite{Christodoulou2000}. Specifically, $\iota_0(t)$ is defined as the largest positive number such that the map
\begin{equation*}
\parOmega_t\times(-\iota_0(t),\iota_0(t))\rightarrow\{ y\in\mathbb{R}^3: \operatorname{dist}(y,\parOmega_t)<\iota_0(t) \}\ \text{given by}\ (x, \iota) \mapsto x + \iota \nu(x),
\end{equation*}
is an injection.

By combining a lower bound on the injectivity radius $\iota_0(t)$ with an upper bound on the second fundamental form $B_{\parOmega_t}$, which measures the curvature, one can derive a positive lower bound for the uniform interior and exterior ball radius via \cite[Lemma 1]{Ginsberg2021}. Specifically, if there exists a constant $K>0$ such that
\begin{equation}\label{inj-B bound}
\frac{1}{\iota_0(t)}+\|B_{\parOmega_t}\|_{L^\infty(\parOmega_t)}\le K, 
\end{equation}
then there exists $r=r(K)>0$ such that $\radi(\Omega_t)\ge r$. Consequently, if condition \eqref{inj-B bound} holds uniformly for all $t\in \left[0,T_*\right)$, i.e.,
\begin{equation*} 
\sup_{t\in \left[0,T_*\right)}\left(\frac{1}{\iota_0(t)}+\|B_{\parOmega_t}\|_{L^\infty(\parOmega_t)}\right)\le K,
\end{equation*}
then the self-intersection singularity will be excluded. 

However, a uniform upper bound on the second fundamental form alone does not, in general, guarantee a uniform positive lower bound for the injectivity radius $\inf_{t \in \left[0,T_*\right)} \iota_0(t)$ or for the uniform interior and exterior ball radius $\inf_{t \in \left[0,T_*\right)} \radi(\Omega_t)$.

In fact, there exist surfaces whose curvature remains uniformly bounded while their injectivity radius tends to zero. Such configurations were employed by Coutand and Shkoller \cite{Coutand2019} to construct initial domains for the viscous water wave equations that lie sufficiently close to self-intersection (see Fig.\;2 and Fig.\;3 in \cite{Coutand2019}), together with divergence-free initial velocity fields that drive the boundary toward self-intersection. 
Notably, the curvature either remains unchanged or undergoes only minimal variation during the deformation that leads to the self-intersection in finite time. Similar constructions were later developed by Hong, Luo, and Zhao \cite{Hong2025} in the context of the viscous and non-resistive incompressible MHD equations.

Moreover, there exist surfaces for which the curvature becomes unbounded while the injectivity radius simultaneously tends to zero. To illustrate this, consider a dumbbell-shaped surface whose connecting neck is gradually squeezed and thinned. As this constriction intensifies, the curvature tends to infinity, and the interior ball radius approaches zero. A natural and interesting question is whether one can construct a class of regular solutions to system \eqref{e:mhd} based on such special geometric configurations, where the curvature of the free boundary blows up and the boundary simultaneously approaches self-intersection within a finite time.
 
\appendix 

\section{Some estimates and formulas}\label{app:1}
\begin{lemma}\label{l:formu 1} 
For a smooth function $f$, it holds
\begin{enumerate}[label={\rm (\arabic*)}]
\item $[\matder,\partial_i]f=- \partial_i v^k \partial_k f,
[\matder,\tangrad]f=-(\tangrad v)^\top\tangrad f,[\matder,\tangrad^2]f=\tangrad^2 v\star \tangrad f+\tangrad v\star \tangrad^2 f,[\tangrad,\nabla]f=\nabla f\star \nabla \nu\star \nu,[\matder,\surflaplace]f=\tangrad^2f\star\nabla v-\tangrad f\cdot \surflaplace v+B\star \nabla v\star \tangrad f,[\partial_\nu,\partial_k] u=-\nabla u\cdot\partial_k \nu$.
\item $\matder \nu=-(\tangrad v)^\top \nu=-\tangrad v_n+B v_\sigma,\tangrad v_n=\tangrad v^\top \nu+B_\parOmega v_\sigma, \matder B=-\tangrad^2v\star\nu-\tangrad v\star B$. 
\end{enumerate} 
\end{lemma}	
\begin{proof}
Most formulas can be found in \cite[Section 3.1]{Shatah2008a} and the others follow from direct calculations.
\end{proof}
\begin{lemma}[{\cite[Proposition A.2]{Shatah2008a}}]\label{l:bou reg}
Let $\Omega\subset\mathbb{R}^3$ be a domain such that $\partial\Omega\in H^{s_0},s_0>2$. Suppose $\|\meancurv\|_{H^{s-2}(\parOmega_t)}\le C$ with $s\ge s_0$, then $\partial\Omega\in H^s$.
\end{lemma}
 
Let $u\in L^2(\parOmega)$. We define the space $H^{1/2}(\parOmega)$ via the harmonic extension: 
\begin{equation*}
\|u\|_{H^{1/2}(\parOmega)}\coloneqq\|u\|_{L^2(\parOmega)}+\inf \{ \|\nabla w\|_{L^2(\Omega)}: w\in H^1(\Omega) \text{ and } w=u \text{ on } \parOmega \}=\|u\|_{L^2(\parOmega)}+ \|\nabla v\|_{L^2(\Omega)},
\end{equation*}
where $v\in H^1(\Omega)$ such that $v|_{\parOmega}=u$ in the trace sense and $\laplace v=0$ in the weak sense. We note that for $u\in H^1(\Omega)$, it holds $\|u\|_{H^{1/2}(\parOmega)}\le \|u\|_{L^2(\parOmega)}+\|\nabla u\|_{L^2(\Omega)}$.
Moreover, for $u\in H^2(\Omega)$ and $v\in H^1(\Omega)$ such that $u|_{\parOmega}$ is the trace of $v$ on $\parOmega$, we have $\|\nabla u\|_{L^2(\Omega)}^2\le \|\nabla(u-v)\|_{L^2(\Omega)}^2+\|\nabla v\|_{L^2(\Omega)}^2
\le \|(u-v)\laplace u\|_{L^1(\Omega)}+\|\nabla v\|_{L^2(\Omega)}^2\le \vare\|u-v\|_{L^2(\Omega)}^2+C_\vare\|\laplace u\|_{L^2(\Omega)}^2+\|\nabla v\|_{L^2(\Omega)}^2$, and therefore
\begin{align*}
\|\nabla u\|_{L^2(\Omega)}^2 
&\le \vare\|\nabla(u-v)\|_{L^2(\Omega)}^2+C_\vare\|\laplace u\|_{L^2(\Omega)}^2+\|\nabla v\|_{L^2(\Omega)}^2\\
&\le \vare\|\nabla u\|_{L^2(\Omega)}^2+C( \|\laplace u\|_{L^2(\Omega)}^2+\| u\|_{H^{1/2}(\parOmega)}^2),
\end{align*}
where we have used the fact that $v-u\in H^1_0(\Omega) $ and  Poincar\'{e}'s inequality.
Therefore, we obtain
\begin{equation}\label{e:har ext 2}
\|\nabla u\|_{L^2(\Omega)}\le C( \|\laplace u\|_{L^2(\Omega)}+\|  u\|_{H^{1/2}(\parOmega)}). 
\end{equation}
Moreover, if $v$ is the harmonic extension of $u|_{\parOmega}$, it holds $\|v\|_{H^1(\Omega)}\le C\|u\|_{H^{1/2}(\parOmega)}$, and we have
\begin{equation}\label{e:har ext 3}
\|u\|_{H^1(\Omega)}\le C( \|\laplace u\|_{L^2(\Omega)}+\|  u\|_{H^{1/2}(\parOmega)}). 
\end{equation}
\begin{lemma}[{\cite[Corollary 2.9]{Julin2024}}]\label{l:2.9 jul}
Let $m \in \Nz $ and $\parOmega \subset \mathbb{R}^3$ be a compact $2$-dimensional hypersurface which is $C^{1, \alpha}$-regular such that $\parOmega=\partial \Omega$ and satisfies the condition $(\operatorname{H}_m)$, i.e.,
\begin{equation}\label{e:2.9 H_m}
\|B\|_{L^4(\parOmega)}\le C\text{, if } m=2,\quad \|B\|_{L^\infty(\parOmega)}+\|B\|_{H^{m-2}(\parOmega)}\le C\text{, if } m>2.
\end{equation} 
Then for all  $k, l\in \N/2$ with $k<l \leq m$ and for $q \in[1, \infty]$, it holds $\|u\|_{H^k(\parOmega)} \leq C\|u\|_{H^l(\parOmega)}^\theta\|u\|_{L^q(\parOmega)}^{1-\theta}$,
where $\theta \in[0,1]$ is given by $1=k-\theta(l-1)+(2-2\theta)/q$, and $\|u\|_{H^k(\Omega)} \leq C\|u\|_{H^l(\Omega)}^\theta\|u\|_{L^q(\Omega)}^{1-\theta}$,
where $\theta \in[0,1]$ is given by $1/2=k/3+\theta(1/2-l/3)+(1-\theta)/q$. Moreover,  for $k, l\in\Nz $  with $k<l \leq m,p\in\left[1,\infty\right),q \in[1, \infty]$, it holds $\|\nabla^ku\|_{L^p(\Omega)}\leq C\|u\|^\theta_{H^l(\Omega)}\|u\|^{1-\theta}_{L^q(\Omega)}$,
where $\theta \in[0,1]$ is given by $1/p=k/3+\theta(1/2-l/3)+(1-\theta)/q$.
\end{lemma}
\begin{lemma}[\cite{Cheng_2016, Kato1988}]\label{l:Kato Ponce} For $f,g\in C_0^\infty(\mathbb{R}^n)$ and $2\le p_1,q_2<\infty,2\le p_2,q_1\le \infty$ with $1/p_1+1/q_1=1/p_2+1/q_2=1/2$, we have for all $k\in \N/2$, $\|fg\|_{H^k}\le C(\|f\|_{W^{k,p_1}}\|g\|_{L^{q_1}}+	\|g\|_{W^{k,q_2}}\|f\|_{L^{p_2}})$.
\end{lemma} 
\begin{lemma}[{\cite[Proposition 2.10]{Julin2024}}]\label{l:2.10 jul}
Assume $\partial\Omega$ is $C^{1, \alpha}$-regular and satisfies $(\operatorname{H}_m)$ defined in \eqref{e:2.9 H_m}. Then for all $k\in \N/2, k \leq m$, it holds $\|f g\|_{H^k(\partial\Omega)} \leq C(\|f\|_{H^k(\partial\Omega)}\|g\|_{L^{\infty}(\partial\Omega)}+\|f\|_{L^{\infty}(\partial\Omega)}\|g\|_{H^k(\partial\Omega)})$,
and $\|f g\|_{H^k(\Omega)} \leq C(\|f\|_{H^k(\Omega)}\|g\|_{L^{\infty}(\Omega)}+\|f\|_{L^{\infty}(\Omega)}\|g\|_{H^k(\Omega)})$. Moreover, assume that $\|B\|_{L^4} \leq C$ and $k \in \Nz $. Then for $p_1, p_2, q_1, q_2 \in[2, \infty]$ with $p_1, q_2<\infty,1/p_1+1/q_1=1/p_2+1/q_2=1/2$,
we have $\|f g\|_{H^k(\parOmega)} \leq C(\|f\|_{W^{k, p_1}(\parOmega)}\|g\|_{L^{q_1}(\parOmega)}+\|f\|_{L^{p_2}(\parOmega)}\|g\|_{W^{k, q_2}(\parOmega)})$.
\end{lemma}  
\begin{lemma}[{\cite[Proposition 2.12]{Julin2024}}]\label{l:2.12 jul}
Assume that $\parOmega$ is $C^{1, \alpha}$-regular. For every $p \in (1,\infty)$, it holds $\|B_\parOmega\|_{L^p(\parOmega)} \le C(1+ \|\meancurv_\parOmega\|_{L^p(\parOmega)})$.
If in addition $\|B_\parOmega\|_{L^4(\parOmega)} \le C$, then for $k =1/2, 1, 2$, it holds $\|B_\parOmega\|_{H^{k}(\parOmega)} \leq C( 1+ \|\meancurv_\parOmega\|_{H^{k}(\parOmega)})$.
Finally, let $m\in\N/2, m\ge 3$, and assume additionally $\|B\|_{L^\infty(\parOmega)}+\|B\|_{H^{m-2}(\parOmega)}\le C$. Then the above estimate holds for all $k\in\N /2$ with $k\le m$.
\end{lemma} 
\begin{lemma}[{\cite[Lemma 3.5]{Julin2024}}]\label{l:3.5 jul}
Let $\Omega$ be a bounded domain with $\partial\Omega\in C^1$ and $\|B\|_{L^4} \leq C$. Then $\|u\|_{H^2(\Omega)} \leq C( \|\partial_\nu u \|_{H^{1/2}(\partial\Omega)}+\|\nabla u\|_{L^2
(\Omega)}+\|\laplace u\|_{L^2(\Omega)})$. Moreover, $\|\nabla u\|_{L^2
(\Omega)}$ can be replaced by $\|u\|_{L^2(\Omega)}$. 
\end{lemma}  

\section*{Acknowledgments}
Both authors were partially supported by the National Natural Science Foundation of China (NSFC) under Grant No. 12171460. Hao was also partially supported by the CAS Project for Young Scientists in Basic Research under Grant No. YSBR-031 and the National Key R\&D Program of China under Grant No. 2021YFA1000800. The authors also expressed their sincere appreciation to the referees for their insightful suggestions, which significantly improved the quality of this manuscript.

\end{document}